\documentclass[12pt]{amsart}
\usepackage{amsmath}
\usepackage{amscd}
\usepackage{amssymb}
\usepackage{amsfonts}
\usepackage{amsthm}
\usepackage{bbm}
\usepackage{cancel}
\usepackage{color}
\usepackage{eucal}
\usepackage{enumerate,yfonts}
\usepackage{enumitem}
\usepackage{relsize}
\usepackage[utf8]{inputenc}

\usepackage[english]{babel}

 \allowdisplaybreaks

\usepackage{pdfsync}
 \usepackage[all,cmtip]{xy}
\usepackage{graphicx}
\usepackage{graphics}
\usepackage{hyperref}
\usepackage{latexsym}
\usepackage{mathrsfs}
 \usepackage{placeins}
\usepackage{pstricks}
\usepackage{bookmark}

\usepackage{stmaryrd}
\usepackage{url}

\DeclareFontFamily{U}{mathx}{\hyphenchar\font45}
\DeclareFontShape{U}{mathx}{m}{n}{
      <5> <6> <7> <8> <9> <10>
     <10.95> <12> <14.4> <17.28> <20.74> <24.88>
    mathx10
      }{}
\DeclareSymbolFont{mathx}{U}{mathx}{m}{n}
\DeclareMathAccent{\widecheck}{\mathalpha}{mathx}{"71}

\newtheorem{thm}{Theorem}[section]
\newtheorem{corollary}[thm]{Corollary}
\newtheorem{lemma}[thm]{Lemma}
\newtheorem{lem}[thm]{Lemma}
\newtheorem{proposition}[thm]{Proposition}
\newtheorem{prop}[thm]{Proposition}

\newtheorem{thm-dfn}[thm]{Theorem-Definition}

\newtheorem*{rmk}{Remark}

\newtheorem{definition}[thm]{Definition}
\newtheorem{remark}[thm]{Remark}

\numberwithin{equation}{section}

\newcommand{\ft}{{\mathfrak t}}
\newcommand{\fl}{{\mathfrak l}}

\newcommand{\fu}{{\mathfrak u}}

\newcommand{\fp}{{\mathfrak p}}

\newcommand{\fa}{{\mathfrak a}}

\newcommand{\fc}{{\mathfrak c}}
\newcommand{\fk}{{\mathfrak k}}

\newcommand{\fm}{{\mathfrak{m}}}
\newcommand{\fh}{{\mathfrak{h}}}

\newcommand{\fF}{{\mathfrak{F}}}

\newcommand{\Lp}{{\mathfrak{p}}}

\newcommand{\Lb}{{\mathfrak{b}}}
\newcommand{\Lt}{{\mathfrak{t}}}
\newcommand{\Lg}{{\mathfrak g}}

\newcommand{\Ln}{{\mathfrak{n}}}

\newcommand{\Ll}{{\mathfrak{l}}}

\newcommand{\fs}{{\mathfrak{s}}}

\newcommand{\rR}{{\mathrm R}}
\newcommand{\rIm}{{\mathrm Im}}
\newcommand{\rC}{{\mathrm C}}
\newcommand{\rCx}{{\mathrm Cx}}
\newcommand{\rf}{{\mathrm f}}
\newcommand{\rn}{{\mathrm n}}

\newcommand{\bC}{{\mathbb C}}

\newcommand{\bG}{{\mathbb G}}
\newcommand{\bZ}{{\mathbb Z}}

\newcommand{\bbR}{{\mathbb R}}

\newcommand{\calF}{{\mathcal F}}

\newcommand{\calH}{{\mathcal H}}

\newcommand{\calT}{{\mathcal T}}

\newcommand{\cB}{{\mathcal B}}

\newcommand{\cO}{{\mathcal O}}
\newcommand{\cA}{{\mathcal A}}
\newcommand{\cF}{{\mathcal F}}
\newcommand{\cN}{{\mathcal N}}
\newcommand{\cH}{{\mathcal H}}

\newcommand{\cT}{{\mathcal T}}

\newcommand{\cP}{{\mathcal P}}
\newcommand{\cE}{{\mathcal E}}

\newcommand{\cZ}{{\mathcal Z}}

\newcommand{\cM}{{\mathcal{M}}}

\newcommand{\cL}{{\mathcal{L}}}

\newcommand{\on}{\operatorname}

\newcommand{\is}{\simeq}

\newcommand{\Loc}{\on{LocSys}}

\newcommand{\nc}{\newcommand}

\nc{\al}{{\alpha}} \nc{\be}{{\beta}} \nc{\ga}{{\gamma}}
\nc{\ve}{{\varepsilon}} \nc{\Ga}{{\Gamma}} 
\nc{\La}{{\fa}}

\nc{\ad }{{\on{ad }}}

\nc{\aff}{{\on{aff}}} \nc{\Aff}{{\mathbf{Aff}}}

\nc{\der}{{\on{der}}}

\nc{\diag}{{\on{diag}}}

\nc{\Fl}{{\calF\ell}}

\nc{\Hg}{{\on{Higgs}}}

\nc{\Id}{{\on{Id}}}

\nc{\Ind}{{\on{Ind}}}
\newcommand{\Lie}{{\on{Lie}}}
\nc{\Op}{{\on{Op}}}

\nc{\res}{{\on{res}}}

\nc{\tr}{{\on{tr}}}

\nc{\GSp}{{\on{GSp}}} \nc{\GU}{{\on{GU}}} \nc{\SL}{{\on{SL}}}
\nc{\SU}{{\on{SU}}} \nc{\SO}{{\on{SO}}}

\nc{\nh}{{\Loc_{J^p}(\tau')}}
\nc{\bnh}{{\Loc_{\breve J^p}(\tau')}}

\nc{\bU}{{\overline{U}}} 
\nc{\IC}{{\on{IC}}}

\renewcommand{\SS}{{\operatorname{SS}}}
\nc{\op}{{\operatorname{P}}}

\newcommand{\br}{\begin{rouge}}
\newcommand{\er}{\end{rouge}}

\newcommand{\bb}{\begin{bluet}}
\newcommand{\eb}{\end{bluet}}

\newcommand{\p}{\perp}

\nc{\ot}{\otimes}

\nc{\oh}{{\operatorname{H}}}
\nc{\gr}{{\operatorname{gr}}}
\nc{\rk}{{\operatorname{rank}}}
\nc{\codim}{{\operatorname{codim}}}
\nc{\img}{{\operatorname{Im}}}
\nc{\Span}{{\operatorname{Span}}}
\nc{\Img}{\operatorname{Im}}
\nc{\Char}{\operatorname{Char}}

\newcommand{\rs}{\mathrm{s}}

\newcommand{\beqn}{\begin{equation*}}
\newcommand{\eeqn}{\end{equation*}}

\newcommand{\beq}{\begin{equation}}
\newcommand{\eeq}{\end{equation}}

\newcommand{\bern}{\begin{eqnarray*}}
\newcommand{\eern}{\end{eqnarray*}}

\newcommand{\ber}{\begin{eqnarray}}
\newcommand{\eer}{\end{eqnarray}}

\newcommand{\inv}{{\mathbin{/\mkern-4mu/}}}

\newcommand{\pa}{{\mathbf{p}}}
\newcommand{\nmod}{\hspace{-.1in}\mod}

\newcommand{\dsum}{{\mathlarger{\subset}\mkern-12mu\raisebox{2pt}{\mbox{\larger[-5]$\mathsmaller{\bigoplus}$}}}}

\begin{document}
\title[Character sheaves for classical symmetric pairs]{Character sheaves for classical symmetric pairs}

        \author{Kari Vilonen}
        \address{School of Mathematics and Statistics, University of Melbourne, VIC 3010, Australia, and Department of Mathematics and Statistics, University of Helsinki, Helsinki, 00014, Finland}
         \email{kari.vilonen@unimelb.edu.au}
         \thanks{Kari Vilonen was supported in part by  the ARC grants DP150103525 and DP180101445 and the Academy of Finland.}
         \author{Ting Xue}
         \address{ School of Mathematics and Statistics, University of Melbourne, VIC 3010, Australia, and Department of Mathematics and Statistics, University of Helsinki, Helsinki, 00014, Finland}
         \email{ting.xue@unimelb.edu.au}
\thanks{Ting Xue was supported in part by the ARC grants DP150103525 and DE160100975.}

\makeatletter
\let\@wraptoccontribs\wraptoccontribs
\makeatother

\contrib[with an appendix by]{Dennis Stanton}
\begin{abstract}
We establish a Springer theory for classical symmetric pairs. We give an explicit description of character sheaves in this setting. In particular we determine the cuspidal character sheaves. 
\end{abstract}

\maketitle

\setcounter{tocdepth}{2} \tableofcontents

\section{Introduction}

In this paper we work out a theory analogous to the generalized Springer correspondence of~\cite{Lu} in the context of symmetric pairs. We concentrate on classical symmetric pairs, but our methods are general and do extend to the other cases with some minor modifications. We have chosen to  treat the case of $SL_n$ in a companion paper~\cite{VX1} as it requires an expansion of the techniques used here.

Let $G$ be a connected complex reductive algebraic group and $\theta:G\to G$ an involution. Let $K=G^\theta$ be the subgroup of fixed points of $\theta$. The pair $(G,K)$ is called a symmetric pair. It is called a {\it split} symmetric pair if there exists a maximal torus $T$ of $G$ that is $\theta$-split, i.e., $\theta(t)=t^{-1}$ for all $t\in T$. Let $\Lg=\on{Lie}G$ and let $\Lg=\Lg_0\oplus\Lg_1$ be the decomposition into $\theta$-eigenspaces so that $d\theta|_{\Lg_i}=(-1)^i$. Let $\cN$ denote the nilpotent cone of $\Lg$ and let $\cN_1=\cN\cap\Lg_1$ be the nilpotent cone in $\Lg_1$.  We write $\Char_K(\Lg_1)$ for the set of irreducible $\bC^*$-conic $K$-equivariant perverse sheaves $\cF$ on $\Lg_1$ whose singular support is nilpotent, i.e., for $\cF$ such that $\SS(\cF)\subset \Lg_1\times \cN_1$ where   $\Lg_1$ and $\Lg_1^*$ are identified via a $K$-invariant non-degenerate bilinear form on $\Lg_1$. We call the sheaves in $\Char_K(\Lg_1)$ character sheaves for the symmetric pair $(G,K)$. 

The classical character sheaves of~\cite{Lu} on a Lie algebra $\Lg$ can be viewed as a special case of character sheaves for symmetric pairs, where one considers the symmetric pair $(G\times G,G)$ with $\theta$ switching the factors in $G\times G$. In this paper we concentrate on the classical symmetric pairs, that is, when $G$ is $GL_n,Sp_{2n}$ or $SO_{n}$, and give a complete description of the set $\Char_K(\Lg_1)$ for all pairs $(G,K)$.   Special cases have been considered before. The case $(GL_{2n},Sp_{2n})$ was considered by Grinberg in~\cite{G3}, Henderson in~\cite{H}, and Lusztig in \cite{L}.  The case $(GL_n,GL_p\times GL_q)$  was considered by Lusztig in~\cite{L} where he treats $GL_n$ in the case of arbitrary finite order semi-simple inner automorphisms. In these instances the Springer theory closely resembles the classical situation. In~\cite{CVX} we consider the case $(SL_n,SO_n)$ where phenomena quite different from the classical case already occur. 

We write  $\Char_K^\rf(\Lg_1)$ for character sheaves whose support is all of $\Lg_1$; we call these full support character sheaves. We show that all full support character sheaves arise from the nearby cycle construction in~\cite{GVX}, which, in turn, is based on ideas in~\cite{G1,G2}. Sheaves in $\Char_K^\rf(\Lg_1)$ are $\on{IC}$-sheaves of certain $K$-equivariant local systems on $\Lg_1^{rs}$, the regular semisimple locus of $\Lg_1$.  The equivariant fundamental group $\pi_1^K(\Lg_1^{rs})$ is an extended braid group, see~\eqref{diagram-fundamental group}. In~\S\ref{csfs} we construct certain Hecke algebras, following~\cite{GVX}, and show how their simple modules give us full support character sheaves, see Proposition~\ref{thm-nearby cycles} and Corollary~\ref{coro-full}.

At the other extreme from full support character sheaves there are  nilpotent support character sheaves, i.e., those supported on the closure of a nilpotent $K$-orbit in $\cN_1$. We write $\Char_K^\rn(\Lg_1)$ for these character sheaves. We show such sheaves only occur when the involution $\theta$ is inner and that they all arise via parabolic induction from  $\theta$-stable Borel subgroups. In particular, the support of these character sheaves are exactly the Richardson orbits attached to $\theta$-stable Borel subgroups. From a geometric point of view these are the nilpotent $K$-orbits whose closures are images of conormal bundles of closed $K$-orbits on the flag manifold $G/B$ under the moment map. They can be viewed as singular supports of discrete series representations. In this context they have been studied and classified by Trapa in~\cite{T}. We make use of his classification in our determination of $\Char_K^\rn(\Lg_1)$ in Section~\ref{sec-biorbital}. 

Following standard terminology,  we call a character sheaf {\it cuspidal} if it does not appear as a direct summand (up to shift) in parabolic induction of character sheaves in $\on{Char}_{L^\theta}(\Ll_1)$ from $\theta$-stable Levi subgroups $L$ contained in proper $\theta$-stable parabolic subgroups of $G$. In~\cite{Lu}  Lusztig has worked out the cuspidal character sheaves in the classical case, which have nilpotent supports. One of our main results classifies cuspidal character sheaves for classical symmetric pairs as follows.

\begin{thm} (Corollary~\ref{cuspidal corollary})
\label{def of cuspidal}
 The cuspidal character sheaves consist of 
\begin{enumerate}
\item  The constant sheaf $\bC_{\Lg_1}$ for the pair $(GL_2,GL_1\times GL_1)$,
\item Full support character sheaves for split symmetric pairs.
\end{enumerate}
\end{thm}
We note that the skyscraper sheaf at the origin for the pair $(GL_1,GL_1)$ is also cuspidal.  The full support cuspidal sheaves in (2) are obtained via the nearby cycle construction from simple modules of Hecke algebras of~\cite{GVX} as explained in~\S\ref{csfs} and in Proposition~\ref{thm-nearby cycles}. 
These Hecke algebras are not in general associated to Weyl groups of subgroups of $G$. However, as we will explain in~\S\ref{sec-endoscopy}, they are associated to Weyl groups of subgroups of the dual group $\check G$ and so one can think of them as arising from endoscopic groups. 

\begin{rmk}
For other isogeny classes such as special linear groups and spin groups, there are further cuspidal character sheaves of full support, as well as cuspidal character sheaves  which are not of full support at various central characters. In particular, we will have cuspidal character sheaves for special linear groups on quasi-split symmetric pairs. We treat this case in~\cite{VX1}. For spin groups there are cuspidals for symmetric pairs which are not quasi-split, see~\cite{X}. 
\end{rmk}

We determine the character sheaves  in Theorems~\ref{thm-type A}, \ref{thm-type C}, \ref{thm-type BD}. For each nilpotent $K$-orbit $\cO$ on $\cN_1$ we construct in~\S\ref{b} the corresponding dual stratum $\widecheck\cO\subset \Lg_1$. The character sheaves are locally constant along the strata $\widecheck\cO$. For the $0$-orbit the dual stratum is $\Lg_1^{rs}$.  In~\S\ref{sec-support} we determine explicitly the nilpotent orbits $\cO$ such that $\widecheck\cO$ supports a character sheaf and we explicitly describe, analogously to the $\Lg_1^{rs}$ case, the equivariant fundamental groups $\pi_1^K(\widecheck\cO)$ of these $\widecheck\cO$ in terms of extended braid groups, see~\eqref{identification}. We  write down explicitly the local systems on the $\widecheck\cO$ in terms of  certain Hecke algebras whose $\on{IC}$-sheaves are character sheaves. To prove the theorems we make a detailed study of parabolic induction of character sheaves from $\theta$-stable Levi subgroups contained in proper $\theta$-stable parabolic subgroups. To show that the nearby cycle construction and parabolic induction give a complete set of the character sheaves we rely on a counting argument, as in the classical case treated in~\cite{Lu}.

We work with varieties over complex numbers and  with sheaves with complex coefficients. However, we can use any field of characteristic zero as coefficients, although sometimes we require the field to contain roots of unity. In~\cite{GVX} we work over the integers. Thus, it seems reasonable that one can develop a modular theory along the lines of this paper. Furthermore, our results should hold for $\ell$-adic sheaves in the finite characteristic setting, but at the moment~\cite{GVX} is written in the classical topology only. 

 In \cite{LY}, the authors have studied the decomposition of the derived category $D_K(\cN_1)$ into blocks using spiral induction  where they treat the general case of finite order semi-simple automorphisms. Their work can be viewed as another generalization of the Springer theory. We expect that the theory presented in this paper can be extended to the case of finite order automorphism $\theta$, see partial results in~\cite{GVX2,VX2}, although we do not expect to obtain as detailed results as those presented in this paper. Graded Lie algebras are intimately connected to $p$-adic groups via the Moy-Prasad filtration (see, for example, \cite{RY}). Our theory is also directly related to the affine Springer fibers studied by Oblomkov and Yun ~\cite{OY}. Thus, we expect that our theory presented here, and its generalization to the higher order cases, to have applications in $p$-adic groups as well as geometry of  affine Springer fibers. We will connect our results to the work of Lusztig and Yun, and address these applications in future work.

The paper is organized as follows. In Section~\ref{sec-preliminaries} we recall the preliminaries on symmetric pairs, nilpotent orbits, restricted roots,  little Weyl groups, and set up the notation. In particular, we give an explicit description of the classical symmetric pairs we work with and explicitly describe the associated Lie theoretic data. In Section~\ref{sec-general strat} we describe the general strategy to determine the set $\Char_K(\Lg_1)$. 
In Section~\ref{sec-biorbital} we describe the nilpotent support character sheaves. In Section~\ref{sec-full support} we apply the nearby cycle construction in~\cite{GVX} and describe the full support character sheaves. In Section~\ref{sec-main theorems} we state the main theorems, Theorems~\ref{thm-type A}, \ref{thm-type C}, and~\ref{thm-type BD}, where the character sheaves are determined. In Section~\ref{sec-proof} we prove the main theorems  by combining the results in the previous sections, parabolic induction and counting arguments. In Appendix~\ref{dual strata} we discuss the dual strata in the classical situations. Appendix~\ref{combinatorics} by Dennis Stanton contains proofs of combinatorial formulas which are crucial for the proofs in this paper.

{\bf Acknowledgement.} We thank Jeff Adams, Volker Genz, George Lusztig, Monty McGovern,  Peter Trapa,  David Vogan and Roger Zierau for helpful discussions. We also thank the Research Institute for Mathematical Sciences at Kyoto University for hospitality and a good working environment. Special thanks are due to Misha Grinberg, Dennis Stanton, Cheng-Chiang Tsai, and Zhiwei Yun for invaluable discussions. Furthermore, Dennis Stanton has kindly supplied us with an appendix which is a crucial ingredient in our proofs.

\section{Preliminaries}\label{sec-preliminaries}

Throughout the paper we work with algebraic groups over $\bC$ and with sheaves with complex coefficients.  For perverse sheaves we use the conventions of~\cite{BBD}. If $\cF$ is a perverse sheaf up to a shift we often write $\cF[-]$ for the corresponding perverse sheaf. 

\subsection{Character sheaves}
Let $G$ be a connected reductive algebraic group over $\bC$ and $\theta:G\to G$ an involution. Let $K=G^\theta$. A torus $A\subset G$ is called $\theta$-split, if $\theta(t)=t^{-1}$ for all $t\in A$. The symmetric pair $(G,K)$ is called split if there exists a maximal torus $A$ of $G$ which is $\theta$-split. Note that the involution $\theta$ is inner if and only if $\on{rank}K=\on{rank}G$.

Let $\Lg$ be the Lie algebra of $G$. The involution $\theta$ gives rise to a decomposition $\Lg=\Lg_0\oplus\Lg_1$ such that $d\theta|_{\Lg_i}=(-1)^i$. Let $\cN$ denote the nilpotent cone of $\Lg$ and let $\cN_1=\cN\cap\Lg_1$. The group $K$ acts on $\cN_1$ with finitely many orbits \cite{KR}. We identify $\Lg_1$ and $\Lg_1^*$ via a $K$-invariant non-degenerate bilinear form on $\Lg_1$, which can be obtained from a $G$-invariant and $\theta$-invariant non-degenerate bilinear form on $\Lg$.

Let $\cA_K(\Lg_1)$ denote the set of irreducible $K$-equivariant perverse sheaves on $\cN_1$. 
That is,
\beqn
\begin{gathered}
\cA_K(\Lg_1)=\{\on{IC}(\cO,\cE)\,|\, \cO\subset\cN_1\text{ is an $K$-orbit and $\cE$ is an irreducible $K$-equivariant}\\
\ \ \text{ local system on $\cO$ (up to isomorphism)}\}. 
\end{gathered}
\eeqn
Recall the set $\Char_K(\Lg_1)$ of character sheaves, that is, the set of  irreducible $\bC^*$-conic $K$-equivariant  perverse sheaves on $\Lg_1$ whose singular support is nilpotent. 
Consider the Fourier transform 
$
\fF:P_K(\Lg_1)\to P_K(\Lg_1).
$
 By definition the functor $\fF$ gives us a bijection
\beq\label{twosides}
\fF:\cA_K(\Lg_1) \ \xrightarrow{\sim} \ \Char_K(\Lg_1)\,.
\eeq
This implies, in particular, that the set $\Char_K(\Lg_1)$ is finite. Note that Lusztig calls the sheaves in $\cA_K(\Lg_1)$ orbital complexes and the sheaves in $\Char_K(\Lg_1)$ anti-orbital complexes. 

There are two important extreme cases of character sheaves. The character sheaves supported on all of $\Lg_1$ are called full support character sheaves,  and we write $\Char_K^\rf(\Lg_1)$ for them. The character sheaves supported on nilpotent $K$-orbits, are called nilpotent support character sheaves and we write $\Char_K^\rn(\Lg_1)$ for them.

\subsection{Restricted roots and the little Weyl group}Let $(G,K)$ be a symmetric pair. We fix a Cartan subspace $\fa$ of $\Lg_1$, i.e., a maximal abelian subspace consisting of semisimple elements. We write $A$ for the maximal $\theta$-split torus of $G$ with Lie algebra $\fa$.  Let $T\subset Z_G(\fa)$ be a maximal torus. Then $T\supset A$ and $T$ is $\theta$-stable. Let $C=T^\theta$ and $\fc=\on{\Lie}C\subset \Lg_0$. We have $\on{Lie}T=\Lt = \fa \oplus \fc$. 

Let $\Phi\subset X^*(T)=\on{Hom}(T,\bG_m)$ be the root system of $(\Lg,T)$.  For each $\alpha\in\Phi$, let $\check\alpha\in X_*(T)=\on{Hom}(\bG_m,T)$ denote the corresponding coroot. We recall the following notions
\begin{eqnarray*}
&&\alpha \in \Phi \ \ \text{is real}  \text{ if } \alpha |_{\fc} = 0 \iff \theta \alpha = -\alpha,\  \
\alpha \in \Phi \ \ \text{is imaginary}  \text{ if } \alpha |_{\La} = 0 \iff \theta \alpha = \alpha, \\
&&\alpha \in \Phi \ \ \text{is complex}  \text{ otherwise}.
\end{eqnarray*}
Let $\Phi^{\mathrm{R}}\subset\Phi$ (resp. $\Phi^{\rIm}\subset\Phi$, $\Phi^{\rCx}\subset\Phi$) be the set of real (resp. imaginary,  complex) roots. 
Then $\Phi^{\rR}$ and $\Phi^{\rIm}$ are subroot systems of $\Phi$.
Furthermore, let $\rho^\rR=\frac{1}{2}\sum_{\alpha\in\Phi^{\rR,+}}\alpha$, $\rho^{\rIm}=\frac{1}{2}\sum_{\alpha\in\Phi^{\rIm,+}}\alpha$ (with respect to some positive systems of $\Phi^\rR$ and $\Phi^{\rIm}$ respectively) and let $$\Phi^{\mathrm{C}}=\{\alpha\in\Phi^{\rCx}\,|\,(\alpha,\rho^\rR)=(\alpha,\rho^{\rIm})=0\}.$$
 Then $\Phi^{\rC}$ is also a subroot system of $\Phi$.
 
 Let $W_\fa:=N_K(\fa)/Z_K(\fa)$ be the little Weyl group of the pair $(G,K)$. We have (see~\cite[Proposition 4.16]{V})
 \beqn
 W_\fa\cong W^{\rR}\rtimes (W^\rC)^\theta
 \eeqn
 where $W^{\rC}$ (resp. $W^{\rR}$) is the Weyl group of the root system $\Phi^{\rC}$ (resp. $\Phi^{\rR}$) and $(W^\rC)^\theta\subset W^\rC$ is the subgroup that consists of the elements in $W^\rC$ commuting with $\theta$. It is well-known that $W_\fa$ is also the Weyl group of the restricted root system, denoted by $\Sigma$, i.e., the (not necessarily reduced) root system formed  by the restrictions of $\alpha\in\Phi$ to $\fa$. We have
 \beqn
 \Lg=Z_\Lg(\fa)\oplus\bigoplus_{\bar\alpha\in\Sigma}\Lg_{\bar\alpha},
 \eeqn
 where  $\Lg_{\bar\alpha}$ denotes the root space corresponding to the restricted root $\bar\alpha$. Note that  if $\bar\alpha$ is the restriction of a complex root $\alpha$, then the root space $\Lg_{\bar{\alpha}}$ is at least two dimensional, since $\alpha$ and $-\theta\alpha$ restrict to the same restricted root $\bar\alpha=(\alpha-\theta\alpha)/2$.
 
Let $s\in W_\fa$ be a reflection, i.e., $s=s_{\bar\alpha}$ for some $\bar\alpha\in\Sigma$. We define
\beq\label{definition of delta(s)}
\delta(s):=\frac{1}{2}\sum_{\substack{\bar\alpha\in\Sigma\\ s_{\bar\alpha}=s}}\dim\Lg_{\bar\alpha}.
\eeq
If $\delta(s)=1$, we write $\alpha_s$ for the unique positive real root in $\Phi^{\rR}$ such that $s_{\bar{\alpha}_s}=s$.

\subsection{The equivariant fundamental group}\label{ssec-equiv fdgp} We say that an element $x\in \Lg_1$ is regular in $\Lg_1$ if $\on{dim}Z_K(x)\leq\on{dim}Z_K(y)$ for all $y\in\Lg_1$. Let $\Lg_1^{rs}$ denote the set of regular semisimple elements of $\Lg_1$ and
set $\fa^{rs}=\fa\cap\Lg_1^{rs}$.  We write $\pi_1^K(\Lg_1^{rs},a)$ for the equivariant fundamental group with base point $a\in \fa^{rs}$.

Consider the adjoint quotient map
\beqn
f:\Lg_1\to\Lg_1\inv K\cong\La\slash W_\La.
\eeqn
As in~\cite{GVX} it gives rise to the following commutative diagram with exact rows
\beq\label{diagram-fundamental group}
\begin{CD}
1 @>>> I@>>>\widetilde B_{W_\La} :=\pi_1^K(\Lg_1^{rs},a)@>{\tilde{q}}>>B_{W_\La}@>>> 1
\\
@. @| @VV{\tilde p}V @VV{p}V @.
\\
1 @>>> I@>>>\widetilde W_\La:=N_K(\La)/Z_K(\La)^0@>{q}>> W_\La@>>> 1\,,
\end{CD}
\eeq
where
\beqn
I=Z_K(\La)/Z_K(\La)^0
\eeqn
is a 2-group and $B_{W_\fa}=\pi_1(\fa^{rs}/W_\fa,\bar a)$ is the braid group associated to $W_\fa$, $\bar a=f(a)$.

\indent In~\cite{GVX}, we have constructed an explicit splitting of the top exact sequence and thus we write
\beq
\label{semidirect}
 \widetilde B_{W_\La} \is I\rtimes B_{W_\La},
\eeq
and we note that the braid group $B_{W_\La}$ acts on $I$ through the quotient $p:B_{W_\La}\to W_\La$. Such a splitting is not unique. The example in~\S\ref{example} illustrates how the splitting affects the labelling of character sheaves.

\subsection{The classical symmetric pairs}\label{ssec-classical pairs} In this subsection we summarize the detailed structure of classical symmetric pairs we will use to classify character sheaves. The classical symmetric pairs are (see for example~\cite{He})
\begin{eqnarray*}
&&\text{(A\rm{I})}\ (GL_n,O_n))\quad
\text{(A\rm{II})}\ \ (GL_{2n},Sp_{2n})\ \ \text{(A\rm{III})}\ (GL_n,GL_p\times GL_q)\quad\,p+q=n
\\
&&\text{(BD\rm{I})}\ (SO_N,S(O_p\times O_q))\quad\,p+q=N\quad \text{(D\rm{III})}\ (SO_N,GL_{N/2})\\
&&\text{(C\rm{I})}\ (Sp_{2n},GL_n)\quad
\text{(C\rm{II})}\ (Sp_{2n},Sp_{2p}\times Sp_{2q})\quad p+q=n.
\end{eqnarray*}
Type (A\rm{I}) has been dealt with in~\cite{CVX} and type (A\rm{II}) has been studied in~\cite{G3},~\cite{H}, and~\cite{L}. Thus we will focus on the remaining cases. We will make use of the following concrete descriptions of $(G,K)$. 

In type A, let  $V=V^+\oplus V^-$ be a $\bC$-vector space of dimension $n$ and let $G=GL_V$. In type BD (resp. C), let $V=V^+\oplus V^-$ be a $\bC$-vector space of dimension $N$ (resp. $2n$) equipped with a non-degenerate symmetric bilinear (resp. symplectic) form $(\, ,\, )$ and let $G=SO_{V,(\, ,\, )}$ (resp. $G=Sp_{V,(\, ,\, )}$).

 {(A\rm{III})}\ \ $\on{dim}V^+=p$, $\on{dim}V^-=q$, $K=GL_{V^+}\times GL_{V^-}$. 

 {(BD\rm{I}) (resp. (C\rm{II}))}\ \  $(V^+,V^-)=0$, $\on{dim}V^+=p$ (resp. $2p$), $\on{dim}V^-=q$ (resp. $2q$), $K=S(O_{V^+}\times O_{V^-})$ (resp.  $K=Sp_{V^+}\times Sp_{V^-}$).
 
 {(D\rm{III}) (C\rm{I})}\ \   $(\ ,\ )|_{V^+}=(\ ,\ )|_{V^-}=0$, $\on{dim}V^+=\on{dim}V^-=n$, $K=G\cap(GL_{V^+}\times GL_{V^-})\cong GL_n$.

We write $r=\on{dim}\fa$. Then $r=\min(p,q)$ in type AIII, BDI, CII, $r=n$ in type CI, and $r=[n/2]$ in type DIII. Note that the split pairs are types A\rm{I}, C\rm{I}, and BD\rm{I} with $r=[N/2]$.

We describe the various root systems and the little Weyl groups, for the convenience of the reader.  As usual, we write $\Phi=\{\pm(\epsilon_i-\epsilon_j)\,|\,1\leq i<j\leq n\}$ in type $A_{n-1}$, $\Phi=\{\pm(\epsilon_i\pm\epsilon_j),\ 1\leq i<j\leq n;\ \pm\epsilon_i\text{ (resp. $\pm2\epsilon_i$)},\,1\leq i\leq n\}$ in type $B_n$ (resp. $C_n$), and $\Phi=\{\pm(\epsilon_i\pm\epsilon_j)\,|\,1\leq i<j\leq n\}$ in type $D_n$.  Let
\beqn
W_n\text{ denote the Weyl group of type $B_n$ (or $C_n$)}  \text{ and $W_n'$ denote the Weyl group of type $D_n$}.
\eeqn
We use the convention that $W_0=W_0'=\{1\}$.  In what follows we indicate the dimension of the restricted root spaces $\Lg_{\bar\alpha}$ by the  number within the parentheses after each restricted root $\bar\alpha$ in $\Sigma$. 
\bern
(\rm{AIII})&&\Phi^{\rR}=\{\pm(\epsilon_{2i-1}-\epsilon_{2i}),\, 1\leq i\leq r\}, \\&& \Phi^{\rC}=\{\pm(\epsilon_{2i-1}-\epsilon_{2j-1}),\, \pm(\epsilon_{2i}-\epsilon_{2j}),\, 1\leq i<j\leq r\},\\&& \Sigma=\{\pm(\epsilon_i\pm\epsilon_j)\,(2),\, 1\leq i<j\leq r,\, \pm2\epsilon_i\,(1), \,\pm\epsilon_i\, (2n-4r),\,1\leq i\leq r\}
\\
(\rm{BI})
&&\Phi^{\rR}=\{\pm(\epsilon_i\pm\epsilon_j),\,1\leq i<j\leq r,\ \pm\epsilon_i,\,1\leq i\leq r\},\ 
\Phi^{\rC}=\emptyset,\\
&& \Sigma=\{\pm(\epsilon_i\pm\epsilon_j)\,(1),\,1\leq i<j\leq r,\ \pm\epsilon_i\,(|p-q|),\,1\leq i\leq r\}
\\
(\rm{DI})&&
\Phi^{\rR}=\{\pm(\epsilon_i\pm\epsilon_j),\,1\leq i<j\leq r\},\ \Phi^{\rC}=\{\pm(\epsilon_r\pm\epsilon_n)\}\ (r<n),\ \Phi^{\rC}=\emptyset\text{ $(r=n)$},\\
&&\Sigma=\{\pm(\epsilon_i\pm\epsilon_j)\,(1),1\leq i<j\leq r,\ \pm\epsilon_i\,(|p-q|),1\leq i\leq r\}
\\
(\rm{CI})&&\Phi^{\rR}=\Sigma=\{\pm(\epsilon_i\pm\epsilon_j),\ 1\leq i<j\leq r;\ \pm2\epsilon_i,\,1\leq i\leq r\},\,
\\
(\rm{CII})
&&\Phi^{\rR}=\{\pm(\epsilon_{2i-1}+\epsilon_{2i}),\,1\leq i\leq r\},\\&& \Phi^{\rC}=\{\pm(\epsilon_{2i-1}-\epsilon_{2j-1}),\ \pm(\epsilon_{2i}-\epsilon_{2j}),\,1\leq i<j\leq r\},\\
&&  \Sigma=\{\text{$\pm(\epsilon_i\pm\epsilon_j)\,(4)$, $1\leq i<j\leq r$, $\pm2\epsilon_i\,(3)$, $\pm\epsilon_i\,(4n-8r)$, $1\leq i\leq r$}\}.
\\
(\rm{DIII})&&
\Phi^{\rR}=\{\pm(\epsilon_{2i-1}-\epsilon_{2i}),\,1\leq i\leq r\},\\&& \Phi^{\rC}=\{\pm(\epsilon_{2i-1}-\epsilon_{2j-1}),\,\pm(\epsilon_{2i}-\epsilon_{2j}),\,1\leq i<j\leq r\},\\
&&\Sigma=\{\text{$\pm(\epsilon_i\pm\epsilon_j)\,(4)$, $1\leq i<j\leq [r]$, $\pm2\epsilon_i\,(1)$,  $\pm\epsilon_i\,(4)$ ($n$ odd), $1\leq i\leq r$}\}.
\eern
We have $W_\fa\cong W_r$ except in type DI when $r=n$. In the latter case $W_\fa\cong W_r'$.

\subsection{Nilpotent orbits and component groups of the centralizers}
\label{nilp comp} In this subsection we recall the classification of nilpotent $K$-orbits on $\cN_1$ and the description of the components groups $A_K(x)=Z_K(x)/Z_K(x)^0$, $x\in\cN_1$. Using the Sekiguchi-Kostant correspondence, this is reduced to the same question for real nilpotent orbits, see, for example, \cite{CM} and \cite{SS}. For completeness and to fix notations, we recall the results here.

 A signed Young diagram is a Young diagram with each box labeled $+$ or $-$ so that signs alternate across rows. Two signed Young diagrams  are regarded equivalent if and only if one can be obtained from the other by interchanging rows of equal length. A signed Young diagram is said to have signature $(p,q)$ if there are $p$ boxes labeled $+$ and $q$ boxes labeled $-$.

From now on, we write a signed Young diagram as
\begin{subequations} 
\beq\label{signed Young diagram}
\lambda=(\lambda_1)^{p_1}_+(\lambda_1)^{q_1}_-(\lambda_2)^{p_2}_+(\lambda_2)^{q_2}_-\cdots(\lambda_s)^{p_s}_+(\lambda_s)^{q_s}_-,
\eeq
where $\lambda=(\lambda_1)^{p_1+q_1}(\lambda_2)^{p_2+q_2}\cdots(\lambda_s)^{p_s+q_s}$ is the corresponding partition, $\lambda_1>\lambda_2>\cdots>\lambda_s>0$, for $i=1,\ldots, s$, $p_i+q_i>0$ is the multiplicity of $\lambda_i$ in $\lambda$, and  $p_i\geq 0$ (resp. $q_i\geq 0$) is the number of rows of length $\lambda_i$ that begins with sign $+$ (resp. $-$). For later use, it is more convenient to use numbers rather than signs. Thus   we will  sometimes replace the subscript $+$ by $0$ and $-$ by $1$ and write the signed Young diagram in~\eqref{signed Young diagram} as
\beq\label{signed Young diagram-2}
\lambda=(\lambda_1)^{p_1}_0(\lambda_1)^{q_1}_1(\lambda_2)^{p_2}_0(\lambda_2)^{q_2}_1\cdots(\lambda_s)^{p_s}_0(\lambda_s)^{q_s}_1.
\eeq 
\end{subequations}
 Given a signed Young diagram $\lambda$, we write $\cO_\lambda$ for the corresponding nilpotent orbit in $\cN_1$.

{(\rm{AIII})}\ \ The orbits are parametrized by signed Young diagrams with signature $(p,q)$. We have $A_K(x)=1$ for all $x\in\cN_1$.

 {(\rm{BDI})}\ \ The orbits are parametrized by signed Young diagrams of signature $(p,q)$ such that, in terms of~\eqref{signed Young diagram}, $p_i=q_i$ when $\lambda_i$ is even, except that each signed Young diagram with all $\lambda_i$ even corresponds to two orbits. The latter case can only happen when $(G,K)=(SO_{4m},S(O_{2m}\times O_{2m}))$; we write $\cO_{\lambda}^{\rm{I}}$ and $\cO_{\lambda}^{\rm{II}}$ for those two orbits.  We have 
\beq\label{component group-BD}
\begin{gathered}
A_K(x)=(\bZ/2\bZ)^{r_\lambda},\\
r_\lambda=|\{i\in[1,s]\,|\,\lambda_i\text{ odd  and }p_i>0\}|+|\{i\in[1,s]\,|\,\lambda_i\text{ odd  and }q_i>0\}|-1,\\
\text{ if at least one part of $\lambda$ is odd};\ \ r_\lambda=0\text{ otherwise}.
\end{gathered}
\eeq
\indent{(\rm{DIII})}\ \ The orbits are parametrized by signed Young diagrams of signature $(n,n)$ such that, 
in terms of~\eqref{signed Young diagram},  $p_i=q_i$ if $\lambda_i$ is odd, and both $p_i$ and $q_i$ are even if $\lambda_i$ is even. We have that $A_K(x)=1,\text{ for all }x\in\cN_1.$

 {(\rm{CI})}\ \ The orbits are parametrized by signed Young diagrams of  signature $(n,n)$ such that, in terms of~\eqref{signed Young diagram}, $p_i=q_i$ if $\lambda_i$ is odd. We have 
\beq\label{component group-C}
\begin{gathered}
A_K(x)=(\bZ/2\bZ)^{r_\lambda},
\\r_\lambda=|\{i\in[1,s]\,|\,\lambda_i\text{ even  and }p_i>0\}|+|\{i\in[1,s]\,|\,\lambda_i\text{ even  and }q_i>0\}|.
\end{gathered}
\eeq
\indent{(\rm{CII})}\ \ The orbits are parametrized by signed Young diagrams of signature $(2p,2q)$ such that, in terms of~\eqref{signed Young diagram}, $p_i=q_i$ if $\lambda_i$ is even, and  both $p_i$ and $q_i$ are even if $\lambda_i$ is odd. We have that $A_K(x)=1,\text{ for all }x\in\cN_1.$

\subsection{Weyl groups and Hecke algebras}\label{sec-Hecke}
In our setting Hecke algebras with (unequal) parameters $\pm 1$ arise. In this subsection we recall some results about their simple modules and the generating functions of the numbers of  simple modules. 

 Let $\calH_{W_n,c_0,c_1}$ denote the Hecke algebra generated by $T_{s_i}$, $i\in[1,n]$, with relations
\beqn
\begin{gathered}
 T_{s_i}T_{s_{i+1}}T_{s_i}=T_{s_{i+1}}T_{s_i}T_{s_{i+1}},\ i\in[1,n-2],\ 
T_{s_{n-1}}T_{s_n}T_{s_{n-1}}T_{s_n}= T_{s_{n}}T_{s_{n-1}}T_{s_{n}}T_{s_{n-1}},\\
T_{s_i}T_{s_j}=T_{s_j}T_{s_i},\ |i-j|>1;\ (T_{s_i}-c_0)(T_{s_i}+1)=0,\ i\in[1,n-1],\ (T_{s_n}-c_1)(T_{s_n}+1)=0.
\end{gathered}
\eeqn
When $c_0=1$ and $c_1=-1$, the group algebra  $\bC[S_n]$ is naturally a subalgebra of $\calH_{W_{n},1,-1}$. By~\cite[\S5.4]{DJ}, there is a natural bijection between the set of simple modules of $\bC[S_n]$ and the set of simple modules of $\calH_{W_{n},1,-1}$ as follows: each simple module of $\bC[S_n]$ naturally extends to a simple module of $\calH_{W_{n},1,-1}$ by letting $T_{s_n}$ act by $-1$.  Thus, the simple  modules of $\calH_{W_{n},1,-1}$ are parametrized by $\cP(n)$, the set of partitions of $n$. We write
\beq\label{heck rep-sn like}
\on{Irr}\cH_{W_{n},1,-1}=\{L_\beta\,|\,\beta\in\cP(n)\}.
\eeq
When $c_0=c_1=1$, the set of simple $\calH_{W_{n},1,1}=\bC[W_n]$-modules is parametrized by $\cP_2(n)$, the set of bi-partitions of $n$, i.e., pairs of partitions $(\mu,\nu)$ with $|\mu|+|\nu|=n$.
We write
\beqn
\on{Irr}\cH_{W_{n},1,1}=\{L_\rho\,|\,\rho\in\cP_2(n)\}.
\eeqn
Thus, the numbers of simple modules of $\cH_{W_{n},1,-1}$ and $\cH_{W_{n},1,1}$ are given by the following generating functions
\beqn
\sum_{n\geq 0}|\on{Irr}\cH_{W_{n},1,-1}|x^n=\prod_{s\geq 1} \frac 1 {(1-x^s)},\ 
\ 
\sum_{n\geq 0}|\on{Irr}\cH_{W_{n},1,1}|x^n=\prod_{s\geq 1} \frac 1 {(1-x^s)^2}.
\eeqn
For $\cH_{W_k,-1,-1}$ and  $\cH_{W_k,-1,1}$ the combinatorial description of the simple modules is more involved. However, by~\cite{AM}, we have the following generating functions
\beq\label{generating for Hecke}
\begin{gathered}
\sum_k|\on{Irr}\cH_{W_k,-1,-1}|x^k=\prod_{s\geq 1}(1+x^{2s})(1+x^s),\\ \sum_k|\on{Irr}\cH_{W_k,-1,1}|x^k=\prod_{s\geq 1}(1+x^{2s-1})(1+x^s).
\end{gathered}
\eeq
Let $\calH_{W_n',-1}$ denote the Hecke algebra of $W_n'$ with parameter $-1$.
  According to \cite{G}, 
  \beq\label{Hecke-D}
  |\on{Irr}\calH_{W_n',-1}|=\frac{1}{2}|\on{Irr}\cH_{W_n,-1,1}|,\, n\geq1.
  \eeq

\section{General strategy}\label{sec-general strat}
In this section we describe our general strategy to determine character sheaves for symmetric pairs. We will carry out this strategy for the classical symmetric pairs in the subsequent sections. 
We refer the readers to this section for notational conventions.

\subsection{Dual strata for symmetric spaces and supports of character sheaves}
\label{b}

In this subsection we extend the discussion in Appendix~\ref{dual strata} to symmetric pairs. 

For each nilpotent $K$-orbit $\cO$ in $\cN_1$ we consider its conormal bundle
\beqn
\Lambda_\cO = T^*_\cO\Lg_1 =\{(x,y)\in \Lg_1\times\Lg_1\mid x\in\cO \ \ [x,y] = 0\}\,.
\eeqn
Consider the projection $\widetilde\cO $ of $\Lambda_\cO$ to the second coordinate:
\beqn
\widetilde\cO = \{y\in \Lg_1\mid \text{there exist an}\ \  x\in\cO\ \  \text{with} \ \ [x,y] = 0\}\,.
\eeqn
We construct an open (dense) subset $\widecheck \cO$ of $\widetilde\cO$ such that the projection $\Lambda_\cO \to \widetilde\cO$ has constant maximum rank above $\widecheck \cO$.  Since the Fourier transform preserves the singular support, the $\widecheck \cO$ are subvarieties of $\Lg_1$ with the following property:
\begin{equation*}
\text{For any $\cF\in \op_K(\cN_1)$ the Fourier transform
$\fF(\cF)$ is smooth along all the  $\widecheck \cO$}\,.
\end{equation*}
 Moreover, for each $\on{IC}(\cO,\cE)\in\op_K(\cN_1)$,
\beq
\label{nilpsupport}
\on{Supp}\fF(\on{IC}(\cO,\cE))=\overline{\widecheck{\cO'}},\text{ for some }\cO'\subset\bar\cO.
\eeq

As in Appendix~\ref{dual strata} we consider the adjoint quotient map $\Lg_1 \xrightarrow{f} \Lg_1\inv K\cong\La\slash W_\La$. We have:
\beqn
\begin{CD}
\widetilde\cO @>{\tilde f}>> f(\widetilde\cO) @>{\sim}>>\La^\phi\slash W_{\La^\phi}
\\
@VVV @VVV  @VVV
\\
\Lg_1 @>{f}>> \Lg_1\inv K @>{\sim}>>\La\slash W_\La
\end{CD}
\eeqn
where the vertical arrows are inclusions and the upper righthand corner is constructed as follows. Consider an element $e\in\cO$ and a normal $\fs\fl_2$-triple $\phi= (e,f,h)$ such that $f\in\Lg_1$ and $h\in\Lg_0$. Recall that we have
\beqn
\Lg^e=\Lg^\phi\oplus\fu^e,
\eeqn
where $\Lg^\phi=\Lg^e\cap\Lg^h, \ \ \fu=\oplus_{i\geq 1}\Lg(i)$, see~\eqref{grading}.
Thus
\beqn
\Lg_1^e=\Lg_1^\phi\oplus(\fu^e\cap\Lg_1).
\eeqn

 Let $\La^\phi\subset\Lg_1^\phi$ be a Cartan subspace such that every semisimple element in $\Lg_1^\phi$ is $K^\phi$-conjugate to some element in $\La^\phi$,  where $K^\phi=G^\phi\cap K=Z_K(e)\cap Z_K(f)\cap Z_K(h)$. We choose $\La^\phi$ such that it lies in $\La$. 
Let $W_{\La^\phi}$ be the little Weyl group $N_{K^\phi}(\La^\phi)/Z_{K^\phi}(\La^\phi)$. The same argument  as in Appendix~\ref{dual strata} shows that $f(\widetilde \cO)\cong \La^\phi/W_{\La^\phi}$ and we write 
\beqn
\tilde{f}:\tilde{\cO}\to\La^\phi/W_{\La^\phi}.
\eeqn Note also that, analogously to~\eqref{weyl equality} we have
\beq
\label{symmetric weyl equality}
N_{K}(\La^\phi)/Z_{K}(\La^\phi) = N_{K^\phi}(\La^\phi)/Z_{K^\phi}(\La^\phi)=W_{\La^\phi}\,.
\eeq

For an element $x\in\Lg_1$, we write $x=x_s+x_n$ for the Jordan decomposition of $x$ into semisimple part $x_s$ and nilpotent part $x_n$. Then $x_s,\,x_n\in\Lg_1$ and $[x_s,x_n]=0$. Let  $(\La^\phi)^{rs}$ denote the regular semisimple locus of $\La^\phi$ defined with respect to the symmetric pair $(G^\phi,K^\phi)$. Proceeding again as in  Appendix~\ref{dual strata} we have the following:
\begin{lem} 
 If $x=x_s+x_n\in\Lg_1^e$ and $x_s\in(\La^\phi)^{rs}$, then $x_n\in\bar\cO$.
\end{lem}

We now define 
\beq\label{def of check o}
\widecheck\cO = \{y\in \widetilde\cO \mid y=y_s+y_n, \ \  \tilde{f}(y)\in (\La^\phi)^{rs}/W_{\La^\phi}, \ \ y_n\in\cO\}\,.
\eeq
 This establishes a correspondence:
\beqn
\cO \leftrightarrow \widecheck\cO.
\eeqn
\indent Repeating the arguments in Appendix~\ref{dual strata} we obtain the following description of the equivariant fundamental groups of $\widecheck \cO$:
\beq
\label{identification}
\begin{CD}
1 @>>>  Z_{K^\phi}(\fa^\phi)/Z_{K^\phi}(\fa^\phi)^0@>>>\pi_1^{K^\phi}((\Lg_1^\phi)^{rs})@>{\tilde{q}}>> B_{W_{\La^\phi}}@>>> 1
\\
@. @| @| @| @.
\\
1 @>>> Z_K(\fa^\phi+e)/Z_K(\fa^\phi+e)^0@>>>\pi_1^K( \widecheck\cO)@>{\tilde{q}}>> B_{W_{\La^\phi}}@>>> 1\,,
\end{CD}
\eeq
where $B_{W_{\La^\phi}}=\pi_1((\La^\phi)^{rs}/W_{\La^\phi})$ is the braid group associated to ${W_{\La^\phi}}$. Recall that, as in Appendix~\ref{dual strata}, we use the terminology ``braid group" even when $W_{\La^\phi}$ is not a Coxeter group. 

In view of~\eqref{nilpsupport}, each character sheaf is supported on some $\overline{\widecheck\cO}$. For each symmetric pair  considered here, we will describe explicitly the set of nilpotent orbits $\cO$ for which the corresponding $\overline{\widecheck\cO}$  supports a character sheaf. We also write down the representations of $\pi_1^K(\widecheck\cO)=\pi_1^{K^\phi}((\Lg_1^\phi)^{rs})$ whose $\on{IC}$-sheaves are character sheaves. When $\widecheck\cO$ supports a character sheaf, ${W_{\La^\phi}}$ turns out to be a Coxeter group and the rows in~\eqref{identification} can be split as in~\eqref{diagram-fundamental group} following~\cite[Section 4.4]{GVX}. We will later construct such explicit splittings. 

\subsection{Full support character sheaves}
\label{csfs}

In this subsection we recall the main construction from~\cite{GVX} and explain how it will be used to construct full support character sheaves. All such character sheaves are of the form $\on{IC}(\Lg_1^{rs},\cL)$, where $\cL$ is an irreducible $K$-equivariant local system on $\Lg_1^{rs}$.

Recall the notation from \S\ref{ssec-equiv fdgp}. Let $X_{\bar a}=f^{-1}(\bar a)$, the fiber of the adjoint quotient map $f:\Lg_1\to \La/W_{\La}$ at $\bar a\in  \La^{rs}/W_{\La}$.
Let $\hat I$ denote the set of irreducible characters of the 2-group $I$. Consider a character $\chi\in\hat I$ and note that the equivariant fundamental group of $X_{\bar a}$ is given by\beqn
\pi_1^K (X_{\bar a}, a) \cong I=Z_K(\fa)/Z_K(\fa)^0\,,
\eeqn
where $f(a)=\bar a$.
Therefore, the character $\chi$ gives rise to a rank one $K$-equivariant local system $\cL_\chi$ on $X_{\bar a}$. We base change $f$ to the family
\beqn
f_{\bar a}:\cZ_{\bar a}= \{(x,c)\in\Lg_1\times\bC\,|\,f(x)=c\,\bar a\} \to \bC
\eeqn
where the $\bC$-action on $\La/W_{\La}$ is induced by the action on $\fa$ so that $f(ca)=cf(a)$. By construction this family is $K$-equivariant. We define the nearby cycle sheaf associated to $\chi\in\hat I$ as
\beq\label{nearby cycle sheaves}
P_\chi=\psi_{f_{\bar a}} \cL_{\chi}[-]\in\on{Perv}_K(\cN_1).
\eeq
As the group $K$ is not necessarily connected, a character of the component group $K/K^0=I/I^0$ enters the description of the Fourier transform $\fF P_\chi$, where we have written
 $$
 I^0=Z_{K^0}(\fa)/Z_{K^0}(\fa)^0,
 $$
 see~\cite[\S 3.4]{GVX}. It was denoted by $\tau$ there but we denote it by $\iota$ in this paper, i.e.,
\beqn
\iota:I\to I/I^0=K/K^0\to\{\pm 1\}
\eeqn
is the character determined by the action of $M_\bbR/(M_\bbR\cap K_\bbR^0)\cong K/K^0$ on $\wedge^{\on{top}}(\mathfrak k_\bbR/\fm_\bbR)$, where $K_\bbR$ is the compact form of $K$, $M_\bbR=Z_{K_\bbR}(\fa)$, $\fk_\bbR=\on{Lie}K_\bbR$ and $\fm_\bbR=\on{Lie}M_\bbR$. 
In particular, if $K=K^0$, then $\iota$ is the trivial character. Recall that the $W_\fa$ action on $\hat I$ leaves $\iota$ fixed. Let $\bC_\iota$ denote the rank one $K$-equivariant local system on $\Lg_1^{rs}$ given by the representation of $\pi_1^{K}(\Lg_1^{rs})=I\rtimes B_{W_\fa}$ where $I$ acts via the character $\iota$ and $B_{W_\fa}$ acts trivially. 
We have (\cite[Theorem 3.6]{GVX})
\beq 
\label{big strata}
\fF P_\chi = \operatorname{IC}(\Lg_1^{rs}, \cM_\chi\otimes \bC_\iota)
\eeq
where  $\cM_\chi$ is the $K$-equivariant local system on $\Lg_1^{rs}$ given by the following $\widetilde B_{W_\La}=\pi_1^K(\Lg_1^{rs})$ representation
\beq\label{the representation M_chi}
M_\chi \ = \  \bC[\widetilde B_{W_\La}]\otimes_{\bC[\widetilde B_{W_\La}^{\chi,0}]} (\bC_\chi\otimes\cH_{W_{\La,\chi}^0})\,.
\eeq
To explain the notations in the above formula, recall the map $p:B_{W_\fa}\to W_\fa$ in~\eqref{diagram-fundamental group}. Let 
\beqn
W_{\La,\chi}=\on{Stab}_{W_\fa}(\chi)\subset W_\fa,\ \ B_{W_{\fa}}^\chi=p^{-1}(W_{\fa,\chi})\subset B_{W_\fa}.
\eeqn
Recall the Coxeter group $W_{\La,\chi}^0\subset W_{\La,\chi}$ defined by
\beqn
\begin{gathered}
W_{\La,\chi}^0\ =\text{ the subgroup of $W_\fa$ generated by reflections $s\in W_\fa$ such that, } \\ \text{either }\delta(s)>1,\text{ or }\delta(s)=1\text{ and }\chi(\check{\alpha}_s(-1))=1,
\end{gathered}
\eeqn
where $\delta(s)$ is defined in~\eqref{definition of delta(s)} and $\check\alpha_s(-1)$ is viewed as an element in $I$ via the natural projection $Z_K(\fa)\to I$. 
We show in \S\ref{ssec-dual side} that the quotient $W_{\La,\chi}/W_{\La,\chi}^0$ is a 2-group, see~\eqref{the 2 group}.
We set
\beqn
 B^{\chi,0}_{W_\La} \ = \  p^{-1}(W^0_{\La,\chi})\subset B_{W_\fa}\ \,\text{ and }\,\ \widetilde B^{\chi,0}_{W_\La}= I\rtimes B^{\chi,0}_{W_\La}\subset\widetilde{B}_{W_\fa}\text{ (see~\eqref{semidirect})}.
\eeqn
Let $\cH_{W_{\La,\chi}^0}$ be the Hecke algebra associated to the Coxeter group $W_{\La,\chi}^0$ defined as follows. We choose a set of simple reflections $s_{\bar\alpha_1}, \dots, s_{\bar\alpha_\ell}$ that generate  $W_{\La,\chi}^0$.  
We write $T_i$ for the generators of $\cH_{W_{\La,\chi}^0}$ associated to $s_{\bar\alpha_i}$. Then the Hecke algebra $\cH_{W_{\La,\chi}^0}$ is generated by the $T_i$ subject to the braid relations plus the relations
\begin{equation*}
(T_i-1)(T_i + q_i) \ = \ 0\,,\ \ q_i = (-1)^{\delta(s_{\bar\alpha_i})}\,.
\end{equation*} 
Now $\bC_\chi\otimes\cH_{W_{\La,\chi}^0}$ is the $\bC[\widetilde B_{W_\La}^{\chi,0}]$-module where $I$ acts via the character $\chi$ and $B_{W_{\La}}^{\chi,0}$ acts via the composition of maps $
\bC[B_{W_{\La}}^{\chi,0}]\twoheadrightarrow \bC[B_{W_{\La,\chi}^0}]\twoheadrightarrow \cH_{W_{\fa,\chi}^0}$. Here $B_{W_{\La,\chi}^0}$ is the braid group associated to the Coxeter group $W_{\fa,\chi}^0$. The first map is induced by the map $B_{W_{\La}}^{\chi,0}=\pi_1(\La^{rs}/W_{\La,\chi}^0)\twoheadrightarrow B_{W_{\La,\chi}^0}=\pi_1(\La_\chi^{rs}/W_{\La,\chi}^0)$, which in turn is induced by the inclusion $\La_\chi^{rs}\subset \La^{rs}$, where $\La_\chi^{rs}$ is the root hyperplane arrangement corresponding to $W_{\fa,\chi}^0$. The second map is given by the natural projection map $\bC[ B_{W_{\La,\chi}^0}]\twoheadrightarrow\cH_{W_{\fa,\chi}^0}$.

 Let us write 
\beq\label{loc-full}
\begin{gathered}
\Theta_{(G,K)}=\{\text{irreducible representations of $\pi_1^K(\Lg_1^{rs})$ that appear}\\
\quad\text{ as composition factors of $M_{\chi}\otimes\bC_\iota$, $\chi\in\hat I$}\}.
\end{gathered}
\eeq
For each $\rho\in\Theta_{(G,K)}$, let $\cL_\rho$ denote the corresponding $K$-equivariant local system on $\Lg_1^{rs}$. It follows from~\eqref{nearby cycle sheaves} and~\eqref{big strata} that 
\beq\label{char-fullsupp}
\on{IC}(\Lg_1^{rs},\cL_\rho)\in\on{Char}_K^{\mathrm{f}}(\Lg_1),\ \rho\in\Theta_{(G,K)}.
\eeq
We will show that for the symmetric pairs considered in this paper,
\beq\label{nearby and full support}
\left\{\on{IC}(\Lg_1^{rs},\cL_\rho)\,|\,\rho\in\Theta_{(G,K)}\right\}=\on{Char}_K^{\mathrm{f}}(\Lg_1)\,.
\eeq
\indent To determine the set $\Theta_{(G,K)}$, it suffices to determine the composition factors of $M_{\chi}$.  Let us note that
\beqn
M_\chi \ = \  \bC[\widetilde B_{W_\La}]\otimes_{\bC[\widetilde B_{W_\La}^\chi]}\otimes(\bC_\chi\otimes(\bC[ B_{W_\La}^{\chi}]\otimes_{\bC[B_{W_\La}^{\chi,0}]} \cH_{W_{\La,\chi}^0}))\,.
\eeqn
Making use of the semidirect product decomposition~\eqref{semidirect} we conclude that the irreducible representations of $\widetilde{B}_{W_\fa}$ appearing as composition factors of $M_\chi$ are of the form $\bC[\widetilde B_{W_\La}]\otimes_{\bC[\widetilde B_{W_\La}^\chi]}\otimes(\bC_\chi\otimes\rho)$, where $\rho$ is an irreducible representation of $B_{W_\La}^{\chi}$ which appears as a composition factor in $\bC[ B_{W_\La}^{\chi}]\otimes_{\bC[B_{W_\La}^{\chi,0}]} \cH_{W_{\La,\chi}^0}$. Since $B_{W_\La}^{\chi}/B_{W_\La}^{\chi,0}$ is a 2-group, it suffices to study the decomposition $\bC[ B_{W_\La}^{\chi}]\otimes_{\bC[B_{W_\La}^{\chi,0}]} \tau$ using Clifford theory, where $\tau\in\on{Irr}\cH_{W_{\La,\chi}^0}$. In particular, we see that all irreducible representations of $\widetilde{B}_{W_\fa}$ which appear as composition factors of $M_\chi\otimes\bC_\iota$ can be obtained as quotients of $M_\chi\otimes\bC_\iota$.

For $\tau\in\on{Irr}\cH_{W_{\La,\chi}^0}$, we write $V_{\tau,\chi}= \bC[\widetilde B_{W_\La}]\otimes_{\bC[\widetilde B_{W_\La}^{\chi,0}]}\otimes(\bC_\chi\otimes\tau)$. When $W_{\La,\chi}^0= W_{\La,\chi}$,  $V_{\tau,\chi}$ is an irreducible representation of $\widetilde B_{W_\La}$. When $W_{\La,\chi}^0\neq W_{\La,\chi}$,  
we write $V_{\tau,\chi}^\delta$ for the non-isomorphic irreducible summands   of $V_{\tau,\chi}$ as representations of  $\widetilde B_{W_\fa}$.

\subsection{The nearby cycle construction in the split case and endoscopy}\label{sec-endoscopy}

The groups $W_{\La,\chi}$ and $W_{\La,\chi}^0$ are not in general Weyl groups of subgroups of $G$. We show in this subsection that they are Weyl groups of the dual group $\check G$ in the case of split symmetric pairs. 
We explain at the end of the subsection how the case of non-trivial characters $\chi$ can be viewed as endoscopy. 

In this subsection let $(G,K)$ be a split symmetric pair. Then $A=T$, $W_\La = W=N_G(T)/Z_G(T)$, and $I=Z_K(\La)$. In particular, $\Phi^\rR=\Phi$.

We first discuss the interpretation of  
\beqn
W_{\fa,\chi}^0=\langle s_{\bar\alpha}, \alpha\in\Phi,\,\chi(\check\alpha(-1))=1\rangle\subset W_{\fa,\chi}=\on{Stab}_{W_\fa}(\chi)
\eeqn
as a Weyl group on the dual side.

Recall that we have a surjection $A[2] \to I$ of order two elements of $A$ to $I$. Thus, we can regard the character $\chi\in \hat I$ as a  character of $A[2]$. 

Let $\check{G}$ be the dual group and $\check{A}\subset \check{G}$ the dual torus. It is standard, and not difficult to see, that  we can identify $\on{Hom}(A[2],\bG_m)\cong\check{A}[2]$, the order $2$ elements in $\check A$. Thus, $\chi$ regarded as an element in $\check{A}[2]$, gives rise to an involution on $\check G$ which we denote by $\check\theta_\chi$, i.e.
$
\check\theta_\chi=\on{Int}\chi.
$
Let 
 \beqn
\check K(\chi) = \check G^{\check\theta_\chi},\ \ \check K(\chi)^0=(\check G^{\check\theta_\chi})^0.
\eeqn
 We show that
\begin{subequations}
\beq\label{weyl-1}
W_{\fa,\chi}\ = \ W(\check K(\chi),\check A)
\eeq
\beq\label{weyl-2}
W_{\fa,\chi}^0\ =\ W(\check K(\chi)^0,\check A).
\eeq 
\end{subequations}

 The root system of $(\check G,\check A)$ can be identified with $\check\Phi=\{\check\alpha\,|\,\alpha\in\Phi\}$. With  respect to the $\check\theta_\chi$-stable torus $\check A$, all roots in $\check\Phi$ are imaginary since $\check\theta_\chi|_{\check A}=1$. They split into compact imaginary and non-compact imaginary according to the value of $\chi(\check\alpha(-1))$. So, we set
\beqn
\check \Phi_{\mathrm{ci}}=\{\check \alpha\mid\alpha\in\Phi,\ \chi(\check \alpha(-1))=1\} \qquad \check \Phi_{\mathrm{nci}}=\{\check \alpha\mid\alpha\in\Phi,\ \chi(\check \alpha(-1))=-1\}\,.
\eeqn
By construction we have 
\beqn
    W_{\fa,\chi}\ = \ \{w\in W\mid w:\check \Phi_{\mathrm{ci}} \to \check \Phi_{\mathrm{ci}}, \   w:\check \Phi_{\mathrm{nci}} \to \check \Phi_{\mathrm{nci}}\}\  = \ \{w\in W\mid w:\check \Phi_{\mathrm{ci}} \to \check \Phi_{\mathrm{ci}}\}\,.
\eeqn

Equation~\eqref{weyl-1} can be seen as follows. First, 
$$
W(\check K(\chi),\check A)=N_{\check K(\chi)}(\check A)/Z_{\check K(\chi)}(\check A)=N_{\check K(\chi)}(\check A)/\check A \subset N_{\check G}(\check A)/\check A=W\,.
$$ 
Let $n_w\in N_{\check G}(\check A)$ denote a representative of $w\in W$. Now, $n_w\in \check{K}(\chi)$ if and only if $\chi n_w\chi^{-1}=n_w$ if and only if $n_w^{-1}\chi n_w=n_w$, i.e., if and only if $w\chi=\chi$. Equation~\eqref{weyl-2} follows from the fact that the roots of $( \check K(\chi)^0,\check A)$ are the $\check\alpha$'s such that $\chi(\check\alpha(-1))=1$. 

Note that the Weyl group $W(\check K(\chi),\check A)$ is not necessarily a Coxeter group as $\check K(\chi)$ might not be connected. It is connected if $\check G$ is simply connected and then $G$ is adjoint. The quotient $W(\check K(\chi),\check A)/W(\check K(\chi)^0,\check A)= \check K(\chi)/\check K(\chi)^0$ is a 2-group,  called the $R$-group. 

 The above results  imply that 
\beq\label{split 2-group}
\text{$W_{\fa,\chi}/W_{\fa,\chi}^0=\check K(\chi)/\check K(\chi)^0$ is a 2-group.}
\eeq 
Moreover,
\beq\label{adjoint case-1}
\text{if $G$ is adjoint, then $W_{\fa,\chi}=W_{\fa,\chi}^0$}.
\eeq

Cheng-Chiang Tsai pointed out to us that the previous discussion can be viewed from the endoscopic point of view as follows. 
Note that in the split case for the trivial character $\chi=1$, in~\eqref{the representation M_chi} the Hecke algebra $\cH_{W^0_{\fa,\chi}} = \cH_{W,-1}$, i.e., the Hecke algebra of $W$ with parameter $-1$.

For any $\chi\in \hat I$ we obtained the symmetric pair $(\check G,\check K(\chi))$. By construction, the maximal torus $\check A$ of $\check G$ is also a maximal torus of  $\check K(\chi)$. Let  $G(\chi)$ be the dual group of $\check K(\chi)^0$ and consider the split symmetric pair of $G(\chi)$. By construction the split maximal torus $A(\chi)$ of $G(\chi)$ is canonically isomorphic to $A$ and $W(G(\chi), A)=W(\check K(\chi), \check A) =W_{\fa,\chi}^0$. We have
\begin{equation}
\cH_{W^0_{\fa,\chi}}\\  =\\  \cH_{W_{\fa,\chi}^0,-1}\,.
\end{equation}
The Hecke algebra on the right hand side is the one corresponding the trivial character for the split pair of $G(\chi)$. We can thus interpret the the character sheaves associated to non-trivial characters $\chi$ as arising from the endoscopic groups $G(\chi)$.

\subsection{The groups $W_{\fa,\chi}$ and $W_{\fa,\chi}^0$ for a general symmetric pair}\label{ssec-dual side}

 In this subsection we discuss the groups $W_{\fa,\chi}$ and $W_{\fa,\chi}^0$ for a general symmetric pair. Recall from~\cite[Corollary 5.2]{GVX} that
 \beqn
 W_{\fa,\chi}=W^{\rR}_\chi\rtimes (W^\rC)^\theta,\ \ W^{\rR}_\chi=W^{\rR}\cap W_{\fa,\chi}.
 \eeqn
 To determine $W^{\rR}_\chi$,  let us consider
\beqn
G_\rs= Z_G(\fc)/C^0.
\eeqn
Here $C=T^\theta$ and $\fa=\on{Lie}C$, where $T\subset Z_G(\fa)$ is a maximal torus.
The involution $\theta$ induces an involution $\theta_\rs$ on $G_\rs$ and we write $K_\rs= G_\rs^{\theta_\rs}$. Clearly $A=T/C^0$ is a maximal torus of $G_\rs$ and it is $\theta_\rs$-split. Thus the symmetric pair $(G_\rs,K_\rs)$ is a split pair with $W(G_\rs,A)=W^\rR$ and the root system of $(G_\rs,A)$ can be identified with $\Phi^\rR$. Furthermore, we have
\beqn
I \ = \ Z_K(\La)/Z_K(\La)^0\ =\ C/C^0 \ \ = \ Z_{K_\rs}(\La)/Z_{K_\rs}(\La)^0\,.
\eeqn
Let $W_{\fa_{(G_\rs,K_\rs)}}=W^\rR$ be the little Weyl group of the split symmetric pair $(G_\rs,K_\rs)$. We write $W_{\fa_{(G_\rs,K_\rs)},\chi}$ for $\on{Stab}_{W_{\fa_{(G_\rs,K_\rs)}}}(\chi)$ and $W_{\fa_{(G_\rs,K_\rs)},\chi}^0$ for the corresponding subgroup of $W_{\fa_{(G_\rs,K_\rs)},\chi}$.
Define $W^{\rR,0}_\chi$ to be the subgroup of $W^\rR$ generated by those reflecitions $s_\alpha$, $\alpha\in\Phi^\rR$, such that $\chi(\check\alpha(-1))=1$. Then we have
\beq\label{reduction to Gs}
W^{\rR}_\chi=W_{\fa_{(G_\rs,K_\rs)},\chi},\ \quad W^{\rR,0}_\chi=W_{\fa_{(G_\rs,K_\rs)},\chi}^0.
\eeq

Note that $(W^\rC)^\theta\ltimes (W_{\chi}^R)^0\subset W_{\fa,\chi}^0\subset (W^\rC)^\theta\ltimes W_{\chi}^\rR=W_{\fa,\chi}$. Applying~\eqref{split 2-group} and~\eqref{adjoint case-1} to the split pair $(G_\rs,K_\rs)$ we see that 
\beq\label{the 2 group}
\text{$W_{\fa,\chi}/W_{\fa,\chi}^0\subset W_{\chi}^\rR/(W_{\chi}^\rR)^0$ is a 2-group.}
\eeq 
Moreover,
\beq\label{adjoint case}
\text{if $G_\rs$ is adjoint, then $W_{\fa,\chi}=W_{\fa,\chi}^0=(W^\rC)^\theta\ltimes (W_{\chi}^R)^0$}.
\eeq

\subsection{Nilpotent support character sheaves}\label{sec-nilpotent support cs}

In this subsection we classify nilpotent support character sheaves. They exist in the cases we consider only when $\theta$ is inner and one expects that  it also holds for exceptional groups. 

Note that if the Fourier transform of $\on{IC}(\cO,\cE)$ has nilpotent support then the support of the Fourier transform is also $\bar\cO$. Moreover $\cO=\widecheck \cO$ and so $\cO$ is self dual. This can be seen as follows. Assume that $\fF\on{IC}(\cO,\cE)=\on{IC}(\cO',\cE')$.  By~\eqref{nilpsupport}, $\bar\cO'=\bar{\widecheck\cO}''$, where $\cO''\subset\bar\cO$. Now if $\widecheck\cO''$ is nilpotent, then $\widecheck\cO''=\cO''$. We conclude that $\cO'=\cO''\subset\bar\cO$. Similarly, we have $\cO\subset\bar\cO'$. Thus $\cO'=\cO$.  In particular, we see that if the Fourier transform of $\on{IC}(\cO,\cE)$ has nilpotent support, then $\cO$ is distinguished. That is, $\cO$ consists of distinguished nilpotent elements, i.e., $e\in\cN_1$ such that $\Lg_1^e:=Z_{\Lg_1}(e)$ consists of nilpotent elements.

Here is a general construction of nilpotent support character sheaves. 
Let $X$ denote the flag variety of $G$. Consider a closed $K$-orbit $Q$ on $X$ and its conormal bundle $T_Q^*X$.  Let $B\subset G$ be a point on $Q$, i.e., a $\theta$-stable Borel subgroup. Let $\Lb=\on{Lie}B$ and let $\Ln$ be the nilpotent radical of $\Lb$. We write $\Lb_i=\Lb\cap\Lg_i$, $\Ln_i=\Ln\cap\Lg_i$, $i=0,1$ and $B_K=B\cap K=B^\theta$. 
We have a moment map $$\mu: T_Q^*X=K\times^{B_K}\Ln_1\to \cN_1.$$
Since $\theta$ is inner, we have that (see, for example, \cite[\S3.2]{L})
\beq
\label{b1=n1}
\Lb_1=\Ln_1.
\eeq Using the functoriality of Fourier transform and \eqref{b1=n1}, we see that
\beqn
\fF(\mu_{*}\bC_{T_Q^*X}[-])\cong\mu_{*}\bC_{T_Q^*X}[-].
\eeqn
It follows that all direct summands of $\mu_{*}\bC_{T_Q^*X}[-]$  are nilpotent support character sheaves (up to shift). 

To describe the character sheaves that arise in this way, we first note that the image of $\mu(T_Q^*X)$ is the closure of the uique $K$-orbit $\cO$ in $\cN_1$ such that $\cO\cap\Lb_1=\cO\cap\Ln_1$ is dense in $\Lb_1=\Ln_1$; we denote this orbit by $\cO_B$. We say that the orbit $\cO_B$ is the Richardson orbit attached to (the $K$-orbit of) the $\theta$-stable Borel subgroup $B$. Let us now write 
\beqn
\pi_B:K\times^{B_K}\Ln_1\to\bar{\cO}_B \qquad \text{and}\qquad \mathring\pi_B:K\times^{B_K}\Ln^r_1\to {\cO}_B,\ \  \Ln^r_1= \Ln_1\cap {\cO}_B
\eeqn
for the moment map $\mu$ and its restriction to $\mu^{-1}({\cO}_B)$. 
 Let $\underline{\cN_1}=\cN_1/K$ denote the set of $K$-orbits on $\cN_1$. We write
\beq\label{richardson}
\underline{\cN_1^0}=\{\text{Richardson orbits attached to $\theta$-stable Borel subgroups}\}\subset\underline{\cN_1}.
\eeq
Let $A_K(\cO):=A_K(x)$, $x\in\cO$. We write $\widehat{A}_K(\cO)$ for the set of irreducible characters of $A_K(\cO)$.

For an orbit $\cO\in\underline{\cN_1^0}$, let $\Pi_\cO\subset\widehat{A}_K(\cO)$ denote the set of irreducible characters of $A_K(\cO)$ which appear in the permutation representations of $A_K(\cO)$ on the sets of irreducible components of $\pi_B^{-1}(x)$, where $x\in\cO$ and $B$ runs through the $\theta$-stable Borel subgroups with $\cO_B=\cO$ (up to $K$-conjugacy). Given $\phi\in\widehat{A}_K(\cO)$,  we write $\cE_\phi$ for the irreducible $K$-equivariant local system on $\cO$ corresponding to $\phi$. The $\on{IC}(\cO,\cE_\phi)$ are nilpotent support character sheaves. We see this as follows. Let $d$ be the fiber dimension of the fibration $\mathring\pi_B$. Decomposing $R^d(\mathring\pi_B)_*(\bC)$ into direct summands
\beqn
R^d(\mathring\pi_B)_*(\bC) \ \cong \ \bigoplus_{\phi\in\Pi_{\cO_B}} \cE_\phi\,
\eeqn
and running through all the closed $K$-orbits on $X$ we obtain the set of the $\cE_\phi$ as above. 

We will show that all nilpotent support character sheaves can be obtained in this way:
\beq
\label{nilp summands}
\on{Char}_K^\rn(\Lg_1)=\left\{\on{IC}(\cO,\cE_\phi)\,|\,\cO\in\underline{\cN_1^0},\ \phi\in\Pi_\cO\right\}.
\eeq
 
It turns out that the nilpotent support character sheaves can be determined as follows. Let $\cO\in\underline{\cN_1^0}$. Note that two $\theta$-stable Borel subgroups that are not conjugate under $K$ can give rise to the same Richardson orbit. In this case, the sets of irreducible components of $\pi_{B_1}^{-1}(x)$ and $\pi_{B_2}^{-1}(x)$, $x\in\cO$ can be different. However, we will show that we can find a $\theta$-stable Borel subgroup $B$ such that $\cO_B=\cO$ and such that $\Pi_\cO$ coincides with the set of irreducible characters of $A_K(\cO)$ which appear in the permutation representation of $A_K(\cO)$ on the set of irreducible components of $\pi_B^{-1}(x)$, $x\in\cO$. Thus, to determine $\on{Char}_K^\rn(\Lg_1)$, for any $\cO\in\underline{\cN_1^0}$  we only need to consider that specific $\theta$-stable Borel in~\eqref{nilp summands}.

\begin{remark}
Following Lusztig \cite{L3}, we say that  an orbit $\cO\subset\cN_1$ is $\fF$-thin, if $\fF(\on{IC}(\cO,\bC))$ has nilpotent support. It follows that  (in the cases considered here)
an orbit $\cO\subset\cN_1$ is $\fF$-thin, if and only if, $\cO$ is a Richardson orbit attached to some $\theta$-stable Borel subgroup.
This is consistent with the speculation in \cite{L3} that $\fF$-thin orbits exist only if $\theta$ is inner. 
\end{remark}
\subsection{Induced character sheaves}\label{sec-induction}
In order to produce all character sheaves, we consider parabolic induction from $\theta$-stable Levi subgroups. 
We recall the notion of induction functors in our setting (see \cite{H,L}). Let $L$ be a $\theta$-stable Levi subgroup contained in a $\theta$-stable parabolic subgroup $P\subset G$. We write 
\beqn
\text{$\Ll=\on{Lie}L$, $\Lp=\on{Lie}P$, $L_K=L\cap K$, $P_K=P\cap K$, $\Ll_1=\Ll\cap\Lg_1$, $\Lp_1=\Lp\cap\Lg_1$.}
\eeqn
The  parabolic  induction functor 
$
\on{Ind}_{\fl_1\subset\mathfrak{p}_1}^{\Lg_1}:D_{L_K}(\fl_1)\to D_K(\Lg_1)
$
is defined as follows. Consider the diagram
\beq\label{induction diagram}
\xymatrix{\mathfrak{l}_1&\Lp_1\ar[l]_-{\on{pr}}&K\times\fp_1\ar[l]_-{p_1}\ar[r]^-{p_2}&K\times^{P_K}\fp_1\ar[r]^-{{\pi}}&\Lg_1},
\eeq
where $p_1$ and $p_2$ are natural projection maps and ${\pi}:(k,x)\mapsto \on{Ad}(k)(x)$. The maps in~\eqref{induction diagram} are $K\times P_K$-equivariant, where $K$ acts trivially on $\Ll_1$ and $\Lp_1$, by left multiplication on the $K$-factor on $K\times\Lp_1$ and on $K\times^{P_K}\Lp_1$, and by adjoint action on $\Lg_1$, and $P_K$ acts on $\Ll_1$ by $a.l=\on{pr}(\on{Ad}a(l))$, by adjoint action on $\Lp_1$, by $a.(k,p)=(ka^{-1},\on{Ad}a(p))$ on $K\times\Lp_1$, trivially on $K\times^{P_K}\Lp_1$ and $\Lg_1$. 

Let $\cT$ be a complex in $D_{L_K}(\fl_1)$. Then $(\on{pr}\circ p_1)^*\cT\cong p_2^*\cT'$ for a well-defined complex $\cT'$ in  $D_K(K\times^{P_K}\fp_1)$. Set
$
\on{{I}nd}_{\fl_1\subset\mathfrak{p}_1}^{\Lg_1}\cT={\pi}_{!}\cT'[\dim P-\dim L].
$
Note that as the map ${\pi}$ is proper the induction functor takes semisimple objects to semisimple objects. 
Moreover, the Fourier transform commutes with induction:
\beq\label{commutativity}
\fF( \on{Ind}_{\fl_1\subset\mathfrak{p}_1}^{\Lg_1}\cT)\cong\on{Ind}_{\fl_1\subset\mathfrak{p}_1}^{\Lg_1}(\fF(\cT)).
\eeq

In our setting we will define a family of $\theta$-stable parabolic subgroups $P$ and study the character sheaves obtained by inducing full support character sheaves from the corresponding $\theta$-stable Levi subgroups $L$.

\subsection{An example}
\label{example}
In this subsection we consider the special case $(SL_2,SO_2)$ to illustrate how the description of the character sheaves depends on the choice of the splitting of the top row of~\eqref{diagram-fundamental group}. Strictly speaking we are not considering the $SL_2$ case in this paper, but this example illustrates the basic idea. The choice of a splitting comes up also in~\eqref{explicit fundamental}. It affects the labelling of the local systems in the way we illustrate here. 

We have $\Lg_1=\bC^2$ and $K=\bC^*$. The $\bC^*$ acts on $\bC^2$ by $c \cdot (x,y)= (cx,c^{-1}y)$. The nilpotent cone is given by $\cN_1=\{(x,y)\mid xy=0\}$. There are three nilpotent orbits, the zero orbit $\cO_0$, the two regular nilpotent orbits $\cO_1^{\rm I}$ and $\cO_1^{\rm II}$, corresponding to the coordinate axes. The orbits $\cO_1^{\rm I}$ and $\cO_1^{\rm II}$ are self dual and $\widecheck \cO_0 = \Lg_1^{rs}$. In this case diagram~\eqref{diagram-fundamental group}  becomes:
\beqn
\begin{CD}
1 @>>> \bZ/2\bZ@>>>(\bZ\oplus\bZ)/2\bZ@>{\tilde{q}}>>\bZ@>>> 1
\\
@. @| @VV{\tilde p}V @VV{p}V @.
\\
1 @>>> \bZ/2\bZ@>>>\bZ/4\bZ@>{q}>> \bZ/2\bZ@>>> 1\,.
\end{CD}
\eeqn
There are two character sheaves with full support where the $I=\bZ/2\bZ$ acts nontrivially: the Fourier transforms of the $\on{IC}$-sheaves associated to non-trivial local systems on the orbits $\cO_1^{\rm I}$ and $\cO_1^{\rm II}$. Once we choose a splitting of $\tilde q$ one of those local system corresponds to the trivial representation of $B_1=\bZ$ and the other to the non-trivial one.

\section{Nilpotent support character sheaves }\label{sec-biorbital}

In this section we determine all  nilpotent support character sheaves.
They only arise for the symmetric pairs $(G,K)$ in \S\ref{ssec-classical pairs} for which $\theta$ is an inner involution.  We use the notation of that section. We assume that $\theta$ is inner in this section. Recall our strategy from \S\ref{sec-nilpotent support cs}. At the beginning of this section we will concentrate on specifying for each $\cO\in\underline{\cN_1^0}$ (see~\eqref{richardson}) a special $\theta$-stable Borel $B$ such that we only need to consider
\beqn
\mathring \pi_B:K\times^{B_K}\Ln^r_1\to\cO_B=\cO
\eeqn
to obtain all character sheaves supported on $\bar\cO$. 

The set  $\underline{\cN_1^0}$ of Richardson orbits has been determined in each case by Trapa in \cite{T}.  We recall his result:
\begin{proposition}[\cite{T}]\label{prop-Richardson}
 Suppose that $\theta$ is inner.  Let $\cO_\lambda$ be a $K$-orbit in $\cN_1$ corresponding to a signed Young diagram $\lambda$ of the form~\eqref{signed Young diagram-2}. Then $\cO_\lambda\in\underline{\cN_1^0}$  if and only if
\begin{eqnarray*}
\text{{\rm{(AIII)}}} &&\text{ for each $i$, either $p_i=0$ or $q_i=0$.}\\
\text{{\rm{(CI)}}} &&\text{ for $\lambda_i$ odd, $p_i=q_i=0$ , and for $\lambda_i$ even, either $p_i=0$ or $q_i=0$.}\\
\text{{\rm{(CII)}}}&&\text{ for $\lambda_i$ even, $p_i=q_i\leq 1$, and for $\lambda_i$ odd, either $p_i=0$ or $q_i=0$.}\\
\text{{\rm{(DIII)}}}&& \text{ for $\lambda_i$ odd, $p_i=q_i\leq 1$, and for $\lambda_i$ even, either $p_i=0$ or $q_i=0$.}\\
\text{{\rm{(BDI)}}}&& \text{ the following two conditions hold}
\end{eqnarray*}
\begin{enumerate}
\item for $\lambda_i$ even, $p_i=q_i=0$, and for $\lambda_i$ odd, either $p_i=0$ or $q_i=0$. Namely the signed Young diagram  $\lambda$ is of the form $(2\mu_1+1)_{\epsilon_1}(2\mu_2+1)_{\epsilon_2}\cdots(2\mu_k+1)_{\epsilon_k}$, where $\mu_1\geq\mu_2\geq\cdots\geq\mu_k\geq 0$, $\epsilon_i\in\{0,1\}$ and $\epsilon_i=\epsilon_j$ if $\mu_i=\mu_j$. \item 
\begin{itemize}
\item[(a)] $\text{if $N$ is odd,  $\epsilon_1\equiv q\,\, \nmod 2$ and $\epsilon_{2i}+\mu_{2i}\equiv \epsilon_{2i+1}+\mu_{2i+1}\,\nmod 2$, $i\geq 1$;}$
\item[(b)]$\text{if $N$ is even,  $\epsilon_{2i-1}+\mu_{2i-1}\equiv \epsilon_{2i}+\mu_{2i}\mod 2$, $i\geq 1$.}$
\end{itemize}
\end{enumerate}
 \end{proposition}
Let us write
\beqn
\on{SYD}^0_{(G,K) }
\eeqn
for the set of signed Young diagrams corresponding to the orbits in $\underline{\cN_1^0}$.

\begin{thm}
\label{nil thm-1}
Assume that $(G,K)$ is $
 (Sp_{2n},Sp_{2p}\times Sp_{2q}),\  (Sp_{2n},GL_n),$ or $
 (SO_{2n},GL_n).
$
 The set of nilpotent support character sheaves is
 \beqn
 \on{Char}_K^\rn(\Lg_1)=\{\on{IC}(\cO,\bC)\,|\,\cO\in\underline{\cN_1^0}\}.
 \eeqn
\end{thm}
\begin{remark} The theorem holds also for the symmetric pair $(GL_n, GL_p\times GL_q)$. This is a special case of \cite{L}.
\end{remark}

Suppose that $(G,K)=(SO_N,S(O_p\times O_q))$, where either $p$ or $q$ is even. Let $\cO\in\underline{\cN_1^0}$. We define a set of characters $\Pi_\cO\subset \widehat A_K(\cO)$ and we show that  it is exactly  the set of characters $\Pi_\cO\subset \widehat A_K(\cO)$ defined in \S\ref{sec-nilpotent support cs}. Proposition~\ref{prop-Richardson} implies that $\cO=\cO_\mu$, where 
\beq\label{the orbit mu}
\mu=(2\mu_1+1)_{\epsilon_1}^{m_1}(2\mu_2+1)^{m_2}_{\epsilon_2}\cdots(2\mu_s+1)^{m_s}_{\epsilon_s},
\eeq
$\mu_1>\mu_2>\cdots>\mu_s\geq 0$, $m_i>0$  and $\epsilon_i\in\{0,1\}$, $1\leq i\leq s$. 
From~\S\ref{nilp comp} we conclude 
\beq
\label{comp group explicit}
A_K(\cO)\cong S(O_{m_1}\times \cdots\times O_{m_s})/S(O_{m_1}\times \cdots\times O_{m_s})^0\cong(\bZ/2\bZ)^{s-1}.
\eeq
Let $\delta_i$ correspond to the non-trivial element in $S(O_{m_{i}}\times O_{m_{i+1}})/(SO_{m_{i}}\times SO_{m_{i+1}})=\bZ/2\bZ$, $1\leq i\leq s-1$. Then $A_K(\cO)$ is generated by $\delta_i$, $1\leq i\leq s-1$.

We define $\Omega_\cO=\Omega_\mu\subset\{1,\ldots,s\}$ to be the set of $j\in[1,s]$ such that 
\begin{enumerate}
\item $\sum_{a= j}^sm_a$ is even,
\item if $j\geq 2$, then either $\mu_{j-1}\geq \mu_{j}+2$ or $\epsilon_{j-1}= \epsilon_j$.
\end{enumerate}
We set
$l_\mu=l_\cO:=|\Omega_\cO|.
$ Suppose that $\Omega_\cO=\{j_1,\ldots,j_l\}$, $l=l_\cO$, and we write $j_{l+1}=s+1$. Note that $j_1=1$ if and only if $N$ is even. Thus $l_\cO\geq 1$ when $N$ is even. 
The subset $\Pi_\cO\subset\widehat{A}_K(\cO)$  is defined as follows
\beq
\label{char-bi}
\Pi_\cO=\{\chi\in \widehat{A}_K(\cO)\mid \chi(\delta_r)=1 \ \text{if}\  r+1\notin \Omega_\cO \}\,.
\eeq
In other words $\chi(\delta_r)$ is allowed to take the value $-1$ precisely when $r+1\in \Omega_\cO$.
In particular, $|\Pi_\cO|=2^{l_\cO}$ if $N$ is odd and $|\Pi_\cO|=2^{l_\cO-1}$ if $N$ is even as in the latter case $j_1=1$.

\begin{thm}
\label{nil thm-2}
Assume that $(G,K)=(SO_N,S(O_p\times O_q))$, where either $p$ or $q$ is even. The set of nilpotent support character sheaves  is
\beqn
\begin{gathered}
\on{Char}_K^\rn(\Lg_1)=\{\on{IC}(\cO,\cE_\phi)\,|\,\cO\in\underline{\cN_1^0},\ \phi\in\Pi_\cO\},
\end{gathered}
 \eeqn
where $\Pi_\cO$ is defined in~\eqref{char-bi}.
\end{thm}

The fact that the sheaves in Theorems~\ref{nil thm-1} and~\ref{nil thm-2} belong to $\on{Char}_K^\rn(\Lg_1)$ follows from the discussion in~\S\ref{sec-nilpotent support cs}, Propositions~\ref{sfiber-1} and~\ref{sfiber-2} below. The fact that these are all of $\on{Char}_K^\rn(\Lg_1)$ will follow once we show that we have constructed sufficiently many character sheaves.

\subsection{Conjugacy classes of $\theta$-stable Borel subgroups}

The $K$-conjugacy classes of $\theta$-stable Borel subgroups of $G$ are precisely the closed $K$-orbits on the flag variety $X=G/B$. Let $T$ be a maximal torus of $K$ which is also a maximal torus in $G$ as $\theta$ is inner. The closed $K$-orbits are given by $K$-orbits of the fixed point set $X^T$. The Weyl group $W(G,T)$ acts transitively on $X^T$ and in this manner we get an identification
\beqn
    \{\text{closed $K$-orbits on $X=G/B$}\} \longleftrightarrow W(G,T)/W(K,T)\,.
\eeqn
Recall our concrete description of  $(G,K)$ and the subspaces $V^+,\,V^-$ of $V$ (see \S\ref{ssec-classical pairs}). We denote by $V_i$ (resp. $V_i^+,\,V^-_i$) a subspace of $V$ (resp. $V^+,V^-$) of dimension $i$.
Let $n=[N/2]$ if $G=SO_N$. Let $\tilde{K}=O_p\times O_q$ in type D\rm{I} and $\tilde{K}=K$ otherwise.  
It is easy to see that
the $\tilde{K}$-conjugacy classes of $\theta$-stable Borel subgroups in $G$ are parametrized by {\em ordered} sequences $a_1,\ldots,a_n$, $a_i\in\{0,1\}$, $1\leq i\leq n$, such that
\bern
&&{\sum_{i=1}^na_i=q}\text{  (A\rm{III}) (C\rm{II})},\qquad 
{\sum_{i=1}^na_i=[\frac{q}{2}]}\text{ (BD\rm{I}).}
\eern
Let $B_{(a_i)}$ denote a $\theta$-stable Borel subgroup of $G$ corresponding to such a sequence $(a_i)$. Concretely $B_{(a_i)}$ can be taken as the subgroup of $G$ which stabilizes a flag of the following form
\begin{eqnarray*}
0\subset V_1\subset V_2\subset\cdots\subset V_n=V&&\text{ (A\rm{III})},\\
0\subset V_1\subset V_2\subset\cdots\subset V_n\subset V_n^\p\subset\cdots\subset V_1^\p\subset V&&\text{ others},
\end{eqnarray*}where
$
V_i=V_{i-\sum_{j=1}^ia_j}^+\oplus V_{\sum_{j=1}^ia_j}^-,\ i=1,\ldots,n$.

Let $\cO_{(a_i)}=\cO_{B_{(a_i)}}$. 
In the case of type D\rm{I}, each $\tilde{K}$-orbit of $\theta$-stable Borel subgroups in $G$ decomposes into two $K$-orbits; we denote them by $B_{(a_i)}^{\omega}$, $\omega=\rm{I,II}$. Note that $\cO_{B_{(a_i)}^{I}}=\cO_{B_{(a_i)}^{\rm{II}}}$.

\subsection{Springer fibers}
Let $B$ be a $\theta$-stable Borel subgroup of $G$. Recall the map $\pi_B:K\times^{B_K}\Ln_1\to\bar\cO_B$. Let $x\in\cO_B$. The fiber $\pi_B^{-1}(x)$ is $\cB_x\cap Q$, where $\cB_x\subset X$ is the Springer fiber and $Q$ is the $K$-orbit of $B$ in $X$. Moreover, it is a union of irreducible (actually connected) components of the Springer fiber $\cB_x$ and these components form a single orbit under $A_{K}(x)$ because $T^*_QX$ is $K$-connected.

The following proposition  is proved in \cite[Corollary 33]{BZ} for type C\rm{I}, and in \cite{T2} for types A\rm{III}, C\rm{II}, and D\rm{III}.

\begin{proposition}[\cite{BZ}, \cite{T2}]\label{sfiber-1}
Assume that $(G,K)$ is not of type {\rm{BD\rm{I}}}. Then the fibers $\pi_B^{-1}(x)$ are irreducible for $x\in\cO_B$.
\end{proposition}
As was pointed out earlier, it follows from this proposition and the discussion in~\S\ref{sec-nilpotent support cs} that the IC sheaves in Theorem~\ref{nil thm-1} are nilpotent support character sheaves.

In the remainder of this subsection we study the pairs $(G,K)=(SO_N,S(O_p\times O_q))$.  

Let $\cO=\cO_\mu\in\underline{\cN_1^0}$, where 
$\mu$ is as in~\eqref{the orbit mu}. Recall the set $\Omega_\cO=\Omega_\mu=\{j_1,\ldots,j_{l}\}$, $l=l_\cO$, defined at the beginning of the section.
Let 
\beqn
\text{$\sigma_k=\sum_{a=j_k}^{j_{k+1}-1}m_a$ for $1\leq k\leq l$.}
\eeqn 
The $\sigma_k$ are even. Let $x\in\cO$. Let $\{x,y,h\}$ be a normal $\mathfrak{sl}_2$-triple. Then $h\in\Lg_0$. In particular, $hV^+\subset V^+$ and $hV^-\subset V^-$. Let 
\beqn
V=\oplus_{i\geq 0} V(i)_+^{\oplus p_i}\oplus_{i\geq 0}V(i)_-^{\oplus q_i}
\eeqn
 be the decomposition of $V$ into irreducible $\mathfrak{sl}_2$-modules under $\{x,y,h\}$, where $V(i)_+$ (resp. $V(i)_-$) denotes an irreducible $\mathfrak{sl}_2$-module of highest weight $i$ such that the lowest weight vector lies in $V^+$ (resp. $V^-$). Let $V(i)=V(i)_+$ or $V(i)_-$. We set
\beqn
W_k=\bigoplus_{a=j_k}^{j_{k+1}-1}V(2\mu_a)_0^{\oplus m_a}, \ 1\leq k\leq l,
\eeqn
where $V(2\mu_a)_0$ is the $0$-weight space of $V(2\mu_a)$.
We have $\on{dim}W_k=\sigma_k$ and $(,)|_{W_k}$ is non-degenerate.

 For a subspace $W$ of $V$ let
\beqn
\on{IGr}(k,W)\ = \ \{V_k\subset W\mid \dim V_k =k,\,V_k \ \text{is isotropic in $V$}\}\,.
\eeqn

\begin{proposition}
\label{sfiber-2}
There exists a $\theta$-stable Borel subgroup $B\subset G$ such that $\cO=\cO_B$ and such that for  $x\in\cO$  we have
\beqn
\oh^{top}(\tilde{\pi}_B^{-1}(x)) \cong \bigotimes_{k=1}^l \oh^{top}(\on{IGr}(\sigma_k/2,W_k))\,,
\eeqn
where $\tilde{\pi}_B:\tilde{K}\times^{B_K}\Ln_1\to\bar\cO_B$. This isomorphism is compatible with the action of $A_K(x)$.
\end{proposition}

The group $A_K(x)$ from~\eqref{comp group explicit} acts in the natural way on  $\oh^{top}(\on{IGr}(\sigma_k/2,W_k))=\bC\oplus\bC$. The action of $A_K(x)$ on the set of irreducible (or equivalently, connected) components of $\tilde\pi_B^{-1}(x)$ is transitive when $N$ is odd but has two orbits when $N$ is even. In the $N$ even case we only consider one of the orbits. Taking into account the discussion from~\S\ref{sec-nilpotent support cs} and the explicit description of the action of $A_K(x)$ on $\oh^{top}(\on{IGr}(\sigma_k/2,W_k))$ we conclude that the IC sheaves in Theorem~\ref{nil thm-2} are nilpotent support character sheaves.

\begin{proof}[Proof of Proposition~\ref{sfiber-2}]  
Recall $n=[N/2]$ and the $\tilde K$-orbits of $\theta$-stable Borel subgroups are parametrized by ordered sequences $a_1,\ldots,a_{n}$, where $a_i\in\{0,1\}$ and $\sum_{1\leq i\leq n}a_i=[q/2]$. We define the special $\theta$-stable Borel by induction as follows. Assume that $1\leq k\leq s$ is such that $\epsilon_1=\ldots=\epsilon_k=\epsilon$ and $\epsilon_k\neq \epsilon_{k+1}$ (if $k<s$). Note that if $\mu_k=0$, then $k=s$. Let 
\beq
\label{the number mmu}
m_\mu= 
\sum_{a=1}^km_a \text{ (resp. $\sum_{a=1}^{k-1}m_a+[\frac{m_k}{2}]$)}\text{ if  $\mu_k\neq 0$ (resp. $\mu_k=0$)}.
\eeq
We define
\beqn
a_1=\ldots=a_{m_\mu}=\epsilon.
\eeqn
We write $m=m_\mu$ and without loss of generality we assume that $\epsilon=0$. Let 
\beqn
\text{$V_m=\ker x\cap\on{Im}x^{2\mu_k}$ if $\mu_k\geq 1$,} \text{ and $V_m\in\on{IGr}(m,\on{ker}x)$ if $\mu_k=0$.}
\eeqn
We note that $V_m\subset V^\epsilon$ (here $V^0=V^+$ and $V^1=V^-$). 
 Let $V'=(V_m)^\p/V_m$. Then $(,)$ induces a non-degenerate bilinear form $(,)'$ on $V'$. Moreover, $\theta$ induces an involution $\theta'$ on $SO_{V',(,)'}$ and we obtain a symmetric pair $(G',K')=(SO_{N-2m}, S(O_{p-2m}\times O_q))$.

Consider the map $x':V_m^\p/V_m\to V_m^\p/V_m$ induced by $x\in\cO_\mu$. Then $x'\in\cO'=\cO_{\mu'}$ where 
\beq\label{mmu'}
\mu'=
\begin{cases}
(2\mu_1-1)_{\epsilon+1}^{m_1}(2\mu_2-1)_{\epsilon+1}^{m_2}\cdots(2\mu_k-1)^{m_k}_{\epsilon+1}(2\mu_{k+1}+1)^{m_{k+1}}_{\epsilon+1}\cdots &\text{ if $\mu_k\geq 1$}
\\
(2\mu_1-1)_{\epsilon+1}^{m_1}(2\mu_2-1)_{\epsilon+1}^{m_2}\cdots(2\mu_{k-1}-1)^{m_{k-1}}_{\epsilon+1}1^{m_k-2[m_k/2]}_{\epsilon}&\text{ if $\mu_k=0$}.
\end{cases}
\eeq
Note that in passing from $\mu$ to $\mu'$ only the first $\sum_{i=1}^km_i$ terms in the partition change. 
Let $m_\mu'$ be defined  for $\mu'$ as in~\eqref{the number mmu}. We define 
\beqn
a_{m_\mu+1}=\ldots=a_{m_\mu+m_{\mu'}}\equiv \epsilon+1.
\eeqn
Continuing in this manner we produce a  sequence $a_1,\ldots,a_n$. 

 Let $B'$ be a   $\theta$-stable Borel subgroup defined by the sequence $a_{m_\mu+1},\ldots,a_n$ for the symmetric pair $(G',K')=(SO_{N-2m}, S(O_{p-2m}\times O_q))$.
Let us write $\tilde\pi_B:\tilde K\times^{B_{K}}\Ln_1\to\bar\cO_{B}$ and $\tilde\pi'_{B'}:\tilde K'\times^{B'_{K'}}\Ln'_1\to\bar\cO_{B'}$. By the induction hypothesis, we can assume that $\cO_{B'}=\cO_{\mu'}$, where $\mu'$ is as in~\eqref{mmu'}. Then $\cO_{B}=\cO_\mu$ by \cite[Proposition 7.1]{T}. We write $B$ for the $\theta$-stable Borel subgroup defined by the sequence $a_1,\ldots,a_n$. 

We now determine the components of $\tilde\pi_B^{-1}(x)$ by induction. 
We first show that 
\beq\label{omegas}
\Omega_{\cO}=\Omega_{\cO'}\sqcup\{k\}\text{ (resp. $\Omega_{\cO'}$)}\text{ if $\mu_k=0$ and $m_k$ is even (resp. otherwise)}.
\eeq
Thus
\beq
\label{numbers1}
l_{\cO'}=l_{\cO}-1\text{ (resp. $l_{\cO'}$)  if $\mu_k=0$ and $m_k$ is even (resp. otherwise)}.
\eeq
Since $\epsilon_{k-1}=\epsilon_k$, whether $k\in \Omega_\cO$ depends only on the fact whether $\sum_{a= k}^sm_a$ is even.
Consider first the case $\mu_k=0$. Then $k=s$ and, if $m_k$ is even (resp. odd) then $k\in\Omega_\cO$ (resp. $k\notin\Omega_\cO$) and $k\notin\Omega_{\cO'}$ . Thus we are reduced to the case $\mu_k\geq 1$ and
it suffices to check that $k+1\in\Omega_{\cO}$ if and only if $k+1\in\Omega_{\cO}'$. We have the following cases
\begin{enumerate}
\item $\sum_{a= k+1}^sm_a$ is odd, then $k+1\notin\Omega_{\cO}$ and $k+1\notin\Omega_{\cO'}$;
\item $\sum_{a= k+1}^sm_a$  is even and $\mu_k\geq\mu_{k+1}+2$. Thus $k+1\in\Omega_{\cO}$ and $k+1\in\Omega_{\cO'}$ because
at step $k+1$ there is no sign change.
\item $\sum_{a= k+1}^sm_a$  is even and $\mu_k=\mu_{k+1}+1$. Then, as $\epsilon\neq\epsilon_{k+1}$ we see that  $k+1\notin\Omega_{\cO}$. Also, $k+1\notin\Omega_{\cO'}$ as in this case  $\Omega_{\cO'}\subset \{1,\dots, s\}-\{k+1\}$. 
\end{enumerate}
This proves~\eqref{omegas} and thus~\eqref{numbers1}.

Let us write $l=l_\cO$, $l'=l_{\cO'}$ and let 
\beqn
W_b'=\bigoplus_{a=j_b}^{j_{b+1}-1}V(2\mu_a')_0^{\oplus m_a}, \ 1\leq b\leq l',
\eeqn
where, we recall,  $\mu_a'=\mu_a-1$ for $1\leq a\leq k$. 
It is now easy to verify that
\beq\label{dim of sf}
\dim\tilde\pi_B^{-1}({x})=\dim{(\tilde\pi'_{B'})}^{-1}({x'})+\frac{m^2-m}{2}+\left\{\begin{array}{ll}0&\text{ if }\mu_k\geq 1,\\
\frac{m_k(m_k-2)}{8}&\text{ if }\mu_k=0\text{ and }m_k\text{ even},\\
\frac{m_k^2-1}{8}&\text{ if }\mu_k=0\text{ and }m_k\text{ odd}.\end{array}\right.
\eeq
Note that if $\cF=(U_i)\in\tilde\pi_B^{-1}(x)$, then $U_m\subset \ker x$ as $U_m\subset V^+$. Consider the map
\beqn
p:\tilde\pi_B^{-1}(x)\to \on{IGr}(m,\on{ker}x),\ \cF=(U_i)\mapsto U_m.
\eeqn
Suppose that $\mu_k\geq 1$.  By construction  $p^{-1}(V_m)\cong \cF(V_m)\times {(\tilde\pi'_{B'})}^{-1}(x')$, where $\cF(V_m)$ is the complete flag variety of $V_m=\ker x\cap\on{Im}x^{2\mu_k}$. By~\eqref{dim of sf}, we then conclude that $p^{-1}(V_m)=\tilde\pi_B^{-1}(x)$ and hence $\tilde\pi_B^{-1}(x)\cong \cF(V_m)\times {(\tilde\pi'_{B'})}^{-1}(x')$. Therefore
 \beqn
\oh^{top}(
\tilde\pi_B^{-1}(x)) \cong  \oh^{top}((\tilde\pi'_{B'})^{-1}(x'))\,.
\eeqn
Let us now consider the case when $\mu_k=0$. We first note that 
\beqn
\on{IGr}(m,\ker x)\cong\on{IGr}([m_k/2],V(0)_0^{\oplus m_k})\,,
\eeqn
where $\on{dim}V(0)_0^{\oplus m_k}=m_k$ and $(\,,\,)|_{V(0)_0^{\oplus m_k}}$ is non-degenerate. 
 For each $U_m\in\on{IGr}(m,\ker x)$, the fiber $p^{-1}(U_m)\cong\cF(U_m)\times(\tilde\pi'_{B'})^{-1}(x')$.  Therefore we conclude that 
 \beqn
\oh^{top}(\tilde\pi_B^{-1}(x)) \cong \oh^{top}(\on{IGr}([m_k/2],m_k))\otimes \oh^{top}((\tilde\pi'_{B'})^{-1}(x'))\,.
\eeqn
We have 
\beqn
\oh^{top}(\on{IGr}([m_k/2],m_k))=\bC\text{ (resp. $\bC\oplus\bC$) if $m_k$ is odd (resp. even)}.
\eeqn
We note that $\on{IGr}(\sigma_b/2,W_b)\cong\on{IGr}(\sigma_{b'}/2,W_b')$, $1\leq b\leq l'$ and we can assume by induction that the theorem holds for  $\oh^{top}((\tilde\pi'_{B'})^{-1}(x'))$. The discussion above implies that $\oh^{top}(\tilde\pi_B^{-1}(x)) \cong  \oh^{top}((\tilde\pi'_{B'})^{-1}(x'))$ except when $\mu_k=0$ and $m_k$ is even. In that case $\sigma_l=m_k$ and $W_l=V(0)_0^{\oplus m_k}$, and thus
 \beqn
\oh^{top}(\tilde\pi_B^{-1}(x)) \cong \oh^{top}(\on{IGr}(\sigma_l/2,W_l))\otimes \oh^{top}((\tilde\pi'_{B'})^{-1}(x'))\,.
\eeqn
This concludes the induction step.
\end{proof}

\section{Nearby cycle sheaves}\label{sec-full support}

In this section, we determine the Fourier transform $\fF(P_\chi)$ of the nearby cycle sheaves $P_\chi$ (see~\S\ref{csfs}) explicitly  for the symmetric pairs $(G,K)$ in \S\ref{ssec-classical pairs}. Recall that $r=\on{dim}\fa$.

\subsection{The group $I$ and the Weyl group action}\label{sec-action}

In this subsection we study the groups $I=Z_K(\fa)/Z_K(\fa)^0$ and the action of $W_\fa$ on the set $\hat I$ of its irreducible characters.  We will make use of the discussion in \S\ref{ssec-classical pairs} and \S\ref{ssec-dual side}.
 Note that $I^0=Z_{K^0}(\fa)/Z_{K^0}(\fa)^0$ is generated by $\check{\alpha}(-1)$, $\alpha\in\Phi^{\rR}$, see, for example~\cite[Theorem 7.55]{Ka}.
\begin{lemma}
\label{lemma-orbits on I}
 We have that
\begin{enumerate}
\item
$I=1,\text{ if }(G,K)$ is of type \rm{AIII}, \rm{CII} or \rm{DIII}.
\item $I=(\bZ/2\bZ)^{r},\ I^0=(\bZ/2\bZ)^{r-1}$, if $(G,K)$ is of type {\rm BDI}, and $I=I^0=(\bZ/2\bZ)^r,$ if $(G,K)$ is of type {\rm CI}. Moreover, the action of $W_\fa$ on $\hat I$ has $r+1$ orbits.
\end{enumerate}
\end{lemma}

\begin{proof}  
The claim (1) can be checked easily. Note that $I/I^0=K/K^0$. Thus if $G$ is simply connected, then $K=K^0$ and  $I=I^0$.  The claim (2) is readily checked using the following description of $I$ and $I^0$. 

Suppose that $(G,K)$ is of type BDI. Then $K/K^0=\bZ/2\bZ$. Let $\alpha_i=\epsilon_i-\epsilon_{i+1}$, $1\leq i\leq r-1$, $\alpha_r=\epsilon_r$ ($N$ odd), $\alpha_r=\epsilon_{r-1}+\epsilon_r$ ($N$ even), be a set of simple real roots. Then  $I^0=\langle\gamma_i=\check{\alpha}_i(-1), 1\leq i\leq r-1\rangle$ and  $I=\langle\gamma_i=\check{\alpha}_i(-1), 1\leq i\leq r-1,\gamma_r=\epsilon_r(-1)\rangle$.

Suppose that $(G,K)$ is of type {\rm CI}. Let  $\alpha_i=\epsilon_{i}-\epsilon_{i+1}$, $1\leq i\leq r-1$,\,$\alpha_{r}=2\epsilon_r$, be a set of simple real roots. Then $I=\langle\gamma_i=\check{\alpha}_i(-1),\,1\leq i\leq r\rangle$. \end{proof}

In the remainder of this subsection we assume that $(G,K)$ is of type BDI or CI. 
We fix a set of representatives for the $W_\fa$-orbits in $\hat I$ as follows,
\beq\label{chars}
\chi_0=1,\ \text{$\chi_m(\gamma_i)=1$ if $i\neq m$ and $\chi_m(\gamma_m)=-1$, $1\leq m\leq r$},
\eeq
where $\gamma_i\in I$ are defined in the proof of Lemma~\ref{lemma-orbits on I}. We describe the subgroups $W_{\fa,\chi_m}$ and $W_{\fa,\chi_m}^0$.

Suppose that $(G,K)$ is of type BI. Then $G_\rs$ is semisimple and adjoint. One checks that
\beqn
\begin{gathered}
\{\alpha\in\Phi^{\rR}\,|\,\chi_m(\check{\alpha}(-1))=1\}=\{\pm(\epsilon_i\pm\epsilon_j),1\leq i<j\leq m,\, m<i<j\leq r;\ \pm\epsilon_i,1\leq i\leq r\}.
\end{gathered}
\eeqn
Thus (see~\eqref{adjoint case})
\beqn
\begin{gathered}
W_{\fa,\chi_m}=W_{\fa,\chi_m}^0
=\langle s_1,\ldots,s_{m-1},\tau_m\rangle\times\langle s_{m+1},\ldots,s_r\rangle\cong W_m\times W_{r-m},
\end{gathered}
\eeqn 
where $s_i=s_{\epsilon_i-\epsilon_{i-1}}$, $1\leq i\leq m-1$, $s_r=s_{\epsilon_r}$ and 
$
\tau_m=s_ms_{m+1}\ldots s_{r-1}s_rs_{r-1}\ldots s_m=s_{\epsilon_m}.
$

Suppose that $(G,K)$ is of type DI. Let $s_i=s_{\epsilon_i-\epsilon_{i-1}}$, $1\leq i\leq r-1$, $s_r=s_{\epsilon_{r-1}+\epsilon_{r}}$, and $\tau_{m}=(s_ms_{m+1}\ldots s_{r-2}s_rs_{r-1}\ldots s_{m+1}s_m)=s_{\epsilon_m}s_{\epsilon_n}$. Assume that $r=n$. Then 
\beq\label{stabilizer split D}
\begin{gathered}
{W}_{\fa,\chi_m}=\langle s_{m+1},\ldots,s_n\rangle\rtimes\langle s_1,\ldots,s_{m-1},\tau_{m}\rangle\cong W_{n-m}'\rtimes W_m,\ 1\leq m\leq n-1
\\
W_{\fa,\chi_m}^0=\langle s_1,\ldots,s_{m-1},\tau_ms_{m-1}\tau_m\rangle\times\langle s_{m+1},\ldots,s_n\rangle\cong W_m'\times W_{n-m}',\ 1\leq m\leq n-1,\\
W_{\fa,\chi_m}=W_{\fa,\chi_m}^0=W_n', \ m=0\text{ or }n.
\end{gathered}
\eeq
Assume that  $r< n$.  We have $G_\rs\cong SO(2r)$. Using~\eqref{reduction to Gs} and~\eqref{stabilizer split D}, we see that
\beqn
W^\rR_{\chi_m}=\langle s_1,\ldots,s_{m-1},\tau_m,s_{m+1},\ldots,s_r\rangle.
\eeqn
Write $(W^\rC)^\theta=\{1,t\}=\bZ/2\bZ$, where $t=s_{\epsilon_r}s_{\epsilon_n}$. Then 
$
W_\fa=\langle s_1,\ldots,s_{r-1},t\rangle=W_r.
$
Thus
\beqn
\begin{gathered}
W_{\fa,\chi_m}=W^\rR_{\chi_m}\rtimes (W^\rC)^\theta=\langle s_1,\ldots,s_{m-1},\tau_m,s_{m+1},\ldots,s_r, t\rangle\\
=\langle s_1,\ldots,s_{m-1},s_ms_{m+1}\cdots s_{r-2}s_{r-1}ts_{r-1}\cdots s_{m+1}s_m\rangle\times\langle s_{m+1},\ldots,s_{r-1},t\rangle\cong W_m\times W_{r-m}.
\end{gathered}
\eeqn
On the other hand, one checks that 
$
\{\alpha\in\Phi^{\rR}\,|\,\chi_m(\check{\alpha}(-1))=1\}=\{\pm(\epsilon_i\pm\epsilon_j),1\leq i<j\leq m,\, m+1\leq i<j\leq r\}.
$
Thus 
\beqn
W_{\fa,\chi_m}^0=\langle s_{1},\ldots, s_{m-1},s_{m+1},\ldots,s_{r-1},\, s_{\epsilon_i}, i=1,\ldots,r\rangle=W_{\fa,\chi_m}\cong W_m\times W_{r-m}.
\eeqn

Suppose that $(G,K)$ is of type CI. Let $s_i=s_{\epsilon_i-\epsilon_{i+1}}$, $1\leq i\leq r-1$, $s_r=s_{\epsilon_r}$, and $\tau_m=s_ms_{m+1}\ldots s_{r-1}s_rs_{r-1}\ldots s_m=s_{\epsilon_m}$ . 
We have
\beq\label{eqn-stabiliser-type C}
\begin{gathered}
W_{\fa,\chi_m}=\langle s_1,\ldots,s_{m-1},\tau_m\rangle\times\langle s_{m+1},\ldots,s_r\rangle\cong W_m\times W_{r-m},
\\
W_{\fa,\chi_m}^0=\langle s_1,\ldots,s_{m-1},\tau_ms_{m-1}\tau_m=s_{\epsilon_{m-1}+\epsilon_m}\rangle\times\langle s_{m+1},\ldots,s_r\rangle\cong W_m'\times W_{r-m}.
\end{gathered}
\eeq

\subsection{Nearby cycle sheaves and full support character sheaves} Using the discussions in~\S\ref{ssec-classical pairs},~\S\ref{csfs},  and~\S\ref{sec-action}, we explicitly describe the Fourier transforms of the nearby cycle sheaves. 

Let $\chi_0$ denote the trivial character of $I$. In types BD\rm{I} and C\rm{I}, let $\chi_m\in \hat I$, $0\leq m\leq r$, be the characters defined in~\eqref{chars}. Let
 \beq\label{tildebs}
 \widetilde{B}_{W_k}=(\bZ/2\bZ)^{k}\rtimes B_{W_k},\ \widetilde{B}_{W_k'}=(\bZ/2\bZ)^{k}\rtimes B_{W_k'}.
 \eeq
Let $\cH_{W_r,c_0,c_1}$ (resp. $\cH_{W_n',-1}$) be the Hecke algebra defined in~\S\ref{sec-Hecke}. We write $\cH_{W_r,c_0,c_1}$ also for the representation of $B_{W_r}$ (resp. $B_{W_n'}$)  arising from  the regular representation of $\cH_{W_r,c_0,c_1}$ (resp. $\cH_{W_n',-1}$) via the natural surjective map $\bC[B_{W_r}]\to\cH_{W_r,c_0,c_1}$ (resp. $\bC[B_{W_n'}]\to\cH_{W_n',-1}$). We set $\cH_{W_0,c_0,c_1}=\cH_{W_0',-1}=\bC$.

\begin{prop}\label{thm-nearby cycles}
{\rm (i)} Suppose that $(G,K)$ is of type {\rm{A{III}}, \rm{CII} or \rm{DIII}}.
We have 
\beqn
\fF(P_{\chi_0})=\on{IC}(\Lg_1^{rs},\cH_{W_r,1,-1}).
\eeqn
\noindent{\rm (ii)} Suppose that $(G,K)$ is of type {\rm BCDI}. We have
\beqn
\fF (P_{\chi_m})=\on{IC}(\Lg_1^{rs},\cM_{\chi_m}\otimes\bC_\iota),
\eeqn
where $\cM_{\chi_m}$ is given by the following representation $M_{\chi_m}$ of  $\pi_1^K(\Lg_1^{rs})$
\bern
&&{M_{\chi_m}\cong\bC[\widetilde{B}_{W_r}]\otimes_{\bC[\widetilde{B}_{W_r}^{\chi_m}]}(\bC_{\chi_m}\otimes\cH_{W_m,-1,-1}\otimes\cH_{W_{r-m},-1,-1})}\quad\text{\rm{(BI)}},\\
&&M_{\chi_m}\cong\bC[\widetilde{B}_{W_n}]\otimes_{\bC[\widetilde{B}_{W_n}^{\chi_m}]}(\bC_{\chi_m}\otimes\cH_{W_m,-1,1}\otimes\cH_{W_{n-m},-1,-1})\quad\text{\rm{(CI)}},\\
&&{M_{\chi_m}\cong\bC[\widetilde{B}_{W_n'}]\otimes_{\bC[\widetilde{B}_{W_n'}^{\chi_m,0}]}(\bC_{\chi_m}\otimes\cH_{W_m',-1}\otimes\cH_{W_{n-m}',-1})}\quad\text{\rm{(DI)} $p=q$},\\
&&{M_{\chi_m}\cong
\bC[\widetilde{B}_{W_r}]\otimes_{\bC[\widetilde{B}_{W_r}^{\chi_m}]}(\bC_{\chi_m}\otimes\cH_{W_m,-1,1}\otimes\cH_{W_{r-m},-1,1})}\quad\text{\rm{(DI)} $p\neq q$},
\eern
and $\iota=1$ unless $G=SO_{2n+1}$ in which case $\iota$ is the nontrivial character of $K/K^0=\bZ/2\bZ$. 
\end{prop}

In the case of type \rm{CI}, the  isomorphism in the above proposition is proved in the same way as in~\cite[\S7.6]{GVX}, noting that $W_{\fa,\chi_m}$ is a Coxeter group. In particular, the parameter for the simple reflection $\tau_m$ of $W_{\fa,\chi_m}$ in~\eqref{eqn-stabiliser-type C} is $1$ since $\chi(\check\alpha(-1))=-1$ for $\alpha=\epsilon_k$.

In the case of split type DI, recall from~\eqref{stabilizer split D} that when $1\leq m\leq n-1$, $W_{\fa,\chi_m}=W_{\fa,\chi_m}^0\rtimes\langle\tau_m\rangle$ and $W_{\fa,\chi_m}^0=\langle s_1,\ldots,s_{m-1},\tau_ms_{m-1}\tau_m\rangle\times\langle s_{m+1},\ldots,s_n\rangle$, where $\tau_m^2=1$, and $\tau_m$ acts on the generators of $W_{\fa,\chi_m}^0$ as follows: $\tau_m$ fixes $s_i$ for $i< m-1$ and $m+1\leq i<n-1$, $\tau_m$ interchanges $\tau_ms_{m-1}\tau_m$ (resp. $s_{n-1}$) with $s_{m-1}$ (resp. $s_n$).  Thus when $1\leq m\leq n-1$, for each irreducible representation $\rho$ of $B_{W_n'}^{\chi_m,0}$, the $\bC[B_{W_n'}^{\chi_m}]\otimes_{\bC[B_{W_n'}^{\chi_m,0}]}\rho$ decomposes into two non-isomorphic irreducible representations of $B_{W_n'}^{\chi_m}$.

Recall the set $\Theta_{(G,K)}
$ defined in~\eqref{loc-full}. 
Applying  Proposition~\ref{thm-nearby cycles}, the discussions in~\S\ref{csfs} and the above discussion, we conclude that the sheaves in the following corollary are full support character sheaves. As before, the fact that these are all of them will follow once we construct all character sheaves. 

\begin{corollary}\label{coro-full}
For the symmetric pairs considered in this paper all full support character sheaves arise from the nearby cycle constructions, i.e.,
$$\left\{\on{IC}(\Lg_1^{rs},\cL_\rho)\,|\,\rho\in\Theta_{(G,K)}\right\}=\on{Char}_K^{\mathrm{f}}(\Lg_1)\,.$$ Moreover, we have
\bern
&&\Theta_{(G,K)}=\{V_\tau\mid\tau\in\on{Irr}\cH_{W_r,1,-1}\}\text{ {\rm (AIII) (CII) (DIII)}}\\
&&\Theta_{(SO_{2n+1},S(O_p\times O_q))}=\{V_{\rho_1,\rho_2,\chi_k}\mid\rho_1\in\on{Irr}\cH_{W_k,-1,-1},\rho_2\in\on{Irr}\cH_{W_{r-k},-1,-1},k\in[0,r]\}\\
&&\Theta_{(Sp_{2n},GL_n)}=\{V_{\rho_1,\rho_2,\chi_k}\mid\rho_1\in\on{Irr}\cH_{W_k,-1,1},\rho_2\in\on{Irr}\cH_{W_{n-k},-1,-1},k\in[0,n]\}\\
&&\Theta_{(SO_{2n},S(O_n\times O_n))}=\{V_{\rho_1,\rho_2,\chi_k}^\delta\mid\delta={\rm I,II},\,\rho_1\in\on{Irr}\cH_{W_k',-1},\rho_2\in\on{Irr}\cH_{W_{n-k}',-1},k\in[0,n]\}\\
&&\Theta_{(SO_{2n},S(O_p\times O_q))}=\{V_{\rho_1,\rho_2,\chi_k}\mid\rho_1\in\on{Irr}\cH_{W_k,-1,1},\rho_2\in\on{Irr}\cH_{W_{r-k},-1,1},k\in[0,r]\}\ (p\neq q).
\eern
\end{corollary}
Let us write
\ber\label{rep1}
&&\Theta_n^{B}=\Theta_{(SO_{2n+2t-1},S(O_{n+2t-1}\times O_{n}))},\, \Theta_n^{D}=\Theta_{(SO_{2n+2t},S(O_{n+2t}\times O_{n}))},\,\Theta_n^{C}=\Theta_{(Sp_{2n},GL_n)},\\\
&&\label{rep2}\Theta_n^{D,0}=\Theta_{(SO_{2n},S(O_{n}\times O_{n}))},
\eer
where $t>0$. 
It follows that 
\begin{eqnarray}
 \label{number of fullsupp}
&& |\Theta_{n}^B| =\sum_{k=0}^qd(k)d(n-k),
\ \ |\Theta_{n}^C| =\sum_{k=0}^nd(k)e(n-k),\\
 \label{number of fullsupp D}
&&|\Theta_{n}^{D}|=\sum_{k=0}^ne(k)e(n-k),
\ \ |\Theta_{n}^{D,0}|=\frac{1}{2}\sum_{k=0}^n e(k)e(n-k),
 \end{eqnarray}
where we have used~\eqref{Hecke-D} and 
$
d(k)=|\on{Irr}\cH_{W_k,-1,-1}|,\  e(k)=|\on{Irr}\cH_{W_k,-1,1}|$, see~\S\ref{sec-Hecke}.

 \section{Character sheaves}\label{sec-main theorems}
In this section we give a description of the set $\on{Char}_K(\Lg_1)$ of character sheaves  for the symmetric pairs $(G,K)$ in \S\ref{ssec-classical pairs}.  We use the notation from that section. 
\subsection{Supports of the character sheaves}
\label{sec-support} 
To describe the supports of character sheaves, we define a set $\underline{\cN_1^{\text{cs}}}$ of nilpotent orbits such that for $\cO\in\underline{\cN_1^{\text{cs}}}$ the corresponding $\widecheck\cO$ supports a character sheaf. Conversely, all supports of character sheaves are of this form.  The $\underline{\cN_1^{\text{cs}}}$ consists of the following orbits:
\begin{eqnarray*}
 \text{\rm{(AIII)}}&\ \cO_{k,\mu}=\cO_{1^k_+1^k_-\sqcup\mu}&\ \ 0\leq k\leq r,\ \mu\in \on{SYD}^0_{(GL_{n-2k},GL_{p-k}\times GL_{q-k})},\\
 \text{\rm{(BDI)}}&\ \cO_{m,k,\mu}=\cO_{1^m_+1^m_-2^k_+2^k_-\sqcup\mu}&\ \  m\equiv r\,\nmod 2\text{ if $N$ is even},\ 0\leq m+2k\leq r, 
 \\&&\mu\in \on{SYD}^0_{(SO_{N-2m-4k},S(O_{p-m-2k}\times O_{q-m-2k}))}, \\
  \text{\rm{(CI)}}&\ \cO_{m,k,\mu}=\cO_{1^m_+1^m_-2^k_+2^k_-\sqcup\mu}&\ \ 0\leq m+2k\leq n,
 \ \mu\in \on{SYD}^0_{(Sp_{2n-2m-4k},GL_{n-m-2k})} \\
  \text{\rm{(CII)}}&\ \cO_{k,\mu}=\cO_{1^{2k}_+1^{2k}_-\sqcup\mu}&\ \ 0\leq k\leq r,\ \mu\in\on{SYD}^0_{(Sp_{2n-4k},Sp_{2p-2k}\times Sp_{2q-2k})}\\
 \text{\rm{(DIII)}}&\ \cO_{k,\mu}=\cO_{1^{2k}_+1^{2k}_-\sqcup\mu}&\ \ 0\leq k\leq [n/2],\ \mu\in\on{SYD}^0_{(SO_{2n-4k},GL_{n-2k})}.
 \end{eqnarray*}
Here $\lambda\sqcup\mu$ denotes the signed Young diagram obtained by joining $\lambda$ and $\mu$ together, i.e., the rows of $\lambda\sqcup\mu$ are the rows of $\lambda$ and $\mu$ rearranged according to the lengths of the rows.  Note that $\underline{\cN_1^{0}}\subset\underline{\cN_1^{\text{cs}}}$.  When $G=K$, the set $\on{SYD}^{0}_{(G,K)}$ consists of the signed Young diagram which corresponds to the zero orbit.  Also note that when $(G,K)=(SO_{4k},S(O_{2k}\times O_{2k}))$, we have written $\cO_{0,k,\emptyset}$ for both $\cO_{0,k,\emptyset}^{\rm{I}}$ and $\cO_{0,k,\emptyset}^{\rm{II}}$.
 
 \begin{remark}
For the symmetric pair $(SL_N,SO_N)$ considered in~\cite{CVX},  
the $\widecheck\cO$'s that support character sheave are given by the nilpotent orbits $\cO_{2^k1^{N-2k}}$, where as usual, there are two orbits $\cO_{2^{N/2}}^{\rm I}$ and $\cO_{2^{N/2}}^{\rm II}$ when $N$ is even.
 \end{remark}

Using~\eqref{identification}, one readily checks that the equvariant fundamental groups $\pi_1^K(\widecheck\cO)$ are as follows
\begin{subequations}
\label{explicit fundamental}
 \begin{eqnarray}
 \label{eqpi1-o1}
\text{\rm{(AIII)}\,(C\rm{II})\,(D\rm{III})}\ &&\pi_1^K(\widecheck\cO_{k,\mu})=B_{W_k};
\\
   \label{eqpi1-o2}
 \text{(BD\rm{I})}\ &&\pi_1^K(\widecheck\cO_{m,k,\mu})=\begin{cases} \widetilde{B}_{W_m'}\times \widetilde{B}_{W_k}\text{ if }\mu=\emptyset\\
\widetilde{B}_{W_m}\times \widetilde{B}_{W_k}\times(\bZ/2\bZ)^{r_\mu}\text{ if $\mu\neq \emptyset$}\end{cases}
\\ \label{eqpi1-o3}
 \text{(C\rm{I})}\ &&\pi_1^K(\widecheck\cO_{m,k,\mu})=\widetilde{B}_{W_m}\times \widetilde{B}_{W_k}\times(\bZ/2\bZ)^{r_\mu};
 \end{eqnarray}
 \end{subequations}
 where $r_\mu$ are defined in~\eqref{component group-BD} and~\eqref{component group-C}. Here we use the convention that $B_{W_0}=B_{W_0'}=\widetilde{B}_{W_0}=\widetilde{B}_{W_0'}=\{1\}$.
 
To illustrate the idea, let us explain the calculation in the case of type BDI. Also, for later use, the argument below fixes the splitting of $\pi_1^K(\widecheck\cO_{m,k,\mu})$ into a product.  We assume that $\mu=(\mu_1)^{m_1}\cdots(\mu_s)^{m_s}1^{m_0}$, where  each $\mu_i$ is odd and $\mu_1>\cdots>\mu_s>1$. We then have 
\bern
&& G^{\phi}=\{(g_1,\ldots,g_s,g_0,h)\in \prod_{i=1}^sO_{m_i}\times O_{2m+m_0}\times Sp_{2k}\,|\,\det(g_1\cdots g_sg_0)=1\}
\\
&&K^{\phi}=\{(g_1,\ldots,g_s,h_1,h_2,h_0)\in \prod_{i=1}^sO_{m_i}\times O_m\times O_{m+m_0}\times GL_k\,|\,\det(g_1\cdots g_sh_1h_2)=1\}\,.
\eern
Assume that $\mu\neq\emptyset$, i.e., that $s$ and $m_0$ are not both zero. Then $Z_{K^\phi}(\fa^\phi)/Z_{K^\phi}(\fa^\phi)^0=I_1\times I_2\times I_3$, $I_1=(\bZ/2\bZ)^m$, $I_2=(\bZ/2\bZ)^k$, $I_3=(\bZ/2\bZ)^{s-\delta_{m_0,0}}=(\bZ/2\bZ)^{r_\mu}\cong S(\prod_{i=1}^sO_{m_i}\times O_{m_0})/S(\prod_{i=1}^sO_{m_i}\times O_{m_0})^0$ and $B_{W_{\fa^\phi}}=B_{W_m}\times B_{W_k}$. Moreover the action of $B_{W_m}\times B_{W_k}$ on $I_1\times I_2\times I_3$ factors through the action of $B_{W_m}$ on $I_1$ and the action of $B_{W_k}$ on $I_2$.

\subsection{Explicit description of character sheaves}
In this subsection we state our main theorems which give an explicit description of character sheaves. We begin by writing down representations of the fundamental groups in~\eqref{explicit fundamental}. We use the notation from~\S\ref{sec-Hecke}. Let  $\tau\in\cP(k)$. We continue to write $L_\tau$ for the $B_{W_k}$-representation obtained by  pulling back the  simple module $L_\tau$ of $\cH_{W_k,1,-1}$ (see~\eqref{heck rep-sn like})  via the surjective map $\bC[B_{W_k}]\to \cH_{W_k,1,-1}$. Recall also the sets $\Theta_n^{B,C,D}$ (resp. $\Theta_n^{D,0}$) of irreducible representations of $\widetilde B_{W_n}$ (resp. $\widetilde B_{W_n'}$) defined in~\eqref{rep1} (resp.~\eqref{rep2}).

{\rm(A\rm{III},  C\rm{II}, D\rm{III})} For each $\tau\in\cP(k)$, let $\cT_\tau$ denote the irreducible $K$-equivariant local system on $\widecheck\cO_{k,\mu}$ corresponding to the $\pi_1^K(\widecheck\cO_{k,\mu})$-representation where $B_{W_k}$ acts via $L_\tau$, i.e.
$
\cT_\tau=L_\tau.
$
Here $\cT_\tau=\bC$ is the trivial local system if $k=0$.

{\rm (C\rm{I})} For  $\rho\in\Theta_{m}^C$ and $\tau\in\cP(k)$, let  $\cT_{\rho,\tau}$ denote the irreducible $K$-equivariant local system on $\widecheck\cO_{m,k,\mu}$ corresponding to the representation of $\pi_1^k(\widecheck\cO_{m,k,\mu})$, where $\widetilde{B}_{W_m}$ acts via $\rho$, $\widetilde{B}_{W_k}$ acts via the $B_{W_k}$-representation $L_\tau$, and $(\bZ/2\bZ)^{r_\mu}$ acts trivially, i.e., 
$
\cT_{\rho,\tau} \ = \ L_\rho\boxtimes L_\tau \boxtimes 1.
$

{\rm (B\rm{I} (resp. DI))} Suppose that $\mu\neq \emptyset$. For each $\rho\in\Theta_{m}^B$ (resp. $\Theta_{m}^D$), $\tau\in\cP(k)$, and $\phi\in\Pi_{\cO_\mu}$ (see~\eqref{char-bi}), let $\calT_{\rho,\tau,\phi}$ denote the irreducible $K$-equivariant local system on $\widecheck\cO_{m,k,\mu}$ given by the representation of $\pi_1^K(\widecheck\cO_{m,k,\mu})$, where $\widetilde{B}_{W_m}$ acts via $\rho$, $\widetilde{B}_{W_k}$ acts via the $B_{W_k}$-representation $L_\tau$, and $(\bZ/2\bZ)^{r_\mu}$ acts via $\phi$, i.e.,
$
\cT_{\rho,\tau,\phi} \ = \ L_\rho\boxtimes L_\tau \boxtimes \phi.
$

Suppose that $\mu=\emptyset$. For each $\rho\in\Theta_{m}^{D,0}$ and $\tau\in\cP(k)$, let $\calT_{\rho,\tau}$ denote the irreducible $K$-equivariant local system on $\widecheck\cO_{m,k,\emptyset}$ corresponding to the representation $L_{\rho}\boxtimes L_\tau$ of $\pi_1^K(\widecheck\cO_{m,k,\emptyset})$, where $\widetilde{B}_{W_m'}$ acts via $\rho$ and $\widetilde{B}_{W_k}$ acts via the $B_{W_k}$-representation $L_\tau$, i.e.,
$
\calT_{\rho,\tau}=L_{\rho}\boxtimes L_\tau.
$

Note that $\cT_{\rho,\tau}=\bC$ is the trivial local system and $\cT_{\rho,\tau,\phi}=\cE_\phi$ if $m=k=0$.

 \begin{thm}\label{thm-type A}
 Suppose that $(G,K)$ is of type {\rm{AIII,\,CII}}, or {\rm DIII}. 
 We have
 \beqn
 \on{Char}_K(\Lg_1)=
 \left\{\on{IC}(\widecheck\cO_{k,\mu},\cT_{\tau})\,|\,\tau\in\cP(k)\right\}.
\eeqn
 \end{thm}

\begin{thm}\label{thm-type C}Suppose that $(G,K)$ is of type {\rm CI}. 
We have
 \beqn
 \begin{gathered}
\on{Char}_K(\Lg_1)=\left\{\on{IC}(\widecheck\cO_{m,k,\mu},\cT_{\rho,\tau})\,|\,\rho\in\Theta_{m}^C,\ \tau\in\cP(k)\right\}.
\end{gathered}
 \eeqn
 \end{thm}

 \begin{thm}\label{thm-type BD}
{\rm (i)} Suppose that $(G,K)$ is of type {\rm BI}. We have
 \beqn
 \begin{gathered}
\on{Char}_K(\Lg_1)=\left\{\on{IC}(\widecheck\cO_{m,k,\mu},\cT_{\rho,\tau,\phi})\,|\,\rho\in\Theta_{m}^B,\tau\in\cP(k),\ \phi\in\Pi_{\cO_\mu} \right\}.
 \end{gathered}
 \eeqn
{\rm (ii)}  Suppose that $(G,K)$ is of type {\rm DI}. Then
\begin{eqnarray*}
\on{Char}_K(\Lg_1)&=&\left\{\on{IC}(\widecheck\cO_{m,k,\mu},\cT_{\rho,\tau,\phi})\,|\,\mu\neq\emptyset,\,\rho\in\Theta_{m}^D,\,\,\omega=\mathrm{\rm{I}},\mathrm{\rm{II}},\,\tau\in\cP(k),\ \phi\in\Pi_{\cO_\mu} \right\}\\
&\bigcup&\left\{\on{IC}(\widecheck\cO_{m,k,\emptyset},\cT_{\rho,\tau})\,|\,m\geq 1,\,\rho\in\Theta_{m}^{D,0},\,\tau\in\cP(k)\right\}\ (\text{if $p=q$}\,)\\
&\bigcup&\left\{\on{IC}(\widecheck\cO_{0,n/2,\emptyset}^\omega,\cT_{\tau})\,|\,\omega=\mathrm{\rm{I}},\mathrm{\rm{II}},\,\tau\in\cP(n/2)\right\}\ (\text{if $p=q$ even}\,). 
\end{eqnarray*}
 \end{thm}
\begin{remark}
For the pair $(GL_{2n},Sp_{2n})$, the nilpotent $K$-orbits in $\cN_1$ are parametrized by partitions of $n$. Moreover, $A_K(x)=1$ for $x\in\cN_1$. The restricted root system is of type $A_{n-1}$ with each restricted root space of dimension 4. We have $W_\fa=S_n$ and $I=1$. Thus
\beqn
\on{Char}_K(\Lg_1)=\{\on{IC}(\Lg_1^{rs},\cL_\tau)\,|\,\tau\in \cP(n)\},
\eeqn
where $\cL_\tau$ is the irreducible $K$-equivariant local system on $\Lg_1^{rs}$ given by the irreducible representation of $\pi_1^K(\Lg_1^{rs})=B_{S_n}$, where $B_{S_n}$ acts through $S_n$ via the irreducible representation corresponding to $\tau\in\cP(n)$. This recovers the corresponding result in~\cite{G2, H, L}.
\end{remark} 
As mentioned earlier, Theorems~\ref{nil thm-1} and~\ref{nil thm-2}, and Corollary~\ref{coro-full} follow from the above theorems whose proofs are given in the next two sections. We also obtain the following corollary.

\begin{corollary} 
\label{cuspidal corollary}
 Theorem~\ref{def of cuspidal}  from the introduction holds.  
\end{corollary}

\begin{proof}
We first claim that all full support character sheaves for the pair $(GL_{2n},GL_n\times GL_n)$ can be obtained via parabolic induction
from the constant sheaf $\bC_{\Ll_1}$ for the $\theta$-stable Levi $L=GL_2^n$ and $L^\theta=(GL_1\times GL_1)^n$. This can be seen as follows.  Recall that $G=GL_V$ and $K=GL_{V^+}\times GL_{V^-}$, where $V^+=\on{span}\{e_1,\ldots,e_n\}$ and $V^-=\on{span}\{f_1,\ldots,f_n\}$. Consider the $\theta$-stable parabolic subgroup $P$ that stabilizes  the following flag
$
0\subset V_2\subset V_4\subset\cdots\subset V_{2n-2}\subset V_{2n}=V,\ V_{2m}=\on{span}\{e_1,\ldots,e_m\}\oplus \on{span}\{f_1,\ldots,f_m\}.
$
Let $L$ be the natural $\theta$-stable Levi subgroup. Then $(L,L^\theta)\cong(GL_2^n,\,(GL_1\times GL_1)^n)$. Consider the map
$
\pi:K\times^{P_K}\Lp_1\to\Lg_1.
$
One checks readily that 
\beqn
\pi_*\bC[-]=\bigoplus_{\tau\in\cP(n)}\on{IC}(\Lg_1^{rs},\cL_\tau)\oplus\cdots\,.
\eeqn
Since $\pi_*\bC[-]\cong\on{Ind}_{\Ll_1\subset\Lp_1}^{\Lg_1}\bC_{\Ll_1}$, the claim follows.

One then checks that for the symmetric pairs $(Sp_{4n},Sp_{2n}\times Sp_{2n})$ and\linebreak $(SO_{4n},GL_{2n})$, the full support character sheaves can be obtained via parabolic induction from the full support character sheaves for the $\theta$-stable Levi subgroup $L=GL_{2n}$ with $L^\theta=GL_n\times GL_n$.

To complete the proof we appeal to Theorems~\ref{thm-type A}, \ref{thm-type C}, and~\ref{thm-type BD}.
\end{proof}

\section{Proof of Theorems~\ref{thm-type A}-\ref{thm-type BD}}\label{sec-proof}

In this section we prove Theorems~\ref{thm-type A}, \ref{thm-type C}, and~\ref{thm-type BD} by producing new character sheaves via parabolic induction. Combining these character sheaves with  nilpotent support ones and full support ones already constructed we have then shown that the sheaves listed in Theorems~\ref{thm-type A}, \ref{thm-type C}, and~\ref{thm-type BD} are indeed character sheaves. To prove that we have all of them we show that the number of the sheaves we have constructed coincides with the number of irreducible $K$-equivariant local systems on $\cN_1$ (see~\eqref{twosides}).

Let $\cO_{m,k,\mu}\in\underline{\cN_1^{\text{cs}}}$, where we write $\cO_{0,k,\mu}=\cO_{k,\mu}$ in type AIII, CII, DIII. We begin by associating to each pair $(G,K)$ a family of $\theta$-stable parabolic subgroups $P_{m,k,\mu}$ together with their $\theta$-stable Levi subgroups $L_{m,k,\mu}$ such that
\beqn
K.(\Lp_{m,k,\mu})_1=\overline{\widecheck\cO}_{m,k,\mu},\ 
\eeqn
where $(\fp_{m,k,\mu})_1=\fp_{m,k,\mu}\cap\Lg_1$, $\fp_{m,k,\mu}=\on{Lie}P_{m,k,\mu}$. 
We then study 
 the parabolic induction of full support character sheaves in $\on{Char}_{L_{m,k,\mu}^\theta}^\rf((\Ll_{m,k,\mu})_1)$ (see~\S\ref{sec-induction}).

\subsection{A family of $\theta$-stable parabolic subgroups of $G$}\label{theta stable parabolics}

Let $\cO_{m,k,\mu}\in\underline{\cN_1^{\text{cs}}}$.    Consider an {\em ordered }sequence $a_1$, $\ldots,\ a_{n-m-2k}$, with $a_i\in\{0,1\}$, such that
\begin{eqnarray*}
\text{(A\rm{III} or C\rm{II})}
\ 
\sum {a_i}=q-k,\ \ 
\text{(BD\rm{I})}\  \sum {a_i}=[\frac{q-m-2k}{2}].
\end{eqnarray*}
The ordered sequence $(a_i)$ defines a $\tilde K'$-conjugacy class of  $\theta$-stable Borel subgroup(s) $B'_{(a_i)}$ for the symmetric pair $(G',K')$
\begin{eqnarray*}
\text{(A\rm{III})}\ (GL_{n-2k},GL_{p-k}\times GL_{q-k})
\ \text{(BD\rm{I})}\ (SO_{N-2a},S(O_{p-a}\times O_{q-a}))\\
\text{(C\rm{I})}\ (Sp_{2n-2m-4k},GL_{n-m-2k})\ 
\text{(C\rm{II})}\ (Sp_{2n-4k},Sp_{2p-2k}\times Sp_{2q-2k})\ \ \text{(D\rm{III})}\ (SO_{2n-4k},GL_{n-2k}).
\end{eqnarray*}
We can and will choose a sequence $(a_i)$ such that the corresponding Richardson orbit $\cO_{B'_{(a_i)}}$ is $\cO_\mu$.  Moreover, in type BD\rm{I},  we choose $(a_i)$ such that the  corresponding $\theta$-stable Borel subgroup $B'_{(a_i)}$ is as in Proposition~\ref{sfiber-2} for $\cO_{\mu}$. We write $B'_\mu=B'_{(a_i)}$. We fix such a sequence $(a_i)$ and define the parabolic subgroup $P_{m,k,\mu}=P_{m,k,(a_i)}$ to be the subgroup of $G$ that stabilizes a flag  of the following form
\begin{eqnarray*}
\text{(A\rm{III})}&&0\subset V_1\subset V_2\subset\cdots\subset V_{n-m-2k}\subset V\\
\text{(BD\rm{I}) or (C\rm{I})}&&0\subset V_1\subset\cdots\subset V_{n-m-2k}\subset V_{n-m}\subset V_{n-m}^\p\subset V_{n-m-2k}^\p\subset\cdots\subset V_1^\p\subset V\\
\text{(C\rm{II}) or (D\rm{III})}&&0\subset V_1\subset V_2\subset\cdots\subset V_{n-m-2k}\subset V_{n-m-2k}^\p\subset\cdots\subset V_1^\p\subset V,
\end{eqnarray*}
 where
$
 V_i=V_{i-c_i}^+\oplus V_{c_i}^-,\ 1\leq i\leq n-m-2k,\ 
 V_{n-m}=V_{n-m-k-c_{n-m-2k}}^+\oplus V_{k+c_{n-m-2k}}^-,\ c_i=\sum_{j=1}^ia_j.
$ The parabolic subgroup $P_{m,k,\mu}$ is $\theta$-stable. Let $L_{m,k,\mu}=L_{m,k,(a_i)}$ be the natural $\theta$-stable Levi subgroup of   $P_{m,k,\mu}$.  We have 
\begin{eqnarray*}
\text{(A\rm{III})}&&
 L_{k,\mu}\cong  GL_{2k}\times  GL_1^{n-2k},\ 
L_{k,\mu}^\theta\cong GL_k\times GL_k\times  GL_1^{n-2k};
\\
\text{(BD\rm{I})}&&
 L_{m,k,\mu}\cong SO_{2m+m_0}\times GL_{2k}\times GL_1^{n-m-2k},\\
&&L_{m,k,\mu}^\theta\cong S(O_{m+m_0}\times O_m)\times GL_k\times GL_k\times GL_1^{n-m-2k};
\\
\text{(C\rm{I})} &&L_{m,k,\mu}\cong Sp_{2m}\times GL_{2k}\times GL_1^{n-m-2k},\ L_{m,k,\mu}^\theta\cong GL_{m}\times GL_k\times GL_k\times GL_1^{n-m-2k}; 
\\
\text{(C\rm{II})}&&L_{m,k,\mu}\cong Sp_{4k}\times GL_1^{n-2k},\ L_{m,k,\mu}^\theta\cong Sp_{2k}\times Sp_{2k}\times GL_1^{n-2k};
\\
\text{(D\rm{III})}&&L_{m,k,\mu}\cong SO_{4k}\times GL_1^{n-2k} ,\ L_{m,k,\mu}^\theta\cong GL_{2k}\times GL_k^{n-2k},
\end{eqnarray*}
where $m_0=1$ (resp. $0$) in type ${\rm BI}$ (resp. ${\rm DI}$).  One readily checks that 
$
K.(\fp_{m,k,\mu})_1=\overline{\widecheck\cO}_{m,k,\mu}.
$ 

Note that in type \rm{DI}, when $m=0$, the $\tilde{K}$-conjugacy class of $P_{0,k,\mu}$ in $G/P$ decomposes into two $K$-conjugacy classes; as before,  we write $P_{0,k,\mu}^{\rm I}$ and $P_{0,k,\mu}^{\rm II}$ for the two $\theta$-stable parabolic subgroups that are conjugate under $\tilde{K}$ but not conjugate under $K$. Moreover, when $\mu\neq\emptyset$, $K.(\fp_{0,k,\mu}^{\rm{I}})_1=K.(\fp_{0,k,\mu}^{\rm{II}})_1$, and when $\mu=\emptyset$, let $K.(\fp_{0,n/2,\emptyset}^{\omega})_1=\overline{\widecheck\cO^\omega}_{0,n/2,\emptyset}$, $\omega=\rm{I},\rm{II}$.

For later use, let us describe the locus ${\widecheck\cO}_{m,k,\mu}$ in more detail. 

{(\rm{AIII})} \  Let $x\in {\widecheck\cO}_{k,\mu}$.  Then  $x$ has eigenvalues $\pm a_1,\ldots,\pm a_k$, each of multiplicity $1$, and eigenvalue $0$ of multiplicity $n-2k$. Moreover, there exist $v_i^+\in V^+$, $v_i^-\in V^-$, $i=1,\ldots, k$, such that  $V=W_0\oplus_{i=1}^k (W_{i}\oplus W_{-i}$), where $W_i=\on{span}\{v_i^++v_i^-\}$, $W_{-i}=\on{span}\{v_i^+-v_i^-\}$, $x|W_i=a_i$,\ $x|_{W_{-i}}=-a_i$, $i=1,\ldots,k$, and $W_0$ is the generalized eigenspace of $x$ with eigenvalue $0$. We obtain a symmetric pair $(GL_{W_0}, GL_{p-k}\times GL_{q-k})$ and $x|_{W_0}\in\cO_\mu$. 

{(\rm{BDI})} \  Let $x\in {\widecheck\cO}_{m,k,\mu}$.  Then  $x$ has eigenvalues $\pm a_1,\ldots,\pm a_m$, each of multiplicity $1$, eigenvalues $\pm b_1,\ldots,\pm b_k$, each of multiplicity $2$, and eigenvalue $0$ of multiplicity $N-2m-4k$. Moreover, we have an orthogonal decomposition $V=W_0\oplus \oplus_{i=1}^mW_i\oplus\oplus_{j=1}^k U_j$, where $W_0$ is the generalized eigenspace of $x$ with eigenvalue $0$, $W_i=W_{a_i}\oplus W_{-a_i}$, $x|_{W_{a_i}}=a_i$, $x|_{W_{-a_i}}=-a_i$, and $U_j=U_{b_j}\oplus U_{-b_j}$ with $U_{\pm b_j}$ the generalized eigenspace of $x$ with eigenvalue $\pm b_j$. Furthermore, there exist $v_i^+\in V^+$, $v_i^-\in V^-$, $i=1,\ldots, m$, such that $W_i=\on{span}\{v_i^++v_i^-\}$, $W_{-i}=\on{span}\{v_i^+-v_i^-\}$, and $(v_i^+,v_i^+)\neq 0$, $(v_i^-,v_i^-)\neq 0$. Similarly, $U_{b_j}=\on{span}\{w_j=w_j^++w_j^-,u_j=u_j^++u_j^-\}$, $U_{-b_j}=\on{span}\{w_j^+-w_j^-,u_j^+-u_j^-\}$, where $xw_j=b_jw_j$, $xu_j=b_ju_j+w_j$, $(w_j^+,w_j^+)=(w_j^-,w_j^-)=0$, $(w_j^+,u_j^+)\neq0$ and $(w_j^-,u_j^-)\neq0$. (The properties about the pairing of the basis elements can be deduced using $(xv,w)+(v,xw)=0$ and the fact that $x$ takes $V^+$ to $V^-$ and vice versa.) Since $(,)$ is nondegenerate on $W_0$, we get a symmetric pair $(SO_{N-2m-4k}, S(O_{p-m-2k}\times O_{q-m-2k}))$ and $x|_{W_0}\in\cO_\mu$.

{(\rm{CI})} \  The description of $x\in {\widecheck\cO}_{m,k,\mu}$ is entirely similar to the case of type \rm{BD\,I} except that we have $\langle v_i^+,v_i^-\rangle\neq 0$, $\langle w_j^+,w_j^-\rangle=0$, $\langle w_j^+,u_j^-\rangle\neq 0$ and $\langle w_j^-,u_j^+\rangle\neq 0$. 

{(\rm{CII}) (\rm{DIII})} \  Let $x\in {\widecheck\cO}_{k,\mu}$. Then  $x$ has eigenvalues $\pm a_1,\ldots,\pm a_k$, each of multiplicity $2$,  and eigenvalue $0$ of multiplicity $2n-4k$. We have an orthogonal decomposition $V=W_0\oplus \oplus_{i=1}^kW_i$, $W_i=W_{a_i}\oplus W_{-a_i}$, $x|_{W_{a_i}}=a_i$, $x|_{W_{-a_i}}=-a_i$, $\on{dim}W_{a_i}=\on{dim}W_{-a_i}=2$, and $x|_{W_0}\in\cO_\mu$.

 \subsection{Equivariant fundamental groups}In this subsection, we assume that either $m\neq 0$ or $k\neq 0$.
In what follows, to simplify the notation, we omit the subscripts and write $L=L_{m,k,\mu}$, $P=P_{m,k,\mu}$, and $\overline{\widecheck\cO}=K.\Lp_1$ etc.  We describe the relations among  the equivariant fundamental groups $\pi_1^{P_K}(\Lp_1^r)$, $\pi_1^K(\widecheck\cO)$ and $\pi_1^{L^\theta}(\Ll_1^{rs})$, where 
$
P_K=P\cap K,\ \  \Lp_1^r=\Lp_1\cap{\widecheck\cO},
$
and $\Ll_1^{rs}$ is the set of regular semisimple elements in $\Ll_1$ with respect to the symmetric pair $(L,L^\theta)$.

 Consider the natural projection map $\on{pr}:\Lp_1\to\Ll_1$. From the description of ${\widecheck\cO}_{m,k,\mu}$ at the end of the previous subsection, we see that 
$
 \on{pr}(\Lp_1^r)=\Ll_1^{rs}.
$
 Arguing as in~\cite[\S3.2]{CVX}, we see that 
\beq\label{map-plfd}
\text{the canonical map }\Phi:\pi_1^{P_K}(\Lp_1^r)\to \pi_1^{L^\theta}(\Ll_1^{rs})\text{ is surjective}.
\eeq
 We have
 \begin{eqnarray*}
 \pi_1^{L^\theta}(\Ll_1^{rs})\cong B_{W_k}\ && \text{(A\rm{III}) (C\rm{II}) (D\rm{III})}\\
 \pi_1^{L^\theta}(\Ll_1^{rs})\cong \widetilde{B}_{W_m}\times B_{W_k}\text{ (resp. $\widetilde{B}_{W_m'}\times B_{W_k}$)}\ &&\text{(B\rm{I}) (C\rm{I}) (resp. D\rm{I})}.
 \end{eqnarray*}
  To determine $\pi_1^{P_K}(\Lp_1^r)$ we consider the fibration
 \beqn
\mathring \pi:K\times ^{P_K}\Lp_1^r\to{\widecheck\cO}
 \eeqn
 which is the smooth part of the map $\pi:K\times ^{P_K}\Lp_1\to\overline{\widecheck\cO}$.
 We first note that $\pi_1^{P_K}(\Lp_1^r)= \pi_1^{K}(K\times ^{P_K}\Lp_1^r)$. Now,  $\pi_1^K(\widecheck\cO)$ acts on the set $\on{Irr}(\pi^{-1}(x))$ of irreducible components of $\pi^{-1}(x)$, $x\in\widecheck\cO$, with $\pi_1^{K}(K\times ^{P_K}\Lp_1^r)$ as stabilizer.  Thus we are reduced to study the action of $\pi_1^K(\widecheck\cO)$ on the set $\on{Irr}(\pi^{-1}(x))$. To understand this action we describe the fiber $\pi^{-1}(x)$, $x\in\widecheck\cO$.

For $x\in\widecheck\cO=\widecheck\cO_{m,k,\mu}$,  let $x'=x|_{W_0}$, where $W_0$ is the generalized eigenspace of $x$ with eigenvalue $0$. Let $(G',K')$ be the symmetric pair defined by $W_0$. Then $x'\in\cO_{\mu}\subset\Lp'_1$.  As before let $\tilde{K}'=O_{p'}\times O_{q'}$ when $(G',K')$ is of type \rm{DI} and $\tilde{K}'=K'$ otherwise. Consider $$\tilde\pi_{B'_\mu}:\tilde K'\times^{(B'_\mu)_{K'}}(\Ln'_\mu)_1\to\bar\cO_{\mu}.$$
\begin{lemma}
 We have 
 \beqn
 \pi^{-1}(x)\cong
\tilde\pi_{B'_\mu}^{-1}(x')\,. \eeqn
 \end{lemma}
\begin{proof}
We prove the lemma in the case of type \rm{BDI}. The other cases are entirely similar. We show that the map
\beqn
\begin{array}{l}
\pi^{-1}(x)\to\tilde\pi_{B'_\mu}^{-1}(x') \qquad
 (V_1\subset \cdots\subset V_{n-m-2k}\subset V_{n-m})\mapsto (V_1\subset \cdots\subset V_{n-m-2k})
\end{array}
\eeqn
is an isomorphism. 

Suppose that $(V_1\subset \cdots\subset V_{n-m-2k}\subset V_{n-m})\in\pi^{-1}(x)$. We first show that $x|_{V_{n-m-2k}}=0$, i.e., $V_{n-m-2k}\subset W_0$, the generalized eigenspace of $x$ with eigenvalue $0$. Since $V_1$ is $x$-stable and by definition either $V_1\subset V^+$ or $V_1\subset V^-$, $x|_{V_1}=0$. Continuing by induction with $x$ applying to $V_{i}/V_{i-1}$, which is one dimensional, we conclude that $x|_{V_{n-m-2k}}=0$. Since $\dim W_0=N-2m-4k$ and $n=[N/2]$, we see that $V_{n-m-2k}$ is a maximal isotropic subspace in $W_0$. Since $x|_{W_0}=x'$, it follows that $(V_1\subset \cdots\subset V_{n-m-2k})
\in \tilde\pi_{B'_\mu}^{-1}(x')$. It remains to show that $V_{n-m}$ is uniquely determined by $V_i$, $i=1,\ldots,n-m-2k$. 

We make use of the description of $\widecheck\cO$ given in the end of \S\ref{theta stable parabolics}. Recall the orthogonal decomposition $V=W_0\oplus \oplus_{i=1}^mW_i\oplus\oplus_{j=1}^k U_j$ and the basis vectors  defined there. We claim that $V_{n-m}=V_{n-m-2k}\oplus\oplus_{j=1}^{k}\on{span}\{w_j^+,w_j^-\}$. This can be seen as follows. Consider the map induced by $x$ on $V_{n-m}/V_{n-m-2k}$ and the eigenvalues of $x$. Note that by definition, if a vector $v=v^++v^-\in V_{n-m}$, $v^{\pm}\in V^\pm$, then $v^+\in V_{n-m}$ and $v^-\in V_{n-m}$. Thus we conclude that the eigenvalue can not be $\pm a_i$ as $(v_i^+,v_i^-)\neq 0$ and $V_{n-m}$ is isotropic.  Now there exists some $j$ such that $w_j^+,w_j^-\in V_{n-m}$. Thus $V_{n-m}\subset \{w_j^+,w_j^-\}^\p$. Continuing by induction, we see that the claim holds.
\end{proof}

Using the above lemma and Proposition~\ref{sfiber-1} we conclude that
\begin{subequations}
\begin{eqnarray}\label{equivpi1-p1}
\pi_1^{P_K}(\Lp_1^r)=\pi_1^K(\widecheck\cO)&&\text{ (\rm{DI}) when $\mu=\emptyset$,\text{ or   (\rm{AIII}), (\rm{CI}), (\rm{CII}), (\rm{DIII}})\,.}
\end{eqnarray}
It remains to consider the types BI and DI when $\mu\neq\emptyset$. We first observe that $K\times ^{P_K}\Lp_1^r$ is connected and hence the action of $\pi_1^K(\widecheck\cO)$ on the set of connected components of $\pi^{-1}(x)$ is transitive. Also, by the previous lemma and Proposition~\ref{sfiber-2} the cardinality of the set of connected components of the fiber $\pi^{-1}(x)$, and hence that of the quotient $\pi_1^K(\widecheck\cO)/\pi_1^{P_K}(\Lp_1^r)$, is $2^{l_\mu}$. In particular, the action of $\pi_1^K(\widecheck\cO)$ factors through $A_{\tilde K'}(x')$.  Using the specific splitting of $\pi_1^K(\widecheck\cO)$ as a product in~\eqref{eqpi1-o2}, we can write
\begin{eqnarray}
\label{equivpi1-p2}
\pi_1^{P_K}(\Lp_1^r)\cong \widetilde{B}_{W_m}\times\widetilde{B}_{W_k}\times (\bZ/2\bZ)^{r_\mu-l_\mu}&&\text{ (\rm{BI})}\\
\label{equivpi1-p3}
\pi_1^{P_K}(\Lp_1^r)\cong\widetilde{B}^0_{W_m}\times\widetilde{B}_{W_k}\times (\bZ/2\bZ)^{r_\mu-l_\mu+1}&&\text{  (\rm{DI}) when $\mu\neq\emptyset$},
\end{eqnarray}
\end{subequations}
where $\widetilde{B}^0_{W_m}=(\bZ/2\bZ)^m\rtimes B_{W_m}^0\subset(\bZ/2\bZ)^m\rtimes B_{W_m}$ and $B_{W_m}^0$ is the subgroup of $B_{W_m}$ defined as the kernel of the map $B_{W_m}\to\bZ/2\bZ$ 
where a word in the braid generators $\sigma_1,\ldots,\sigma_m$ ($\sigma_m$ is the special one) is mapping to 0 (resp. 1) if the braid generator $\sigma_m$ appears even (resp. odd) number of times. Moreover in type  \rm{BI} (resp. \rm{DI}), the factor  $(\bZ/2\bZ)^{r_\mu-l_\mu}$ (resp. $(\bZ/2\bZ)^{r_\mu-l_\mu+1}$) is given by $\{a\in A_{K'}(\cO_\mu)\,|\,\phi(a)=1\text{ for all }\phi\in\Pi_{\cO_\mu}\}$.

Note that in type {\rm{DI}} when $\mu\neq\emptyset$, the map $\Phi$ in~\eqref{map-plfd} gives rise to a surjective map
\beq\label{mapzeta}
\zeta:\widetilde{B}^0_{W_m}\twoheadrightarrow\widetilde{B}_{W_m'}.
\eeq
\begin{remark}
It was previously known that such a surjective map $B_{W_m}^0\to B_{W_m'}$ exists, see, for example, a discussion in~\cite{BT}.
\end{remark}

 \subsection{Induced nilpotent orbital complexes}
 
Let us consider the parabolic induction of full support character sheaves on $\Ll_1$. For each irreducible representation $\psi$ of $\pi_1^{L^\theta}(\Ll_1^{rs})$, we write $\psi$ also for the irreducible representation of $\pi_1^{P_K}(\Lp_1^r)$ obtained via pull-back from the surjective map $\Phi:\pi_1^{P_K}(\Lp_1^r)\to \pi_1^{L^\theta}(\Ll_1^{rs})$ in~\eqref{map-plfd}. Then we have:
\beqn
\on{Ind}_{\Ll_1\subset\fp_1}^{\Lg_1}\on{IC}(\Ll_1^{rs},\cL_\psi) \ = \  \pi_! \on{IC}(K\times^{P_K}\Lp_1^r,\cL_\psi)[-]\,.
\eeqn
Furthermore, by the decomposition theorem, we have:
\beqn
 \pi_! \on{IC}(K\times^{P_K}\Lp_1^r,\cL_\psi) \cong \on{IC}(\widecheck\cO, R^{2d}\mathring\pi_!(\cL_\psi))[-]\oplus \{\text{other terms}\}\,,
\eeqn
where $d$ is the relative dimension of $\pi$. By construction we have
\beqn
R^{2d}\mathring\pi_!(\cL_\psi) \cong \bC[\pi_1^K(\widecheck\cO)]\otimes_{\bC[\pi_1^{P_K}(\Lp_1^r)]}\psi\,.
\eeqn
Consider the following decomposition
\beqn
\bC[\pi_1^K(\widecheck\cO)]\otimes_{\bC[\pi_1^{P_K}(\Lp_1^r)]}\psi=\psi_1\oplus\cdots\oplus\psi_k
\eeqn
 into irreducible representations  of $\pi_1^K(\widecheck\cO)$.  By the discussion above we have
 \beq\label{eqn-local systems}
 \bigoplus_{i=1}^k\on{IC}(\widecheck\cO,\cL_{\psi_i})\text{ appears as a direct summand of }\on{Ind}_{\Ll_1\subset\fp_1}^{\Lg_1}\on{IC}(\Ll_1^{rs},\cL_\psi)\text{ up to shift}.
 \eeq
 Let us write $A[-] \subset\mkern-15mu\raisebox{1.2pt}{\mbox{\larger[-12]$\oplus$}} B$ to indicate that $A$ appears as a direct summand of $B$ up to shift.

For each $\tau\in\cP(k)$, we continue to write $L_\tau$ for the $B_{W_k}$-representation obtained by  pulling back $L_\tau\in\on{Irr}\cH_{W_k,1,-1}$  via the surjective map $\bC[B_{W_k}]\to \cH_{W_k,1,-1}$.
The following claims follow from the description of $\pi_1^K(\widecheck\cO)$ in~\eqref{eqpi1-o1}-\eqref{eqpi1-o3}, the description of $\pi_1^{P_K}(\Lp_1^r)$ in~\eqref{equivpi1-p1}-\eqref{equivpi1-p3} and~\eqref{eqn-local systems}.

{\bf ({\rm AIII}) ({\rm CII}) ({\rm DIII})}\ \   For each $\tau\in\cP(k)$,  let $\cL_\tau$ be the irreducible $L^\theta$-equivairant local system on $\Ll_1^{rs}$ given by $L_\tau$. Then
\beqn
\on{IC}(\widecheck\cO,\cT_\tau)[-]\ \dsum\ \on{Ind}_{\Ll_1\subset\fp_1}^{\Lg_1}\on{IC}(\Ll_1^{rs},\cL_\tau).\eeqn

{\bf ({\rm CI})} Let $\cL_\rho\boxtimes\cL_\tau$, $\rho\in\Theta_{m}^C$ and $\tau\in\cP(k)$, be the irreducible $L^\theta$-equivairant local system on $\Ll_1^{rs}$ given by the irreducible representation $\rho\boxtimes L_\tau$ of $\pi_1^{L^\theta}(\Ll_1^{rs})$, where $\widetilde{B}_{W_m}$ acts via $\rho$ and $B_{W_k}$ acts via $L_\tau$.  Then 
  \beqn
  \on{IC}(\widecheck{\cO},\cT_{\rho,\tau})[-]\ \dsum\ \on{Ind}_{\Ll_1\subset\fp_1}^{\Lg_1}\on{IC}(\Ll_1^{rs},\cL_\rho\boxtimes\cL_\tau).
  \eeqn

{\bf ({\rm BI})} Let $\cL_\rho\boxtimes\cL_\tau$, $\rho\in\Theta_{m}^B$, and $\tau\in\cP(k)$, be the  irreducible $L^\theta$-equivairant local system on $\Ll_1^{rs}$ given by the irreducible representation $\rho\boxtimes L_\tau$ of $\pi_1^{L^\theta}(\Ll_1^{rs})$, where $\widetilde{B}_{W_m}$ acts via $\rho$ and $B_{W_k}$ acts via $L_\tau$.  Then
 \beqn
\bigoplus_{\phi\in\Pi_{\cO_\mu}}\on{IC}(\widecheck\cO,\cT_{\rho,\tau,\phi})[-]\ \dsum\  \on{Ind}_{\Ll_1\subset\fp_1}^{\Lg_1}\on{IC}(\Ll_1^{rs},\cL_\rho\boxtimes\cL_\tau).
 \eeqn
 
{\bf ({\rm DI})}  Let $\cL_{\rho_0}\boxtimes\cL_\tau$, $\rho_0\in\Theta_{m}^{D,0}$ and $\tau\in\cP(k)$, be the irreducible $L^\theta$-equivairant local system on $\Ll_1^{rs}$ given by the irreducible representation $\rho_0\boxtimes L_\tau$ of $\pi_1^{L^\theta}(\Ll_1^{rs})$, where $\widetilde{B}_{W_m}$ acts via $\rho_0$ and $B_{W_k}$ acts via $L_\tau$.  Let $\rho_0$ also denote the representation of $\widetilde{B}_{W_m}^0$ obtained via the surjective map $\widetilde{B}_{W_m}^0\to\widetilde{B}_{W_m'}$ (see~\eqref{mapzeta}). Then, as representations of $\widetilde B_{W_m}$, 
$\bC[{\widetilde{B}_{W_m}}]\otimes_{\bC[{\widetilde{B}_{W_m}^0}]}\rho_0=\rho_0^{\rm{I}}\oplus\rho_0^{\rm{II}},
$
where $\rho_0^{\rm{I}}$ and $\rho_0^{\rm{II}}$ are non-isomorphic irreducible representations of $\widetilde{B}_{W_m}$. In view of Proposition \ref{thm-nearby cycles}, we conclude that $\rho_0^{\rm{I}},\rho_0^{\rm{II}}\in\Theta_m^D$.  Moreover, $\Theta_m^D=\{\rho_0^{\rm{I}},\rho_0^{\rm{II}}\mid\rho_0\in\Theta_m^{D,0}\}$. Thus
\begin{eqnarray*}
\bigoplus_{\phi\in\Pi_{\cO_\mu}}\bigoplus_{\omega=\rm{I},\rm{II}}\on{IC}(\widecheck\cO,\cT_{\rho_0^\omega,\tau,\phi})[-]\ \dsum\  \on{Ind}_{\Ll_1\subset\fp_1}^{\Lg_1}\on{IC}(\Ll_1^{rs},\cL_{\rho_0}\boxtimes\cL_\tau)\text{ when $\mu\neq\emptyset$};\\
\on{IC}(\widecheck\cO_{m,k,\emptyset},\cT_{\rho,\tau})[-]\ \dsum\  \on{Ind}_{\Ll_1\subset\fp_1}^{\Lg_1}\on{IC}(\Ll_1^{rs},\cL_\rho\boxtimes\cL_\tau),\text{ when $\mu=\emptyset$ and $m>0$.}\\
\on{IC}(\widecheck\cO_{0,n/2,\emptyset}^\omega,\cT_{\tau})[-]\ \dsum\ \on{Ind}_{\Ll_1^\omega\subset\fp_1^\omega}^{\Lg_1}\on{IC}((\Ll_1^\omega)^{rs},\cL_\tau), \ \omega=\mathrm{\rm{I},\rm{II}},\text{ when $\mu=\emptyset$ and $m=0$.}
\end{eqnarray*}

Since Fourier transform commutes with parabolic induction (see~\eqref{commutativity}), in view of~\eqref{char-fullsupp}, we conclude that  the IC sheaves in Theorems~\ref{thm-type A}-\ref{thm-type BD} are character sheaves. Since they are pairwise non-isomorphic, to prove Theorems~\ref{thm-type A}-\ref{thm-type BD}, it remains to check that the number of the IC sheaves equals the number of character sheaves in each case.

\subsection{Proof of Theorem~\ref{thm-type A}} 
In each case, we prove the theorem by establishing an explicit bijection of the two sides.

\subsubsection{}Suppose that $(G,K)=(GL_n,GL_p\times GL_q)$. 
Given a signed Young diagram $\lambda=(\lambda_1)^{p_1}_+(\lambda_1)^{q_1}_-\cdots(\lambda_s)^{p_s}_+(\lambda_s)^{q_s}_-$ of signature $(p,q)$, we associate it with the following data
\begin{enumerate} 
\item the partition $\alpha=(\lambda_1)^{l_1}\cdots(\lambda_s)^{l_s}\in\cP(k)$, $l_i=\min\{p_i,q_i\}$, $k=\sum\lambda_il_i$
\item  the signed Young diagram of signature $(p-k,q-k)$ 
\beqn
\mu=(\lambda_1)^{p_1-l_1}_+(\lambda_1)^{q_1-l_1}_-\cdots(\lambda_s)^{p_s-l_s}_+(\lambda_s)^{q_s-l_s}_-.
\eeqn
\end{enumerate}
The map
\beqn
\on{IC}(\cO_\lambda,\bC)\mapsto\on{IC}(\widecheck\cO_{1^k_+1^k_-\sqcup\mu},\cT_{\alpha^t})\text{ for all }\lambda;
\eeqn
where $\alpha^t$ denotes the transpose partition of $\alpha$, establishes the required bijection.

\subsubsection{}Suppose that $(G,K)=(Sp_{2n},Sp_{2p}\times Sp_{2q})$, $q\leq n$. Let  
\beqn
\lambda=(2\lambda_1)^{a_1}_+(2\lambda_1)^{a_1}_-\cdots(2\lambda_s)^{a_s}_+(2\lambda_s)^{a_s}_-(2\mu_1+1)^{2p_1}_+(2\mu_1+1)^{2q_1}_-\cdots(2\mu_t+1)^{2p_t}_+(2\mu_t+1)^{2q_t}_-
\eeqn
be a signed Young diagram of signature $(2p,2q)$. 
Let $l_i=\min\{p_i,q_i\}$ and $k=\sum2[\frac{a_i}{2}]\lambda_i+\sum l_i(2\mu_i+1)$. We associate to $\lambda$ the following data
\begin{enumerate}
\item the partition $\alpha=(2\lambda_1)^{[\frac{a_1}{2}]}\cdots(2\lambda_s)^{[\frac{a_s}{2}]}(2\mu_1+1)^{l_1}\cdots(2\mu_t+1)^{l_t}\in\cP(k)$,
\item the signed Young diagram of signature $(2p-2k,2q-2k)$
\beqn
\begin{gathered}
\mu=(2\lambda_1)^{a_1-2[\frac{a_1}{2}]}_+(2\lambda_1)^{a_1-2[\frac{a_1}{2}]}_-\cdots(2\lambda_s)^{a_s-2[\frac{a_s}{2}]}_+(2\lambda_s)^{a_s-2[\frac{a_s}{2}]}_-(2\mu_1+1)^{2p_1-2l_1}_+(2\mu_1+1)^{2q_1-2l_1}_-\\\cdots(2\mu_t+1)^{2p_t-2l_t}_+(2\mu_t+1)^{2q_t-2l_t}_-.
\end{gathered}
\eeqn
\end{enumerate}
The map
\beqn
\on{IC}(\cO_\lambda,\bC)\mapsto\on{IC}(\widecheck\cO_{1^{2k}_+1^{2k}_-\sqcup\mu},\cT_{\alpha^t})
\eeqn
establishes the required bijection.
 
\subsubsection{}Suppose that $(G,K)=(SO_{2n},GL_n)$. 
Let  
\beqn
\lambda=(2\lambda_1+1)^{a_1}_+(2\lambda_1+1)^{a_1}_-\cdots(2\lambda_s+1)^{a_s}_+(2\lambda_s+1)^{a_s}_-(2\mu_1)^{2p_1}_+(2\mu_1)^{2q_1}_-\cdots(2\mu_t)^{2p_t}_+(2\mu_t)^{2q_t}_-
\eeqn
be a signed Young diagram of signature $(n,n)$. 
Let $l_i=\min\{p_i,q_i\}$ and $k=\sum[\frac{a_i}{2}](2\lambda_i+1)+\sum 2\mu_il_i$. We associate to the above Young diagram the following data
\begin{enumerate} 
\item the partition $\alpha=(2\lambda_1+1)^{[\frac{a_1}{2}]}\cdots(2\lambda_s+1)^{[\frac{a_s}{2}]}(2\mu_1)^{l_1}\cdots(2\mu_t)^{l_t}\in\cP(k)$,
\item the signed Young diagram of signature $(n-2k,n-2k)$
\beqn
\begin{gathered}
\mu=(2\lambda_1+1)^{a_1-2[\frac{a_1}{2}]}_+(2\lambda_1+1)^{a_1-2[\frac{a_1}{2}]}_-\cdots(2\lambda_s+1)^{a_s-2[\frac{a_s}{2}]}_+(2\lambda_s+1)^{a_s-2[\frac{a_s}{2}]}_-\\(2\mu_1)^{2p_1-2l_1}_+(2\mu_1)^{2q_1-2l_1}_-\cdots(2\mu_t)^{2p_t-2l_t}_+(2\mu_t+1)^{2q_t-2l_t}_-.
\end{gathered}
\eeqn
\end{enumerate}
The map
\beqn
\on{IC}(\cO_\lambda,\bC)\mapsto\on{IC}(\widecheck\cO_{1^{2k}_+1^{2k}_-\sqcup\mu},\cT_{\alpha^t})
\eeqn
establishes the required bijection.
 
Let us record that
 \beqn
 |\cA_{GL_n}{((\mathfrak{so}_{2n})_1)}|=\sum_{k=0}^{[\frac{n}{2}]}\pa(k)\pa(l,k)2^k\mathbf{do}(n-2l)=\text{Coefficient of $x^n$ in }\prod_{s\geq 1}\frac{1+x^s}{(1-x^{2s})^2}
\eeqn
 where $\pa(k)$ denotes the number of partitions of $k$, $\pa(l, k)$ denotes the number of partitions of $l$ into (not necessarily distinct) parts of exactly $k$ different sizes and
$\mathbf{do}(l)$ denotes the number of partitions of $l$ into distinct odd parts. Here we have used
\beq\label{partition identity-1}
\sum_k \pa(n,k)\,2^k=\text{Coefficient of }x^n\text{ in }\prod_{s\geq 1}\frac{1+x^s}{1-x^s}.
\eeq

\subsection{Proof of Theorem~\ref{thm-type C}}Let $(G,K)=(Sp_{2n},GL_n)$.

\begin{lemma}
We have
\beq\label{cardinality-type C}
|\cA_{GL_n}{((\mathfrak{sp}_{2n})_1)}|=\text{Coefficient of }x^n\text{ in }\prod_{s\geq 1}\frac{(1+x^s)^3}{(1-x^s)^2}.
\eeq
\end{lemma}
\begin{proof}
Let us write $\cP(C_n)$ for the set of partitions of $2n$ such that each odd part occurs with an even multiplicity.  
Given a partition 
\beqn
\lambda=(2\lambda_1)^{m_1}\cdots(2\lambda_s)^{m_s}(2\lambda_{s+1}+1)^{2m_{s+1}}\cdots(2\lambda_{t}+1)^{2m_{t}}\in \cP(C_n),
\eeqn 
using~\eqref{component group-C} we see that the number of orbital complexes $\on{IC}(\cO,\cE)$ supported on the nilpotent orbits whose underlying Young diagrams have shape $\lambda$ is
$
(\prod_{i=1}^sm_i)4^s:=M_\lambda.
$ 
Thus we have that
\beqn
|\cA_{GL_n}{((\mathfrak{sp}_{2n})_1)}|=\sum _{\lambda\in \cP(C_n)}M_\lambda.
\eeqn
For a partition $\mu=(\mu_1)^{m_1}\cdots(\mu_s)^{m_s}$, we define $N_\mu=(\prod_{i=1}^sm_i)4^s$. Then we have
\beqn
\sum _{\lambda\in \cP(C_n)}M_\lambda=\sum_k\sum_{\mu\in\cP(k)}N_\mu\cdot \mathbf{od}(n-k),
\eeqn
where $\mathbf{od}(l)$ denotes the number of the partitions of $l$ into odd parts. Since the number of partition of $n$ into distinct parts equals the number
of partitions of $n$ into odd parts, we have $\sum\mathbf{od}(l)x^l=\prod_{s\geq1}(1+x^s)$. The lemma follows from
\beqn
\sum_{\mu\in\cP(n)}N_\mu=\text{Coefficient of }x^n\text{ in }\prod_{s\geq 1}\frac{(1+x^s)^2}{(1-x^s)^2}.
\eeqn
This can be seen using~\eqref{partition identity-1} and the following observation. 
We can rewrite a partition $(\mu_1)^{m_1}\cdots(\mu_s)^{m_s}$ into the union of two partitions $(\mu_1)^{m_1'}\cdots(\mu_s)^{m_s'}$ and $(\mu_1)^{m_1-m_1'}\cdots(\mu_s)^{m_s-m_s'}$, where $0\leq m_i'\leq m_i$. Thus we get
\beqn
\left(1+\sum_{j=1}^s\prod_{1\leq i_1<i_2<\cdots<i_j\leq s}(m_{i_1}-1)\cdots(m_{i_j}-1)\right)4^s=\left(\prod_{i=1}^sm_i\right)4^s.
\eeqn
\end{proof}

Let us write $b_C(n)=|\on{SYD}^0_{(Sp_{2n},GL_n)}|$. Using Proposition~\ref{prop-Richardson} one readily checks that
\beqn
b_C(n)=\sum_{k}\pa(n,k)2^k=\text{coefficient of }x^n\text{ in }\prod_{s\geq 1}\frac{1+x^s}{1-x^s}.
\eeqn
Using~\eqref{number of fullsupp} and~\eqref{generating for Hecke}, we see that the number of IC sheaves in Theorem~\ref{thm-type C} is
\beqn
\begin{gathered}
\sum_{m=0}^n|\Theta_{m}^C|\sum_{k=0}^{[\frac{n-m}{2}]}\pa(k)\,b_C(n-m-2k)=\text{Coefficient of }x^n\\
\text{ in }\prod_{s\geq 1}(1+x^s)^3\prod_{s\geq 1}\frac{1}{1-x^{2s}}\prod_{s\geq 1}\frac{1+x^s}{1-x^s}
\overset{\eqref{cardinality-type C}}{=}|\on{Char}_{GL_n}((\mathfrak{sp}_{2n})_1)|.
\end{gathered}
\eeqn

\subsection{Proof of Theorem~\ref{thm-type BD}}\label{ssec-proof BD}

Let $(G,K)=(SO_N,S(O_p\times O_q))$. 
We write $\on{SYD}_{p,q}$ for the set of signed Young diagrams that parametrizes $K$-orbits in $\cN_1$. Let us also write
\beq\label{a_p,q1}
A_{p,q}=|\cA_{S(O_p\times O_q)}{((\mathfrak{so}_{p+q})_1)}|.
\eeq
Since $A_K(\cO_\lambda)=(\bZ/2\bZ)^{r_\lambda}$ by~\eqref{component group-BD}, we have
\beqn
A_{p,q}=\sum_{\lambda\in\on{SYD}_{p,q}}2^{r_\lambda}+\pa(\frac{q}{2})\,\delta_{p,q},
\eeqn
where $\pa(k)$ is the number of partitions of $k$ and we set $\pa(\frac{q}{2})=0$ if $q$ is odd. We note that the second term in the above equation arises only when $p=q$ is even. In the latter case there are two nilpotent orbits corresponding to each partition with only even parts.

The formula in the following proposition is derived and proved by Dennis Stanton.
\begin{proposition}\label{total number-BD}
\beqn
\sum_{p,q=0}^{\infty}2A_{p,q}u^pv^q-3\sum_{q=0}^{\infty}\pa(q)u^{2q}v^{2q}=\prod_{k=1}^{\infty}\frac{1}{1-u^{2k}v^{2k}}\prod_{m=0}^{\infty}\frac{(1+u^{m+1}v^m)(1+u^mv^{m+1})}{(1-u^{m+1}v^m)(1-u^mv^{m+1})}.
\eeqn
\end{proposition}
\begin{proof}
The proposition follows from Proposition~\ref{prop-generating function-total} in Appendix~\ref{combinatorics}, since by~\eqref{a_p,q1} and~\eqref{wt(p,q) def} we have
\beq\label{wtp_q}
wt(p,q)=2A_{p,q}\text{ if $p+q$ is odd},\ \ \ \ wt(p,q)=2A_{p,q}-3\pa(\frac{q}{2})\,\delta_{p,q}\text{ if $p+q$ is even}.
\eeq 
\end{proof}

Suppose that either $p$ or $q$ is even. Let us now write
\beqn
b_{p,q}=\sum_{\lambda\in\on{SYD}^0_{(SO_{p+q},S(O_p\times O_q))}}|\Pi_{\cO_\lambda}|
\eeqn
where $\Pi_{\cO}$ is defined in~\eqref{char-bi}. Namely, $b_{p,q}$ is the number of $\on{IC}(\cO,\cE_\phi)$'s such that $\cO\in\underline{\cN_1^0}$ and $\phi\in\Pi_\cO$. Note that we have $b_{p,q}=b_{q,p}$.

Let us write $\cP^{odd}(N)$ for the set of partitions of $N$ into odd parts. We write a partition $\lambda\in\cP^{odd}(N)$ as
\beqn
\lambda=(2\mu_1+1)+(2\mu_2+1)+\cdots+(2\mu_{s}+1)
\eeqn 
where $\mu_1\geq\mu_2\geq\cdots\geq\mu_{s}\geq 0$. Note that $s\equiv N\mod 2$. We set
\beqn
wt_\lambda=3^{\#\{1\leq j\leq (s-1)/2\,|\,\mu_{2j-1}=\mu_{2j}+1\}}\,4^{\#\{1\leq j\leq  (s-1)/2\,|\,\mu_{2j-1}\geq\mu_{2j}+2\}}\text{ when $N$ is odd};
\eeqn
\beqn
wt_\lambda=3^{\#\{1\leq j\leq s/2-1\,|\,\mu_{2j}=\mu_{2j+1}+1\}}4^{\#\{1\leq j\leq s/2-1\,|\,\mu_{2j}\geq\mu_{2j+1}+2\}}\text{ when $N$ is even}.
\eeqn
We then have
\beqn
\sum_{p=0}^{2n+1} b_{p,2n+1-p}=2\sum_{\lambda\in\cP^{odd}(2n+1)}wt_\lambda,\ \quad\sum_{p=0}^{n} b_{2p,2n-2p}=2\sum_{\lambda\in\cP^{odd}(2n)}wt_\lambda.
\eeqn

\begin{proposition}\label{prop-counting biorbital}
We have
\begin{subequations}
\beq\label{number of biorbital-B}
\sum_{p=0}^{2n+1}b_{p,2n+1-p}=\text{Coefficient of $x^{2n+1}$ in }2x\prod_{k\geq 1}(1+x^{4k})^2(1+x^{2k})^2
\eeq
\beq\label{number of biorbital-D}
\sum_{p=0}^{n}b_{2p,2n-2p}=\text{Coefficient of $x^{2n}$ in }\frac{1}{2}\prod_{k\geq 1}(1+x^{4k-2})^2(1+x^{2k})^2.
\eeq
\end{subequations}
\end{proposition}
\begin{proof}
See \S\ref{proof of biorbital counting} in Appendix~\ref{combinatorics}.
\end{proof}

Let us further write
$
f_{m+1,m}=|\Theta_{m}^B|\text{ and }f_{m,m}=|\Theta_{m}^{D,0}|. 
$
It follows from~\eqref{generating for Hecke},~\eqref{number of fullsupp} and~\eqref{number of fullsupp D} that
\begin{subequations}
\beq\label{number of full support-B}
\sum_{m\geq 0}{f_{m+1,m}}x^{m}=\prod_{k\geq1}(1+x^{2k})^2(1+x^k)^2
\eeq
\beq\label{number of full support-D}
\sum_{m\geq 0}{f_{m,m}}x^{m}=\frac{1}{2}\prod_{k\geq1}(1+x^{2k-1})^2(1+x^k)^2.
\eeq
\end{subequations}
Note that in the above we have used the notation that $b_{0,0}=1/2$ and $f_{0,0}=1/2$.

Let $A'_{p,q}$ denote the number of IC sheaves in Theorem~\ref{thm-type BD}. Recall $r=\min\{p,q\}$. We have
\begin{subequations}
\beq\label{tn-B}
A'_{p,q}=\sum_{m=0}^{r}f_{m+1,m}\sum_{k=0}^{[\frac{r-m}{2}]}\pa(k)\,b_{p-m-2k,q-m-2k},\ \ \text{ if $p+q\equiv 1\mod 2$},
\eeq
\beq\label{tn-D1}
A'_{p,q}=2\sum_{m=0}^{\frac{r-1}{2}}f_{2m+1,2m+1}\sum_{k=0}^{\frac{r-2m-1}{2}}\pa(k)\,b_{p-2m-2k-1,q-2m-2k-1},\text{ if $p\equiv q\equiv 1\mod 2 $,}
\eeq
\beq\label{tn-D2}
A'_{p,q}=2\sum_{m=0}^{\frac{q}{2}}f_{2m,2m}\sum_{k=0}^{\frac{q-2m}{2}}\pa(k)\,b_{p-2m-2k,q-2m-2k}+\frac{3}{2}\pa(\frac{q}{2})\delta_{p,q},\text{ if $p\equiv q\equiv 0\mod 2 $}.
\eeq
\end{subequations}
To prove Theorem~\ref{thm-type BD}, it suffices to show that
\beq\label{identity-lr}
\sum_{p=0}^{2n+1}A_{p,2n+1-p}=\sum_{p=0}^{2n+1}A_{p,2n+1-p}',\qquad \sum_{p=0}^{2n}A_{p,2n-p}=\sum_{p=0}^{2n}A_{p,2n-p}'
\eeq
since by construction $A_{p,2n+1-p}'\leq A_{p,2n+1-p}$ and $A_{p,2n-p}'\leq A_{p,2n-p}$.
It then follows that
\beq\label{a=a'}
A_{p,q}=A_{p,q}'.
\eeq
In what follows we prove~\eqref{identity-lr}.
Setting $u=v=x$ in Proposition~\ref{total number-BD} we see that
\begin{subequations}
\beq\label{sum of apq-1}
\sum_{p=0}^{2n+1}A_{p,2n+1-p}=\text{Coefficient of $x^{2n+1}$ in }\frac{1}{2}\prod_{k\geq1}\frac{1}{1-x^{4k}}\prod_{m\geq 1}\frac{(1+x^{2m-1})^2}{(1-x^{2m-1})^2}
\eeq
\beq\label{sum of apq-2}
\sum_{p=0}^{2n}A_{p,2n-p}=\text{Coefficient of $x^{2n}$ in }\frac{1}{2}\prod_{k\geq1}\frac{1}{1-x^{4k}}\prod_{m\geq 1}\frac{(1+x^{2m-1})^2}{(1-x^{2m-1})^2}+\frac{3}{2}\prod_{k\geq1}\frac{1}{1-x^{4k}}.
\eeq
\end{subequations}
The equations~\eqref{number of biorbital-B},~\eqref{number of full support-B}, and~\eqref{tn-B} imply that 
\begin{subequations}
\begin{eqnarray}\label{a'pq-1}
\sum_{p=0}^{2n+1}A_{p,2n+1-p}'&=&\text{  coefficient of $x^{2n+1}$ in }\left(\prod_{k\geq1}(1+x^{4k})^2(1+x^{2k})^2\right)\\
&&\cdot\left(\prod_{k\geq 1}\frac{1}{1-x^{4k}}\right)\left(2x\prod_{m\geq 1}(1+x^{4m})^2(1+x^{2m})^2\right).\nonumber\end{eqnarray}
The equations~\eqref{number of biorbital-D},~\eqref{number of full support-D},~\eqref{tn-D1} and~\eqref{tn-D2} imply that 
\begin{eqnarray}\label{a'pq-2}
\sum_{p=0}^{2n}A_{p,2n-p}'&=&\text{ coefficient of $x^{2n}$ in } 2\left(\frac{1}{2}\prod_{k\geq1}(1+x^{4k-2})^2(1+x^{2k})^2\right)\\
&&\cdot\left(\prod_{k\geq 1}\frac{1}{1-x^{4k}}\right)\cdot\left(\frac{1}{2}\prod_{m\geq 1}(1+x^{4m-2})^2(1+x^{2m})^2\right)+\frac{3}{2}\prod_{k\geq 1}\frac{1}{1-x^{4k}}\nonumber.  \end{eqnarray}
\end{subequations}
Let us now write 
\beqn
F(x)=\prod_{m\geq 1}\frac{(1+x^{2m-1})^2}{(1-x^{2m-1})^2}.
\eeqn
Then~\eqref{identity-lr} follows from~\eqref{sum of apq-1},~\eqref{sum of apq-2},~\eqref{a'pq-1},~\eqref{a'pq-2} and the following identities
\begin{subequations}
\beq\label{Fodd}
F(x)-F(-x)=8\,x\,\prod_{k\geq 1}(1+x^{4k})^4(1+x^{2k})^4
\eeq
\beq\label{Feven}
F(x)+F(-x)=2\,\prod_{k\geq 1}(1+x^{4k-2})^4(1+x^{2k})^4.
\eeq
\end{subequations}
The proof of equations~\eqref{Fodd} and~\eqref{Feven} is given in \S\ref{functions} in Appendix~\ref{combinatorics}. This completes the proof of Theorem~\ref{thm-type BD}.

\begin{corollary}Let $(G,K)=(SO_N,S(O_p\times O_q))$.
\begin{enumerate}
\item The number of character sheaves 
\beqn
\text{$|\on{Char}_K(\Lg_1)|=$ coefficient of $x^q$ in }\frac{1}{1+x^{p-q}}\prod_{s\geq 1}\frac{1+x^s}{(1-x^s)^3}+\frac{3\delta_{p,q}}{2}\prod_{s\geq 1}\frac{1}{1-x^{2s}}.
\eeqn
\item Suppose that either $p$ or $q$ is even. The number of nilpotent support character sheaves,  
\beqn
\text{$|\on{Char}_K^\rn(\Lg_1)|=$  coefficient of $x^q$ in }\left\{\begin{array}{ll}\displaystyle{
\frac{1}{1+x^{p-q}}\prod_{s\geq 1}\frac{(1+x^{2s-1})^2}{(1-x^{2s})^2}}&\text{ if $N$ is odd},
\\
\displaystyle{\frac{1}{1+x^{p-q}}\prod_{s\geq 1}\frac{(1+x^{2s})^2}{(1-x^{2s})^2}}&\text{ if $p\equiv q\equiv 0\,\nmod 2$}.\end{array}\right.
\eeqn
\end{enumerate}
\end{corollary}
\begin{proof}
We can assume that $p\geq q$. The case when $q< p$ follows readily. Part (1) follows from~\eqref{wtp_q} and  Corollary~\ref{prop-total number-pq} in Appendix~\ref{combinatorics}, as $|\on{Char}_K(\Lg_1)|=A_{p,q}$.

We prove part (2) in the case of $p$ and $q$ both even. The other case is entirely similar and simpler. Fix $2l=p-q\geq 0$. It follows from~\eqref{tn-D1},~\eqref{tn-D2} and~\eqref{a=a'} that 
\beqn
\sum_{q}A_{q+2l,q}\,x^{q}=2\left(\sum_{m}f_{m,m}x^{m}\right)\left(\sum_{k}\pa(k)x^{2k}\right)\left(\sum_{q}b_{2q+2l,2q}x^{2q}\right)+\frac{3}{2}\delta_{l,0}\sum_{q}\pa(q)x^{2q}.
\eeqn
By part (1), we have that
\beqn
\sum_{q}A_{q+2l,q}\,x^{q}=\frac{1}{1+x^{2l}}\prod_{s\geq 1}\frac{1+x^s}{(1-x^s)^3}+\frac{3}{2}\delta_{l,0}\prod_{s\geq 1}\frac{1}{1-x^{2s}},
\eeqn
Thus using~\eqref{number of full support-D}, we obtain 
\begin{eqnarray*}
\sum_{q}b_{2q+2l,2q}x^{2q}&=&\frac{1}{1+x^{2l}}\prod_{s\geq 1}\frac{(1+x^s)(1-x^{2s})}{(1-x^s)^3(1+x^{2s-1})^2(1+x^s)^2}=\frac{1}{1+x^{2l}}\prod_{s\geq 1}\frac{(1+x^{2s})^2}{(1-x^{2s})^2}.
\end{eqnarray*}
\end{proof}

\appendix

\section{Dual strata}
\label{dual strata}

We will describe the dual strata of nilpotent orbits. Let $G$ be a reductive algebraic group which for the purposes of this section can be assumed to be semisimple. We write $\Lg$ for its Lie algebra and $\cN$ for the nilpotent cone. We identify $\Lg^*$ with $\Lg$ via the Killing form. For each nilpotent $G$-orbit $\cO$ in $\cN$ we consider  its conormal bundle
\beqn
\Lambda_\cO = T^*_\cO\Lg =\{(x,y)\in \Lg\times\Lg\mid x\in\cO \ \ [x,y] = 0\}\,.
\eeqn
Consider the projection $\widetilde\cO $ of $\Lambda_\cO$ to the second coordinate:
\beqn
\widetilde\cO = \{y\in \Lg\mid \text{there exist an}\ \  x\in\cO\ \  \text{with} \ \ [x,y] = 0\}\,.
\eeqn
We construct an open (dense) subset $\widecheck \cO$ of $\widetilde\cO$ such that the projection $\Lambda_\cO \to \widetilde\cO$ has constant maximum rank above $\widecheck \cO$. Thus the $\widecheck \cO$ are subvarieties of $\Lg$ and they have the following property:
\begin{equation*}
\text{For any $\cF\in \op_G(\cN)$ the Fourier transform
$\fF(\cF)$ is smooth along all the  $\widecheck \cO$}\,.
\end{equation*}
This property follows from the fact that  the Fourier transform preserves the singular support.

 Before giving the general construction we make a few general comments. If $\cO=\{0\}$ is the zero orbit then $\widetilde\cO=\Lg$ and  $\widecheck \cO=\Lg^{rs}$, the set of regular semisimple elements in $\Lg$. The general description of  $\widecheck \cO$ will be similar.
 
Let  $\ft$ be a Cartan subspace of $\Lg$ and $W$ the Weyl group of $G$. Cconsider the adjoint quotient $\Lg \xrightarrow{f} \Lg\inv G\cong\Lt\slash W$. We have:
\beqn
\begin{CD}
\widetilde\cO @>{\tilde f}>> f(\widetilde\cO) @>{\sim}>>\Lt^\phi\slash W^\phi
\\
@VVV @VVV  @VVV
\\
\Lg @>{f}>> \Lg\inv G @>{\sim}>>\Lt\slash W
\end{CD}
\eeqn
where the vertical arrows are inclusions. 
We first analyze $f(\widetilde\cO)$ and explain the upper righthand corner. Let $\Lg^{ss}$ denote the set of semisimple elements in $\Lg$ and
\beqn
\widetilde\cO^{ss}\ = \  \{y\in \Lg^{ss}\mid \text{there exist an}\ \  x\in\cO\ \  \text{with} \ \ [x,y] = 0\}
\eeqn
denote the semisimple locus of $\widetilde\cO$.  Then we have $f(\widetilde\cO) = f(\widetilde\cO^{ss})$  because we can regard $\Lg\inv G\cong\Lt\slash W$ as consisting of the set of semisimple $G$-orbits in $\Lg$. 
Consider an element $e\in\cO$ and consider a corresponding $\fs\fl_2$-triple $\phi= (e,f,h)$. 
Then we have
\beqn
\Lg^e=\Lg^\phi\oplus\fu^e,
\eeqn
where
\beq\label{grading}
\Lg^\phi=\Lg^e\cap\Lg^h, \ \ \fu=\oplus_{i\geq 1}\Lg(i),\ \Lg(i)=\{z\in\Lg\,|\,[h,z]=iz\}.
\eeq
Recall that $\Lg^\phi$ is reductive and $\fu$ is nilpotent. In particular, the algebra $\Lg^\phi$ is the centralizer of the $\fs\fl_2$ given by $\phi= (e,f,h)$ and we write $G^\phi$ for the corresponding centralizer group.

We observe that any element in $\widetilde\cO^{ss}$ is $G$-conjugate to an element in $(\Lg^e)^{ss}$.  Thus, we see that:
\beqn
f(\widetilde\cO) = f(\widetilde\cO^{ss})=\widetilde\cO^{ss}/G=(\Lg^e)^{ss}/G^e\,.
\eeqn
Furthermore, any  element in $(\Lg^e)^{ss}$ is $G^e$-conjugate to an element in $(\Lg^\phi)^{ss}$. As $G^e = G^\phi \ltimes U^e$ is a semidirect product of a reductive group and a unipotent group it is easy to see that if two elements in $(\Lg^\phi)^{ss}$ are $G^e$-conjugate they are also $G^\phi$-conjugate. Thus we conclude that $(\Lg^e)^{ss}/G^e=(\Lg^\phi)^{ss}/G^\phi$. 
Finally, putting this all together, we get:
\beqn
f(\widetilde\cO) =(\Lg^e)^{ss}/G^e= (\Lg^\phi)^{ss}/G^\phi=\Lg^\phi\inv G^\phi\,.
\eeqn
Let $\Lt^\phi\subset\Lg^\phi$ be a maximal abelian subspace such that every semisimple element in $\Lg^\phi$ is $G^\phi$-conjugate to some element in $\Lt^\phi$. We choose $\Lt^\phi$ such that it lies in $\Lt$ and adjust $\ft$ so that  $h\in\ft$. Thus we have $\ft^\phi=\ft^e$. Writing $\Phi$ for the roots of $\Lg$ with respect to $\ft$ and writing $e=\sum_{\alpha\in\Phi}a_\alpha x_\alpha$, $x_\alpha\in\Lg_\alpha$, we have a concrete expression for $\ft^\phi$:
\beqn
\ft^\phi \ = \ \bigcap_{\alpha\in\Phi_e} \{t\in\ft\mid \alpha(t)=0\} \qquad \Phi_e=\{\alpha \in\Phi \mid a_\alpha \neq 0\}\ \ \text{with $e=\sum_{\alpha\in\Phi}a_\alpha x_\alpha$}\,,
\eeqn
i.e., $\ft^\phi$ is given by the intersection of the root hyperplanes for those roots that occur in the expression of $e$. 
Write 
\beqn
W^\phi=N_{G^\phi}(\Lt^\phi)/Z_{G^\phi}(\Lt^\phi). 
\eeqn
Thus, we conclude that 
\beqn
f(\widetilde\cO)=\Lg^\phi\inv G^\phi =\Lt^\phi/W^\phi.
\eeqn
Let us call the composition map $\tilde{f}:\widetilde{\cO}\to\Lt^\phi/W^\phi$. Note that the equality
$\widetilde\cO^{ss}/G= (\Lg^\phi)^{ss}/G^\phi$, which follows from the arguments above, shows that
\beq
\label{weyl equality}
W^\phi= N_{G^\phi}(\Lt^\phi)/Z_{G^\phi}(\Lt^\phi) = N_{G}(\Lt^\phi)/Z_{G}(\Lt^\phi) = N_{G}(L)/L \,,
\eeq
where $L=Z_{G}(\Lt^\phi)$. Note that $W^\phi$ is not necessarily a Coxeter group. 

In what follows we write $x=x_s+x_n$ for the decomposition of an element in its semisimple and nilpotent parts and then $[x_s,x_n]=0$. We now define 
\beqn
\widecheck\cO = \{y\in \widetilde\cO \mid y=y_s+y_n, \ \  \tilde{f}(y)\in (\Lt^\phi)^{rs}/W^\phi, \ \ y_n\in\cO\}\,;
\eeqn
here the $(\Lt^\phi)^{rs}$ stands for  elements in $\Lt^\phi$ that are regular semisimple in the ambient algebra $\Lg^\phi$. By construction, the group $G$ acts on $\widecheck\cO$. 
To show that $\widecheck\cO$  has the desired property we first claim:
\begin{lem} 
 If $x=x_s+x_n\in\Lg^e$ and $x_s\in(\Lt^\phi)^{rs}$, then $x_n\in\bar\cO$.
\end{lem}
\begin{proof}
First, recall that $e$ is distinguished if all  elements in $\Lg^e$ are nilpotent. By \cite[Theorem 1]{P}, if $e$ is distinguished and $x\in\Lg^e$, then $x\in\bar\cO$.

Let now $e$ be arbitrary. We have $x=x_s+x_n\in\Lg^e$ and $x_s\in(\Lt^\phi)^{rs}$. Then also $x_n\in\Lg^e$. But, now $G^{x_s}= Z_G(\ft^\phi)$ because $x_s$ is semisimple. As $[x_s,x_n]=0$,  $x_n\in Z_\Lg(\Lt^\phi)$. Consider the group $H=Z_G(\ft^\phi)/T^\phi$. Let us write $\bar e$  for $e$ regarded as a nilpotent element in $\fh=\operatorname{Lie}(H)$. The element $\bar e$ is a distinguished nilpotent in $\fh$. To see this it suffices to show that if $y\in Z_\Lg(\Lt^\phi)\cap\Lg^e$ is semisimple then $y\in\Lt^\phi$. Since  $\Lt^\phi$ is a maximal toral subalgebra of $\Lg^e$, $y\in\Lt^\phi$. But, now $G^{x_s}/T^\phi= H$  and $x_n$ can be viewed as an element in $\fh^{\bar e}$. As the $H$-orbit $\cO_{\bar e}$ of $\bar e$ is distinguished we see that $x_n$ lies in  $\overline{\cO_{\bar e}}$  and hence $x_n\in \bar\cO$.
\end{proof}

We now observe that $\tilde f^{-1}((\Lt^\phi)^{rs})$ is open dense in $\widetilde \cO$ and it remains to show that $\widecheck \cO$ is dense in  $\tilde f^{-1}((\Lt^\phi)^{rs})$. Let us decompose an element $y\in \tilde f^{-1}((\Lt^\phi)^{rs})$ into its semisimple and nilpotent parts $y=y_s+y_n$ and as $y\in \widetilde \cO$ there exists an element $x\in\cO$ such that $[x,y]=0$. Now, $G$ acts on $\tilde f^{-1}((\Lt^\phi)^{rs})$ and we can arrange the element $x$ to be $e$. Thus, up to $G$-action $y=y_s+y_n\in\Lg^e$. Thus, by the above lemma $y_n\in \bar\cO$. But, $\widecheck \cO$ consists of such elements with $y_n\in \cO$. Thus, $\widecheck \cO$ is open (and dense) in $\widetilde \cO$. It is now easy to check that  the projection $\Lambda_\cO \to \widetilde\cO$ has constant rank above  $\widecheck \cO$ and hence  $\widecheck \cO$ is a subvariety. 

This establishes a correspondence:
\beqn
\cO \leftrightarrow \widecheck\cO
\eeqn
Let us now analyze the equivariant fundamental group $\pi^G_1(\widecheck\cO)$. For an element $a\in (\Lt^\phi)^{rs}$ let $a'=a+e$ and $X_{a'}=Ga'$, the $G$-orbit of $a'$. We have the following exact sequence:
\beqn
1 \to \pi_1^G(X_{a'}) \to \pi_1^G(\widecheck\cO)\to B_{W^\phi}\to  1\,
\eeqn
where  $B_{W^\phi} = \pi_1((\Lt^\phi)^{rs}/W^\phi)$ denotes the braid group. We note that we use this terminology even when $W^\phi$ is not a Coxeter group. 
We have that 
\beqn
\pi_1^G(X_{a'}) \ = \ Z_G(a')/Z_G(a')^0\,.
\eeqn
Now, 
\beqn
Z_G(a')= Z_G(e+\Lt^\phi)=G^e\cap G^{\Lt^\phi}= (G^\phi\cdot U^e) \cap G^{\Lt^\phi}=Z_{G^\phi}({\Lt^\phi})\cdot (U^e\cap U^{\Lt^\phi})\,.
\eeqn
Thus:
\beqn
\pi_1^G(X_{a'}) \ = \ Z_{G^\phi}(\Lt^\phi)/Z_{G^\phi}(\Lt^\phi)^0\,.
\eeqn
We conclude that we have the following exact sequences:
\bern
&&1 \to Z_{G^\phi}(\Lt^\phi)/Z_{G^\phi}(\Lt^\phi)^0 \to \pi_1^G(\widecheck\cO)\to B_{W^\phi}\to  1\,\\
&&1 \to Z_{G^\phi}(\Lt^\phi)/Z_{G^\phi}(\Lt^\phi)^0 \to \pi_1^{G^\phi}((\Lg^\phi)^{rs})\to B_{W^\phi}\to  1\,.
\eern
The natural map $\ft^\phi \to \ft^\phi +e$ induces an embedding 
\beqn
(\Lg^\phi)^{rs} \subset \widecheck\cO \qquad gt\mapsto g(t+e) \qquad \text{for} \ \ \ t\in (\ft^\phi)^{rs} \ \ g\in G^\phi\,.
\eeqn
This embedding identifies $ \pi_1^G(\widecheck\cO)$ and $ \pi_1^{G^\phi}((\Lg^\phi)^{rs})$ and the two exact sequences above:
\beqn
\begin{CD}
1 @>>>  Z_{G^\phi}(\Lt^\phi)/Z_{G^\phi}(\Lt^\phi)^0@>>>\pi_1^{G^\phi}((\Lg^\phi)^{rs})@>{\tilde{q}}>> B_{W^\phi}@>>> 1
\\
@. @| @| @| @.
\\
1 @>>> Z_{G^\phi}(\Lt^\phi)/Z_{G^\phi}(\Lt^\phi)^0@>>>\pi_1^G( \widecheck\cO)@>{\tilde{q}}>> B_{W^\phi}@>>> 1\,.
\end{CD}
\eeqn

\section{The combinatorial formulas}\label{combinatorics}
\begin{center}
{\em by Dennis Stanton\footnote{School of Mathematics, University of Minnesota, USA. E-mail: \texttt{stant001@umn.edu}.}}
\end{center}
Let us write
$$
(A;q)_\infty=\prod_{s=0}^{\infty} (1-Aq^s),\ (A;q)_n=\frac{(A;q)_\infty}{(Aq^n;q)_\infty}= \prod_{i=0}^{n-1} (1-Aq^i),$$
$$ (A,B;q)_n= (A;q)_n (B;q)_n, \text{etc.}
$$
\subsection{Weighted type $(p,q)$ allowable signed Young diagrams}
Recall from \S\ref{ssec-proof BD} that $\on{SYD}_{p,q}$ denotes the set of signed Young diagrams that label the nilpotent $K$-orbits for the symmetric pair $(SO_N,S(O_p\times O_q))$.  For each $\lambda\in\on{SYD}_{p,q}$, we associate the weight $2^{r_\lambda+1}$ if $\lambda$ has at least one odd part and weight $1$ otherwise, where $r_\lambda$ is defined in~\eqref{component group-BD}. 
\begin{definition} Let $wt(p,q)$ denote the weighted sum of the type $(p,q)$ allowable signed Young diagrams, i.e., 
\beq\label{wt(p,q) def}
wt(p,q)=\sum_{\lambda\in\on{SYD}_{p,q}}2^{r_\lambda+1}-\pa(\frac{q}{2})\delta_{p,q},
\eeq
where $\pa(k)$ is the number of partitions of $k$ and we set $\pa(q/2)=0$ if $q$ is odd.
\end{definition}

\begin{prop}\label{prop-generating function-total}
 The generating function for $wt(p,q)$ is
$$
F(u,v)=\sum_{p,q=0}^\infty wt(p,q) u^p v^q=
\prod_{k=1}^\infty \frac{1}{1-u^{2k}v^{2k}}
\prod_{m=0}^\infty \frac{(1+u^{m+1}v^m)(1+u^mv^{m+1})}{(1-u^{m+1}v^m)(1-u^mv^{m+1})}.
$$
\end{prop}

\begin{proof} An odd part of $2m+1$ can occur, with multiplicity $0,1$, $2,\dots ,r,\dots$, which contributes
$$
\begin{aligned}
1+&2(u^{m+1}v^m+u^mv^{m+1})+(2u^{2(m+1)}v^{2m}+4u^{2m+1}v^{2m+1}+2u^{2m}v^{2(m+1)})+
\cdots \\
+&(2u^{r(m+1)}v^{rm}+4u^{rm+r-1}v^{rm+1}+\cdots+ 4u^{rm+1}v^{rm+r-1}+2u^{rm}v^{r(m+1)})+
\cdots\\
=& 1+2\sum_{r=1}^\infty (u^mv^m)^r\left( \frac{u^{r+1}-v^{r+1}}{u-v}+uv \frac{u^{r-1}-v^{r-1}}{u-v}
\right)
=
\frac{(1+u^{m+1}v^m)(1+u^mv^{m+1})}{(1-u^{m+1}v^m)(1-u^mv^{m+1})}.
\end{aligned}
$$
An even part $2k$, since it has even multiplicity, clearly gives
$$
\frac{1}{1-u^{2k}v^{2k}}.
$$
\end{proof}

Recall the $q$-Gauss theorem is
\beq\label{q-gauss}
\sum_{m=0}^\infty \frac{(a;q)_m (b;q)_m}{(q;q)_m (c;q)_m} \left( \frac{c}{ab}\right)^m=
\frac{(c/a;q)_\infty (c/b;q)_\infty}{(c;q)_\infty (c/ab;q)_\infty}.
\eeq

\begin{corollary} \label{prop-total number-pq}
Let $k\geq 0$. The generating function of the $k^{th}$ diagonal is
$$
F_{kdiag}(t)=\sum_{p=0}^\infty wt(p+k,p) t^{2p} =
\frac{2}{1+t^{2k}}
\frac{\prod_{m=1}^\infty (1+t^{2m})}
{\prod_{m=1}^\infty (1-t^{2m})^3}.
$$
\end{corollary}
\begin{proof} 
We use the $q$-binomial theorem with $q=uv$ for $F(u,v)=\sum_{p,q=0}^\infty wt(p,q) u^p v^q$:
$$
\begin{aligned}
F(u,v)=& \frac{1}{(u^2v^2;u^2v^2)_\infty} \times 
\frac{(-u;uv)_\infty}{(u;uv)_\infty}\times 
\frac{(-v;uv)_\infty}{(v;uv)_\infty}\\
=&\frac{1}{(q^2;q^2)_\infty}\sum_{m=0}^\infty \frac{(-1;q)_m}{(q;q)_m} u^m
\sum_{m=0}^\infty \frac{(-1;q)_m}{(q;q)_m} v^m.
\end{aligned}
$$
So the $k$-th diagonal term is
$$
\frac{1}{(q^2;q^2)_\infty}\sum_{m=0}^\infty \frac{(-1;q)_{m+k}(-1;q)_m}{(q;q)_{m+k} (q;q)_m} q^m.$$
This sums by the $q$-Gauss theorem to
$$
\begin{aligned}
&=\frac{1}{(q^2;q^2)_\infty}
\frac{(-1;q)_k}{(q;q)_k}\sum_{m=0}^\infty \frac{(-q^k;q)_{m}(-1;q)_m}{(q^{k+1};q)_{m} (q;q)_m} q^m\\
&=\frac{1}{(q^2;q^2)_\infty}
\frac{(-1;q)_k}{(q;q)_k}
\frac{(-q;q)_\infty (-q^{k+1},q)_\infty}{(q^{k+1};q)_\infty (q;q)_\infty}=\frac{2}{1+q^k}
\frac{(-q;q)_\infty}{(q;q)_\infty^3},
\end{aligned}
$$
which is the given answer with $q=t^2.$
\end{proof}

\subsection{Proof of Proposition~\ref{prop-counting biorbital}}\label{proof of biorbital counting}
Recall that $\cP^{odd}(N)$ denotes the set of partitions of $N$ into odd parts. We write a partition $\lambda\in\cP^{odd}(N)$ as
\beqn
\lambda=(2\mu_1+1)+(2\mu_2+1)+\cdots+(2\mu_{s}+1)
\eeqn 
where $\mu_1\geq\mu_2\geq\cdots\geq\mu_{s}\geq 0$. Note that $s\equiv N\mod 2$. We set
\beqn
wt_\lambda=3^{\#\{1\leq j\leq (s-1)/2\,|\,\mu_{2j-1}=\mu_{2j}+1\}}\,4^{\#\{1\leq j\leq  (s-1)/2\,|\,\mu_{2j-1}\geq\mu_{2j}+2\}}\text{ when $N$ is odd};
\eeqn
\beqn
wt_\lambda=3^{\#\{1\leq j\leq s/2-1\,|\,\mu_{2j}=\mu_{2j+1}+1\}}4^{\#\{1\leq j\leq s/2-1\,|\,\mu_{2j}\geq\mu_{2j+1}+2\}}\text{ when $N$ is even}.
\eeqn
Let us write 
\beqn
b_n=\sum_{\lambda\in\cP^{odd}(2n+1)}wt_\lambda,\ \ c_n=\sum_{\lambda\in\cP^{odd}(2n)}wt_\lambda.
\eeqn

\subsubsection{Equation~\eqref{number of biorbital-B}}\label{counting-biorbital-B}The equation~\eqref{number of biorbital-B} is equivalent to
\beqn
\sum_{n=0}^\infty b_n q^{2n+1}=q (-q^4;q^4)_\infty^2 (-q^2;q^2)_\infty^2.
\eeqn
This can be seen as follows. Suppose $\lambda$ with odd parts has an odd number of parts, say $2k+1$. Consider the columns of $\lambda$, which have possible sizes $1,2,\cdots, 2k+1.$

The part $2k+1$ occurs an odd number of times, the generating function is
$$
\frac{q^{2k+1}}{1-q^{4k+2}}.
$$
The part $2k$ occurs an even number of times, the generating function is
$$
\frac{1}{1-q^{4k}}.
$$
The part $2k-1$ occurs an even number of times, the weighted generating function is
$$
1+3q^{4k-2}+\frac{4q^{8k-4}}{1-q^{4k-2}}.
$$
This continues down to part size 1, to obtain the generating function
$$
\sum_{n=0}^\infty b_n q^{2n+1} =
\sum_{k=0}^\infty \frac{q^{2k+1}}{1-q^{4k+2}}
\frac{1}{(1-q^4)(1-q^8)\cdots (1-q^{4k})}
\prod_{i=1,{\text{odd}}}^{2k-1} \left(1+3q^{2i}+\frac{4q^{4i}}{1-q^{2i}}\right).
$$
Because
$$
1+3x+\frac{4x^2}{1-x}=\frac{(1+x)^2}{1-x}
$$
this may be written as
\beqn
\sum_{n=0}^\infty b_n q^{2n+1} =
\sum_{k=0}^\infty \frac{q^{2k+1}}{1-q^{4k+2}}
\frac{(-q^2;q^4)_k (-q^2;q^4)_k}{(q^4;q^4)_k (q^2;q^4)_k}=\frac{q}{1-q^2}\sum_{k=0}^{\infty}\frac{(-q^2;q^4)_k(-q^2;q^4)_k}{(q^6;q^4)_k(q^4;q^4)_k}q^{2k}.
\eeqn
Applying $q$-Gauss theorem~\eqref{q-gauss} with $q\rightarrow q^4$ and $a=b=-q^2,c=q^6$, we obtain
\beqn
\sum_{n=0}^\infty b_n q^{2n+1} =\frac{q}{1-q^2} \frac{(-q^4;q^4)_\infty (-q^4;q^4)_\infty}{(q^6;q^4)_\infty (q^2;q^4)_\infty}=q (-q^4;q^4)_\infty^2 (-q^2;q^2)_\infty^2.
\eeqn
 Here we have used the number of partitions of $n$ into odd parts equals the number of partitions of $n$ into distinct parts, i.e.,
\beqn
(-q;q)_\infty=1/(q;q^2)_\infty.
\eeqn

\subsubsection{Equation~\eqref{number of biorbital-D}}The equation~\eqref{number of biorbital-D} is equivalent to
\beqn
\sum_{n=0}^\infty c_n q^{2n}=\frac{1}{4}(-q^2;q^4)_\infty^2 (-q^2;q^2)_\infty.
\eeqn
Assume $\lambda$ is a partition into odd parts with $2k$ parts. We argue as 
in \S\ref{counting-biorbital-B}, this time the even parts have weights.
Thus
\begin{eqnarray*}
\sum_{n=0}^\infty c_n q^{2n} &=&
\sum_{k=0}^\infty \frac{q^{2k}}{1-q^{4k}}
\frac{1}{(1-q^2)(1-q^6)\cdots (1-q^{2(2k-1)})}
\prod_{i=1,{\text{even}}}^{2k-2} \left(1+3q^{2i}+\frac{4q^{4i}}{1-q^{2i}}\right).
\\
& =&
\sum_{k=0}^\infty \frac{q^{2k}}{1-q^{4k}}
\frac{(-q^4;q^4)_{k-1} (-q^4;q^4)_{k-1}}{(q^4;q^4)_{k-1} (q^2;q^4)_k}=
\frac{1}{4}\sum_{k=0}^\infty\frac{(-1;q^4)_k(-1;q^4)_k}{(q^2;q^4)_k(q^4;q^4)_k}q^{2k}
\\
&=&\frac{1}{4} \frac{(-q^2;q^4)_\infty (-q^2;q^4)_\infty}
{(q^2;q^4)_\infty (q^2;q^4)_\infty}= \frac{1}{4}(-q^2;q^4)_\infty^2 (-q^2;q^2)_\infty.
\end{eqnarray*}

\subsection{Proof of equations~\eqref{Fodd} and~\eqref{Feven}}\label{functions}

\begin{definition} \texorpdfstring{}{Lg}
Let
$$
F(q)=\frac{(-q;q^2)_\infty^2}{(q;q^2)_\infty^2}.
$$
\end{definition}

The equations~\eqref{Fodd} and~\eqref{Feven} are
\begin{prop}
$$
F(q)-F(-q)= 8q (-q^4;q^4)_\infty^4 (-q^2;q^2)_\infty^4\ \text{ and }\ F(q)+F(-q)=2(-q^2;q^4)_\infty^4 (-q^2;q^2)_\infty^4.
$$
\end{prop} 

Recall Ramanujan's $_1\psi_1$ formula
\begin{prop}
$$
_1\psi_1(a,b;q;x)=:\sum_{n=-\infty}^\infty \frac{(a;q)_n}{(b;q)_n} x^n=
\frac{(ax,q/ax,q,b/a;q)_\infty}{(x,b/ax,b,q/a;q)_\infty}.
$$
for $|b/a|<|x|<1.$
\end{prop}

Letting $q\rightarrow q^2,$ $a=-1, b=-q^2, x=q$ in Ramanujan's $_1\psi_1$ gives
$$
\sum_{k=-\infty}^\infty \frac{q^k}{1+q^{2k}}=\frac{1}{2}\  _1\psi_1(-1,-q^2;q^2;q)=
\frac{1}{2}\frac{(-q,-q,q^2,q^2;q^2)_\infty}{(q,q,-q^2,-q^2;q^2)_\infty}=
\frac{1}{2}F(q) \frac{(q^2;q^2)_\infty^2}{(-q^2;q^2)_\infty^2}.
$$

We take the odd terms in the sum
$$
\sum_{k=-\infty, odd}^\infty \frac{q^k}{1+q^{2k}}=
\frac{1}{4}\left( F(q)-F(-q)\right) \frac{(q^2;q^2)_\infty^2}{(-q^2;q^2)_\infty^2}
$$
and the even terms in the sum
$$
\sum_{k=-\infty, even}^\infty \frac{q^k}{1+q^{2k}}=
\frac{1}{4}\left( F(q)+F(-q)\right) \frac{(q^2;q^2)_\infty^2}{(-q^2;q^2)_\infty^2}.
$$
However
\begin{eqnarray*}
\sum_{k=-\infty, odd}^\infty \frac{q^k}{1+q^{2k}}&=&
q\frac{1}{1+q^2}\  _1\psi_1(-q^2,-q^6;q^4;q^2)\\
&=&
q\frac{1}{1+q^2} \frac{(-q^4,-1,q^4,q^4;q^4)_\infty}{(q^2,q^2,-q^6,-q^2;q^4)_\infty}
=
2q \frac{(-q^4,-q^4,q^4,q^4;q^4)_\infty}{(q^2,q^2,-q^2,-q^2;q^4)_\infty},
\\
\sum_{k=-\infty, even}^\infty \frac{q^k}{1+q^{2k}}&=&
\frac{1}{2}\  _1\psi_1(-1,-q^4;q^4;q^2)=
\frac{1}{2} \frac{(-q^2,-q^2,q^4,q^4;q^4)_\infty}{(q^2,q^2,-q^4,-q^4;q^4)_\infty}.
\end{eqnarray*}
So 
\begin{eqnarray*}
\frac{1}{4}\left( F(q)-F(-q)\right)&=&
2q \frac{(-q^4,-q^4,q^4,q^4;q^4)_\infty}{(q^2,q^2,-q^2,-q^2;q^4)_\infty}
\frac{(-q^2;q^2)_\infty^2}{(q^2;q^2)_\infty^2}= 2q  \frac{(-q^4;q^4)_\infty^2}{(q^2;q^4)_\infty^2} \frac{(-q^4;q^4)_\infty^2}{(q^2;q^4)_\infty^2}\\&= &2q \frac{(-q^4;q^4)_\infty^4}{(q^2;q^4)_\infty^4}
= 2q (-q^4;q^4)_\infty^4 (-q^2;q^2)_\infty^4\\
\frac{1}{4}\left( F(q)+F(-q)\right)&=&\frac{1}{2} \frac{(-q^2,-q^2,q^4,q^4;q^4)_\infty}{(q^2,q^2,-q^4,-q^4;q^4)_\infty}\frac{(-q^2;q^2)_\infty^2}{(q^2;q^2)_\infty^2}
=\frac{1}{2}(-q^2;q^4)_\infty^4 (-q^2;q^2)_\infty^4.
\end{eqnarray*}

\end{document}